\documentclass{article}
\usepackage{eurosym}
\usepackage[T1]{fontenc}
\usepackage{epsf,epsfig,subfigure}
\usepackage{float,color,ulem}
\usepackage{makeidx}
\usepackage{amssymb}
\usepackage{amsfonts}
\usepackage{amsmath}
\usepackage[english]{babel}
\usepackage{graphicx,multirow}
\usepackage{indentfirst}
\usepackage[colorlinks=true]{hyperref}
\hypersetup{urlcolor=blue,linkcolor=blue ,citecolor=blue,colorlinks=true}
\usepackage{bbm}
\usepackage{geometry}
\makeatletter\@addtoreset{equation}{section}
\newfloat{figure}{H}{lof}
\newfloat{table}{H}{lot}
\floatname{figure}{\figurename}
\floatname{table}{\tablename}
\newtheorem{theorem}{Theorem}[section]

\newtheorem{assumption}[theorem]{Assumption}
\newtheorem{corollary}[theorem]{Corollary}
\newtheorem{definition}[theorem]{Definition}

\newtheorem{lemma}[theorem]{Lemma}

\newtheorem{remark}[theorem]{Remark}
\newenvironment{proof}[1][Proof]{\noindent\textit{#1.} }{\hfill \rule{0.5em}{0.5em}\\\par}
\numberwithin{equation}{section}
\newcommand{\R}{\mathbb{R}}
\newcommand{\T}{\mathbb{T}}
\newcommand{\Z}{\mathbb{Z}}
\newcommand{\di }{{\,\mathrm{d}}}

\begin{document}
	
	\title{\textbf{Asymptotic behavior of a nonlocal advection system with two populations}}
	\author{\textsc{Xiaoming Fu\thanks{The research of this author is supported by China Scholarship Council.} and Pierre Magal}
	\\
	%EndAName
	{\small \textit{Univ. Bordeaux, IMB, UMR 5251, F-33400 Talence, France}} \\
	{\small \textit{CNRS, IMB, UMR 5251, F-33400 Talence, France.}}}
	\maketitle
	\begin{abstract}
		In this paper, we consider a nonlocal advection model for two populations on a bounded domain. The first part of the paper is devoted to the existence and uniqueness of solutions and the associated semi-flow properties. Here we use the notion of solution integrated along the characteristics. Next, by proving segregation property, we construct an energy functional to investigate the asymptotic behavior of the solution. In order to get some compactness of the positive orbit, we use the narrow convergence in the space of Young measures. By using this idea, we get a description of the asymptotic behavior of the solution in the space of Young measures. The last section of the paper is devoted to numerical simulations, which confirm and complement our theoretical results.  
	\end{abstract}
\section{Introduction}
In this work,  we study a two species population dynamic model with nonlocal advection term
\begin{equation}		\label{1.1}
		\begin{cases}
				\partial _{t}u_1(t,x) +\mathrm{div }\bigl(u_1(t,x)\mathbf{v}(t,x)\bigr)=
				u_1(t,x)h_1(u_1(t,x),u_2(t,x)),\\
				\partial _{t}u_2(t,x) +\mathrm{div }\bigl(u_2(t,x)\mathbf{v}(t,x)\bigr)=
				u_2(t,x)h_2(u_1(t,x),u_2(t,x)),
		\end{cases}\;\;t>0,\;x\in \mathbb{R}^{N},
\end{equation}
and the velocity field ${\bf v}=-\nabla P $ is derived from pressure 
\begin{equation*}
	P(t,x):=\left(\rho\ast (u_1+u_2)(t,.)\right)(x),
\end{equation*}
where $ \ast $ is the convolution in $ \R^N $. Suppose equations \eqref{1.1} are supplemented with a periodic initial distribution
\begin{equation}\label{1.2}
\textbf{u}_0(x):=\begin{pmatrix}
u_1(0,x)\\
u_2(0,x)
\end{pmatrix}\in \R_+^2 \text{ where $\mathbf u_0$ is a $2\pi$--periodic
	function in each direction.}  
\end{equation}
In this article we consider solutions of equations \eqref{1.1} which are \textbf{periodic in space}. Here a function $u(x)$ is said to be \textbf{$2\pi$-periodic in each direction} (or for simplicity \textbf{periodic}) if  
$$
u(x+2k\pi)=u(x),\;\forall k\in \mathbb{Z}^N,x\in \R^N.
$$
When $u(x)$ is periodic we can reduce the convolution to the $N$--dimensional torus $\mathbb{T}^{N}:=\R^N / 2\pi \mathbb{Z}^N $  by making the following observations 
\begin{equation*}
\begin{aligned}
\left(\rho\ast u \right)(x) &=\int_{\R^N}\rho(x-y)u(y) \di  y\\
%&= \sum_{k\in \Z^N}\int_{[2k\pi,2\pi+2k\pi]} \rho(x-y) u(y)\di y\\
&= \sum_{k\in \Z^N}\int_{[0,2\pi]^N} \rho(x-(y+2 k \pi )) u(y+2 k \pi)\di  y \\
&=\sum_{k\in \Z^N}\int_{[0,2\pi]^{N}} \rho(x-y-2k\pi) u(y)\di  y,
\end{aligned}
\end{equation*}
hence we can reformulate as
\begin{equation*}
\left(\rho\ast u(\cdot) \right)(x) =\frac{1}{(2\pi)^N}\int_{{[0,2\pi]}^{N}} K(x-y)u(y)\di  y,
\end{equation*}
where $K$ is again $ 2\pi $--periodic in each direction and defined by 
\begin{equation*}
K(x)=(2\pi)^N \sum_{k\in \mathbb Z^N}\rho(x+2\pi k),\;x\in\R^N.
\end{equation*} 
The fast decay of $ \rho  $ is necessary to ensure the convergence of the above series (see Remark \ref{REM1.3} for details).
Now we can rewrite the velocity field $ \textbf{v} $ as follows:
\begin{equation}\label{1.3}
	\mathbf{v}(t,x)=-\nabla \left[ K\circ (u_1+u_2)(t,\cdot)\right] (x),
\end{equation}%
where $\circ $ denotes the convolution operator on the $N$--dimensional torus $
\mathbb{T}^{N}:=\R^N / 2\pi \mathbb{Z}^N \simeq [0, 2 \pi]^N$ defined for each $2 \pi$-periodic in each direction and measurable functions $\varphi$ and $\psi$ by 
\begin{equation*}
	\left( \varphi \circ \psi \right) (x)=|\mathbb{T}^{N}|^{-1}\int_{\mathbb{T}%
		^{N}}\varphi(x-y)\psi (y)\di  y.
\end{equation*}

	Our motivation for this problem comes from the observations in the biological experiments for two types of cells co-cultured on the monolayer. One can find an example of such a co-culture in \cite[Figure 1]{Pasquier2012}. Cells are growing and meanwhile forming segregated islets and the growth stops when they become locally saturated. 
	
	In this article, we intend to study this mechanism by using a non-local advection equation with contact inhibition. As we will see, our model captures the finite propagating speed in cell co-culturing. In the context of cell population, the impact of cell adhesion and repulsion on the movement and patterning of cell populations has been studied by many authors, for example, Armstrong et al. \cite{Armstrong2006} and Painter et al. \cite{Painter2015}. For a more general perspective, our study is connected to the cell segregation and border formation,  Taylor et al. \cite{Taylor2017} concluded the heterotypic repulsion and homotypic cohesion can account for cell segregation and border formation. We also refer the readers to Dahmann et al. \cite{Dahmann2011} and the references therein for more about boundary formation with its application in biology.  These observations and results in biological experiments lead us to adopt a nonlocal advection term which is able to explain the finite propagation speed of cells and cell segregation. The segregation property was brought up in the 80's by using cross diffusion by Shigesada, Kawasaki and Teramoto \cite{Shigesada1979} and Mimura and Kawasaki \cite{Mimura1980}. Since then, the cross-diffusion has been widely studied and we refer to Lou and Ni \cite{Lou1996, Lou1999} for more results about this subject. We also refer to the introduction of Bertozzi, Garnett and Laurent \cite{Bertozzi2012} for a list of applications for different fields. 
	
	Further studies on mathematical analysis of such a non-local advection equation with linear and nonlinear diffusion have been investigated by \cite{Bedrossian2011,Bertozzi2011,Bertozzi2010,Dyson2010}. This class of systems has been recently studied in \cite{Bertozzi2012,Bodnar2006} and \cite{Raoul2012} with an additional heterogeneous transport term. The traveling wave solution of such nonlocal equation with diffusive perturbation were considered by many authors, we refer the readers to \cite{Bernoff2011, Hamel2017, Leverentz2009, Mogilner1999, Nadin2008} for swarming models and propagation of front. Let us also mention that system \eqref{1.1} is also related to hyperbolic Keller-Segel equation (see Perthame and Dalibard  \cite{Perthame2009}).

	The single species model of equation \eqref{1.1} has been studied by Ducrot and Magal in \cite{Ducrot2014} (see the derivation of the model therein). Compared to the work in \cite{Ducrot2014}, one of the technical difficulties in this paper is that we do not have a $L^2$ uniform boundedness of the solution a priori. This is because we allow function $ h $ to be of more general type than that in \cite{Ducrot2014} (see Assumption \ref{ASS1.1} and \ref{ASS4.1}). This difficulty obliges us to find another way to prove the $ L^{\infty} $ uniform boundedness of the solution (see Lemma \ref{LEM4.9}, Remark \ref{REM4.10} and Theorem \ref{THM4.11}). Moreover,  we prove the segregation property for two species by using the notion of solution integrated along the characteristics.  Our key Assumption \ref{ASS4.4} on the positivity of Fourier coefficients enables to construct a decreasing energy functional, this condition has also been considered in \cite{Bernoff2011} and \cite{Ducrot2011}.   Using this important property, we show the $ L^{\infty} $ convergence of the sum of two species when the initial distribution is strictly positive (see Corollary \ref{COR4.12}). Furthermore, the segregation property 
is preserved asymptotically when $ t $ tends to infinity (see Lemma \ref{LEM5.14}).
	In Section 7, by using numerical simulations, we obtain some results which have not been proved theoretically. In fact,  we show the necessity of using a weaker sense of convergence (narrow convergence) to encompass the lack of compactness of the solution and to study the limit for each species.
\begin{assumption}
	\label{ASS1.1} For $ i=1,2, $ $ h_i: \R_+^2 \to \R$ are $ C^1 $ functions with $$ \sup_{u_1,u_2\geq 0} h_i(u_1,u_2) <\infty,\quad \sup_{u_1,u_2\geq 0}\partial_{u_j}h_i(u_1,u_2)<\infty,\quad j=1,2.$$
\end{assumption}
An example of function $h_i$ is the following function 
$$
h_i(u_1,u_2)=\lambda_i  (1-(u_1+u_2)). 
$$
Therefore $u_i \times h_i(u_1,u_2)$ is a logistic function. As assumed in \cite{Ducrot2014}, the map $ h_i $ is not bounded from below. Here we not will make such an assumption.  

Motivated by the model derived from Ducrot et al. in \cite{Ducrot2011}  which describes the contact inhibition (i.e. cells stop growing when they are locally saturated), we would like to use the following non-linear function. 
$$
h_i(u_1,u_2)= \frac{b_i }{1+ \gamma_i \left(u_1+u_2\right)}-\mu_i
$$
where $b_i>0$ is the division rate, $\mu_i>0$ is the mortality rate and $\gamma_i>0$ is a coefficient.  

In such a case the map $h_i$ is bounded from below therefore we can not apply the same arguments than \cite{Ducrot2014} to obtain an $L^\infty$ bounded for the solution. To encompass this difficult, we resort to another approach. This shows that our results can be applied to a larger class of non-linearity than \cite{Ducrot2014}.

\begin{assumption}
	\label{ASS1.2} The kernel $K:\mathbb{R}^N\rightarrow \mathbb{R}$ is a $\mathbb{T}^N-$ periodic function of the class $C^m$ on $\mathbb{R}^N$ for some integer $ m\geq \frac{N+5}{2} $.  

\end{assumption}
\begin{remark}\label{REM1.3} The above regularity Assumption \ref{ASS1.2} can be reduced to $ m\geq 3 $ in proving the existence and uniqueness of solutions. The higher regularity is crucial for Lemma \ref{LEM4.9}.  For the dimension $ N\leq 3 $, the regularity condition in Assumption \ref{ASS1.2} is always satisfied whenever $ K\in C^4 $. For the choice of $ \rho $ in \eqref{1.1}, it suffice to choose $ \rho \in W^{m+1,1}(\R^N) $. For each multi-index $ \alpha $ with $ |\alpha|\leq m$, the series
	\[ x\longmapsto \sum_{k\in \Z^N} D^\alpha \rho(x+2 \pi k) \]
	is uniformly converging on $ \T^N $. Hence, $ K $ satisfies Assumption \ref{ASS1.2}.
\end{remark}

The plan of the paper is the following. In Section 2, we investigate the existence and the uniqueness of solution integrated along the characteristics. In Section 3, we study the segregation property. In Sections 4 and 5, the asymptotic behavior of segregated solutions has been studied by using Young measures (a generalization of $ L^\infty $ weak $ * $--convergence).  Section 6 is devoted to numerical simulations where we explore some further results that are not proved analytically, these numerical simulations complement our theoretical part.   

\section{Solution Integrated along the Characteristics}
In this section we study the existence and uniqueness of
solution for \eqref{1.1}-\eqref{1.3} with
initial data $\mathbf u_0\in L^\infty_{per}\left(\mathbb{R}^N\right)^2$.

Before going further let us introduce some notations that will be used in the following.
For each $k\in \mathbb{N}$, let us denote by $C_{per}^k\left(\mathbb{R}%
^N\right)$ the Banach space of functions of the class $C^k$ from $\mathbb{R}%
^N$ into $\mathbb{R}$ and $[0,2\pi]^N $--periodic endowed with the usual
supremum norm
\begin{equation*}
\left\| \varphi \right\Vert _{C^k }=\sum_{p=0}^k\; \sup_{x\in 
	\mathbb{R}^N}\left\vert D^p\varphi (x)\right\vert.
\end{equation*}
For each $p\in \left[ 1,+\infty \right] $, let us denote by $%
L_{per}^{p}\left( \mathbb{R}^N\right)$ the space of measurable and $%
[0,2\pi]^N-$periodic functions from $\mathbb{R}^N$ to $\mathbb{R}$ such that 
\begin{equation*}
\left\| \varphi \right\Vert_{L^p}:=\left\| \varphi \right\Vert _{L^{p}\left( \left(0 , 2 \pi\right)^N
\right) }<+\infty .
\end{equation*}%
Then $L_{per}^{p}\left( \mathbb{%
	R}^N\right) $ endowed with the norm $%
\left\| \varphi \right\Vert_{L^p}$
is a Banach space. We also introduce its positive cone $L^p_{per,+}\left(
\mathbb{R}^N\right)$ consisting of function in $L_{per}^{p}\left( \mathbb{%
	R}^N\right) $ almost everywhere positive.
\begin{remark}\label{REM1.4}
	When we study the product space $ C_{per}^k\left(\mathbb{R}^N\right)^n, L_{per}^{p}\left( \mathbb{R}^N\right)^n $ when $ n\in \mathbb N $, for simplicity, we use the same notation $ \| \cdot\|_{C^k} $ and $ \| \cdot\|_{L^p} $ for the norm in product space.
\end{remark}

 We first investigate the characteristic curves of the problem. 

\begin{lemma}\label{LEM2.1}
	Let Assumption \ref{ASS1.2} be satisfied. Let $u_i\in
	C\left( \left[ 0,\tau \right],L_{per}^{1}\left( \mathbb{R}^N\right)
	\right),i=1,2 $ be given. Then by setting $\mathbf{v}(t,x)=-\nabla\left[ K\circ (u_1+u_2)(t,\cdot)\right]
	(x) $, the following non-autonomous system for each $s\in \left[ 0,\tau \right]$ and each $ z\in\mathbb{R}^N$: 
	\begin{equation*}
		\left\{ 
		\begin{array}{l}
			\partial _{t}\Pi _{\mathbf{v}}(t,s;z)=\mathbf{v}(t,\Pi _{\mathbf{v}%
			}(t,s;z)),\text{ for each }t\in \left[ 0,\tau \right] , \\ 
			\Pi _{\mathbf{v}}(s,s;z)=z,%
		\end{array}%
		\right.  
	\end{equation*}%
	generates a unique non-autonomous continuous flow $\left\{ \Pi _{\mathbf{v}%
	}(t,s)\right\} _{t,s\in \left[ 0,\tau \right] }$, that is to say,
	\begin{equation*}
		\Pi _{\mathbf{v}}(t,r;\Pi _{\mathbf{v}}(r,s;z))=\Pi _{\mathbf{v}%
		}(t,s;z),\forall t,s,r\in \left[ 0,\tau \right] ,\text{ and }\Pi _{\mathbf{v}%
	}(s,s;.)=I
\end{equation*}%
and the map $(t,s,z)\rightarrow \Pi _{\mathbf{v}}(t,s;z)$ is continuous.
Moreover for each $t,s\in \left[ 0,\tau \right] ,$ we have \ 
\begin{equation*}
	\Pi_{\mathbf{v}} (t,s;z+2\pi k )=\Pi_{\mathbf{v}} (t,s;z)+2\pi k ,\forall
	z\in \mathbb{R^N},\,k\in\mathbb{Z}^N,
\end{equation*}%
the map $z\rightarrow \Pi_{\mathbf{v}} (t,s;z)$ is continuously
differentiable and one has the determinant of Jacobi matrix: 
\begin{equation}\label{2.1}
	\det(\partial_{z}\Pi_{\mathbf{v}} (t,s;z))=\exp \left( \int_{s}^{t} \mathrm{%
		div }\; \mathbf{v}(l,\Pi _{\mathbf{v}}(l,s;z)) dl\right).
\end{equation}
\end{lemma}

\begin{proof}
	By the assumption, we have $$ \mathbf{v}(t,x)\in  C\left( \left[ 0,\tau \right], C_{per}^{1}\left( \R^N \right)^N \right),$$ and we have the following estimations
	\begin{equation*}
	\begin{split}
	&\left\| \mathbf{v}(t,\cdot)\right\Vert _{C^0 }\leq \left\| \nabla
	K\right\Vert _{C^0 }\left\| (u_1+u_2)(t,\cdot)\right\|_{L^1}, \\
	&\left\| \mathrm{div}\,\mathbf{v}(t,\cdot)\right\Vert _{C^0 }\leq
	\left\| \Delta K\right\Vert _{C^0 }\left\| (u_1+u_2)(t,\cdot)\right\|_{L^1}.
	\end{split}%
	\end{equation*}	
	Therefore, the first part of the results follows by using classical arguments on ordinary differential
	equations.
	For the proof of \eqref{2.1},
	note that 
		\begin{equation*}
			\left\{ 
			\begin{array}{l}
				\partial_{t}\partial_z \Pi _{\mathbf{v}}(t,s;z)=\partial_x\mathbf{v}(t,\Pi _{\mathbf{v}}(t,s;z)) \partial_z\Pi _{\mathbf{v}%
			}(t,s;z),\,t\in \left[ 0,\tau \right] , \\ 
				\partial_z\Pi _{\mathbf{v}}(s,s;z)=I.
			\end{array}
			\right. 
		\end{equation*}%
	For any  matrix-valued $ C^1 $ function $ A:t\mapsto A(t) $, the Jacobi's formula reads as follows 
	\begin{equation*}
		\frac{\di  }{\di  t} \det A(t) = \det A(t) \text{tr}\left(A^{-1}(t) \frac{d}{d t}A(t)\right),
	\end{equation*}
	hence we obtain
		\begin{equation*}
		\frac{\di  }{\di  t} \det \partial_z \Pi _{\mathbf{v}}(t,s;z)  = \det \partial_z \Pi _{\mathbf{v}}(t,s;z) \times \text{tr}\left( \partial_x\mathbf{v}(t,\Pi _{\mathbf{v}%
		}(t,s;z))\right)
		\end{equation*}
	and since $ \text{tr}\left( \partial_x\mathbf{v}(t,\Pi _{\mathbf{v}%
	}(t,s;z))\right)= \mathrm{%
		div }\; \mathbf{v}(t,\Pi _{\mathbf{v}}(t,s;z)) $ therefore the result follows.
\end{proof}

\bigskip

In order to precise the notion of solution in this work, assume
first that $$\textbf{u}=\begin{pmatrix}
u_1,u_2
\end{pmatrix}\in C^{1}\left( \left[ 0,\tau \right] \times \R^N,\R\right)^2 \cap C\left( [0,\tau ],C_{per ,+}^{0}(\mathbb{R}^{N})\right)^2 $$
is a classical solution of \eqref{1.1}-\eqref{1.3}. 
We consider the solution with each component $ u_i(t,\cdot) $ along the characteristic curve $ \Pi_{\textbf{v}}(t,0;x) $ respectively, we obtain for $ i=1,2, $
\begin{equation*}
	\begin{aligned}
		{\displaystyle\frac{\di }{\di  t}}\Bigl(u_i(t,\Pi _{\mathbf{v}}(t,0;z)\Bigr)=&\partial _{t}u_i(t,\Pi _{\mathbf{v}}(t,0;z))+\nabla u_i(t,\Pi _{\mathbf{v}}(t,0;z))\cdot \mathbf{v}	(t,\Pi _{\mathbf{v}}(t,0;z)) \\ 
		=&u_i(t,\Pi _{\mathbf{v}}(t,0;z))\Big[-\mathrm{%
			div }\; \mathbf{v}(t,\Pi _{\mathbf{v}}(t,0;z)) +h_i(\textbf{u}(t,\Pi _{\mathbf{v}}(t,0;z))\Big],
	\end{aligned}
\end{equation*}
where $ h_i(\textbf{u}(t,\Pi _{\mathbf{v}}(t,0;z))=h_i(u_1(t,\Pi _{\mathbf{v}}(t,0;z),u_2(t,\Pi _{\mathbf{v}}(t,0;z)) $. Hence a classical solution of \eqref{1.1}-\eqref{1.3} (i.e. $C^1$ in time and space) must satisfy 
\begin{equation}\label{2.2}
	u_i(t,\Pi _{\mathbf{v}}(t,0;z))=\exp \left( \int_{0}^{t}h_i\bigl(\textbf{u}(l,\Pi _{%
		\mathbf{v}}(l,0;z))-\mathrm{%
		div }\; \mathbf{v}(l,\Pi _{\mathbf{v}}(l,0;z))dl\right) u_i\left(0, z\right),i=1,2,  
\end{equation}%
or equivalently 
\begin{equation}\label{2.3}
	u_i(t,z)=\exp \left( \int_{0}^{t}h_i\bigl(\textbf{u}(l,\Pi _{\mathbf{v}}(l,t;z))\bigr)-
	\mathrm{div }\; \mathbf{v}(l,\Pi _{\mathbf{v}}(l,t;z))dl\right) u_i\left(0,
	\Pi_{\mathbf{v}}(0,t;z)\right), i=1,2, 
\end{equation}%
where
\begin{equation} \label{2.4}
	\mathbf{v}(t,x)=-\frac{1}{|\T^N|}\int_{\mathbb{T}^{N}}\nabla K(x-y)(u_1+u_2)(t,y)\di  y. 
\end{equation}%
The above computations lead us to the following definition of solution.

\begin{definition}[Solution integrated along the characteristics]\label{DEF2.2}
	Let $\mathbf{u}_{0}\in L_{per ,+}^{\infty }\left( \mathbb{R}%
	^{N}\right)^2$, $\tau>0$ be given. A function $\mathbf u\in C\left( \left[
	0,\tau \right] ,L_{per ,+}^{1}\left( \mathbb{R}^{N}\right) \right)^2 \cap
	L^{\infty }\left( (0,\tau),L_{per ,+}^{\infty }\left( \mathbb{R}%
	^{N}\right) \right)^2 $ is said to be a solution integrated along the
	characteristics of \eqref{1.1}-\eqref{1.3}, if $u_i$ satisfies \eqref{2.3} for $ i=1,2 ,$
	with $\mathbf{v}$ defined in \eqref{2.4}.
\end{definition}
	 We will use a fixed point theorem to prove the existence and the uniqueness of the solutions integrated along the characteristics. Consider  
	 \begin{equation}\label{2.5}
	 	\mathbf w =\begin{pmatrix}
	 	w_1,w_2
	 	\end{pmatrix},\quad w_i(t,x):=u_i(t,\Pi _{\mathbf{v}}(t,0;x)),\,i=1,2,
	 \end{equation}
and we will construct a fixed point problem for the pair $ (\mathbf w,\mathbf v) $. 

	 If there exists a solution integrated along the characteristics, then by \eqref{2.2} we have  
	 \begin{equation}\label{2.6}
	 	{w}_i(t,x)=\exp \left( \int_{0}^{t}h_i\bigl(\textbf{w}(l,x)\bigr)-\mathrm{div}\,\mathbf{v}(l,\Pi _{\mathbf{v}}(l,0;x))dl\right)  u_i(0,x),\;i=1,2.
	 \end{equation}%
	 where $ h_i(\mathbf w(t,x)) = h_i(w_1(t,x),w_2(t,x)) $ for $ i=1,2 $. 
By the definition of $ \textbf{v} $ we obtain
\begin{equation}\label{2.7}
\begin{aligned}
\mathbf{v}(t,x)=&-\frac{1}{|\T^N|}\int_{\mathbb{T}^{N}}\nabla K(x-y)(u_1+u_2)(t,y) \di  y\\
=&-\int_{\R^N}\nabla \rho({x}-y)(u_1+u_2)(t,y) \di   y\\
=&-\int_{\R^N}\nabla \rho\left(x-\Pi _{\mathbf{v}}(t,0;z)\right) \sum_{i=1,2}u_i(t,\Pi_\textbf{v}(t,0;z)) \; \det \partial_{z}(\Pi_\textbf{v}(t,0;z)) \di   z\\
=&-\int_{\R^N}\nabla \rho\left(x-\Pi _{\mathbf{v}}(t,0;z)\right) \sum_{i=1,2}w_i(t,z) \; \det \partial_{z}(\Pi_\textbf{v}(t,0;z)) \di   z\\
\end{aligned}
\end{equation}
where we have used the change of variable $y=\Pi _{\mathbf{v}}(t,0;z)$. By using the determinant of Jacobi matrix in \eqref{2.1} and \eqref{2.6} we see that
\[ w_i(t,z) \det\,\partial_{z} (\Pi_\textbf{v}(t,0;z))=e^{\int_{0}^{t}h_i(\textbf{w}(l,z)) \di   l} u_i\left(0, z\right),\;i=1,2,\]
thus equation \eqref{2.7} becomes
\begin{equation}\label{2.8}
\begin{aligned}
\mathbf{v}(t,x)
=&-\int_{\R^N}\nabla \rho\left({x}-\Pi _{\mathbf{v}}(t,0;z)\right) \sum_{i=1,2}e^{\int_{0}^{t}h_i(\textbf{w}(l,z))dl} u_i\left(0,z\right)dz\\
=&-\frac{1}{|\T^N|}\int_{\mathbb{T}^{N}}\nabla K({x}-\Pi _{\mathbf{v}}(t,0;z))\sum_{i=1,2}e^{\int_{0}^{t}h_i(\textbf{w}(l,z))dl} u_i\left(0,z\right) dz.\\
\end{aligned}
\end{equation}
Therefore incorporating equations  \eqref{2.6} and \eqref{2.8} we are led to find the solution of the following problem
\begin{equation}\label{2.9}
\left\lbrace\begin{aligned}
&w_i(t,x)=\exp \left( \int_{0}^{t}h_i\bigl(\mathbf w(l,x))-\mathrm{div}\,\mathbf{v}(l,\Pi _{\mathbf{v}}(l,0;x))dl\right)  u_i\left(0, x\right),\;i=1,2,\\
&\mathbf{v}(t,x)=-\frac{1}{|\T^N|}\int_{\mathbb{T}^{N}}\nabla K({x}-\Pi _{\mathbf{v}}(t,0;z))\sum_{i=1,2}e^{\int_{0}^{t}h_i\bigl(\textbf{w}(l,z))\bigr)dl} u_i\left(0,z\right) dz.
\end{aligned}\right.
\end{equation}%
In order to choose a proper space for $ \textbf{w} $ and $ \textbf{v} $, we observe the following estimation
\begin{equation*}
\| \int_{0}^{t}h_i\bigl(\mathbf w(l,x))-\mathrm{div}\,\mathbf{v}(l,\Pi _{\mathbf{v}}(l,0;x))dl \|_{L^\infty}\leq t\left(\bar{h}+\|\mathbf v \|_{C^1} \right),\,i=1,2.
\end{equation*}
where $ \bar{h}:=\sup_{u_1,u_2\geq 0}\sum_{i=1,2} h_i(u_1,u_2)  $. Hence we can choose the following spaces
\begin{equation*}
 \textbf{w}=\begin{pmatrix}
 w_1,w_2
 \end{pmatrix}\in C\left( \left[ 0,\tau \right] ,L_{per ,+}^{\infty}\left( \mathbb{%
 	R}^{N}\right) \right)^2,\; \textbf{v}\in C([0,\tau],C_{per}^1(\R^N)^N).
\end{equation*}%
	Our fixed point problem can be written as
	\begin{equation*}
			 \begin{pmatrix}
			 \textbf{w}\\ \textbf{v}
			 \end{pmatrix}\in C\left( \left[ 0,\tau \right] ,L_{per ,+}^{\infty}\left( \mathbb{R}^{N}\right) \right)^2\times C([0,\tau],C_{per}^1(\R^N)^N) \text{ and }  \mathcal T  \begin{pmatrix}
			 \textbf{w}\\ \textbf{v}
			 \end{pmatrix} =  \begin{pmatrix}
			 \mathbf{w^1}\\ \mathbf{v^1}
			 \end{pmatrix},
	\end{equation*}
	wherein $\mathbf{w^1}$ and $ \mathbf{v^1} $ are defined by 
	\begin{equation}\label{2.10}
			\begin{aligned}
			&\mathbf{w^1}(t,x)=\begin{pmatrix}
			\exp \left( \int_{0}^{t}h_1\bigl(\textbf{w}(l,x)\bigr)-\mathrm{div}\,\mathbf{v^1}(l,\Pi _{\mathbf{v}}(l,0;x))dl\right) u_1\left(0, x\right)\\
			\exp \left( \int_{0}^{t}h_2\bigl(\textbf{w}(l,x)\bigr)-\mathrm{div}\,\mathbf{v^1}(l,\Pi _{\mathbf{v}}(l,0;x))dl\right) u_2\left(0, x\right)
			\end{pmatrix},\\
			&\mathbf{v^1}(t,x)=-\frac{1}{|\T^N|}\int_{\mathbb{T}^{N}}\nabla K({x}-\Pi _{\mathbf{v}}(t,0;z))\sum_{i=1,2}e^{\int_{0}^{t}h_i\bigl(\textbf{w}(l,z))\bigr)dl} u_i\left(0,z\right) dz.
			\end{aligned}			
	\end{equation}
	
%\begin{remark}
%	The following Theorems \ref{THM2.3} and \ref{THM2.5} show us that the existence of the continuous semiflow and the uniqueness of solution can be ensured and any solution starting from a positive initial value will stay positive. If the initial value $ \mathbf u_0 $ belongs to a proper space, then the solution corresponds to the classical solution.
%\end{remark}    

\begin{theorem}\label{THM2.3}
	Let Assumption \ref{ASS1.1} and Assumption \ref{ASS1.2} be satisfied.
	For each $\mathbf{u}_{0}\in L_{per,+}^{\infty }\left( \mathbb{R}^N\right)^2 ,$
	system \eqref{1.1}-\eqref{1.2} has a unique solution integrated along the
	characteristics 
	\[ t\mapsto U(t)\mathbf u_0 \text{ in } C\left( \left[ 0,+\infty
	\right) ,L_{per,+}^{1}\left( \mathbb{R}^N\right) \right)^2\cap
	L^\infty_{loc}\left([0,\infty),L_{per,+}^{\infty}\left( \mathbb{R}%
	^N\right)\right)^2. \]
	Moreover $\left\{ U(t)\right\} _{t\geq 0}$ is a
	continuous semiflow on $L_{per,+}^{1}\left( \mathbb{R}^N\right)^2,$
	that is to say 
	
	\begin{itemize}
		\item[{\rm (i)}] $U(t)U(s)=U(t+s),\forall t,s\geq 0$ and $U(0)=I$;
		
		\item[{\rm (ii)}] The map $(t,\mathbf u_0)\rightarrow U(t)\mathbf u_0$ maps every bounded set
		of $\left[ 0,+\infty \right) \times L_{per,+}^{\infty }\left( \mathbb{R}^N%
		\right)^2 $ into a bounded set of $L_{per,+}^{\infty }\left( \mathbb{R}^N%
		\right)^2 $;
		
		\item[{\rm  (iii)}] If $\left\{ t_{n}\right\} _{n\in \mathbb{N%
			}}(\subset \left[ 0,+\infty \right) )\rightarrow t<+\infty $ and $\left\{
			\mathbf u_0^{n}\right\} _{n\in \mathbb{N}}$ is bounded sequence in $%
			L_{per,+}^{\infty }\left( \mathbb{R}^N\right)^2 $ such that $\left\|
			\mathbf u_0^{n}-\mathbf u_0\right\|_{L^1}\rightarrow 0\text{ as }n\rightarrow +\infty$, then 
			\begin{equation*}
			\left\| U(t_{n})\mathbf u_0^{n}-U(t)\mathbf u_0\right\|_{L^1}\rightarrow 0\text{ as }n\rightarrow +\infty,
			\end{equation*}
		\end{itemize}
		where the norm is the product norm of $ L_{per,+}^{1}\left( \mathbb{R}^N\right)^2 $ (see Remark \ref{REM1.4}), the same for the following notation.
		
		\noindent The semiflow $U$ also satisfies the two following properties 
		\begin{equation}\label{2.11}
		U(t)\mathbf u_0\geq 0,\forall \mathbf u_0\geq 0,\forall t\geq 0,  
		\end{equation}%
		\begin{equation}
		\left\| U(t)\mathbf u_0\right\|_{L^1}\leq e^{t\bar{h}}\left\| \mathbf u_0\right\|_{L^1},\forall t\geq 0.
		\label{2.12}
		\end{equation}%
		where we define 
		\begin{equation}\label{2.13}
		\bar{h}:=\sup_{u_1,u_2\geq 0}\sum_{i=1,2} h_i(u_1,u_2) 
		\end{equation}
	\end{theorem}	
We need the following lemma before we prove Theorem \ref{THM2.3}. 
\begin{lemma}\label{LEM2.4}
	Suppose $ \mathbf{v},\tilde{\mathbf{v}} \in C([0,\tau],C_{per}^1(\R^N)^N) $. 
	Then for any $ \tau>0  $, we have
	\begin{equation*}
		\sup_{t\in [0,\tau]}\| \Pi_{\mathbf{v}}(t,0;\cdot)-\Pi_{\mathbf{\tilde{v}}}(t,0;\cdot) \|_{L^\infty}\leq  \tau \sup_{t\in [0,\tau]}\| \mathbf{v}(t,\cdot)-\tilde{\mathbf{v}}(t,\cdot) \|_{L^\infty} e^{\tau \sup_{t\in [0,\tau]}\|\mathbf{v}(t,\cdot) \|_{C^1}}.
	\end{equation*}
\end{lemma}
\begin{proof}
	For any fixed $ t\in [0,\tau] $,
	\[ \partial_t \left(\Pi_\textbf{v}(t,0;x)- \Pi_{\mathbf{\tilde{v}}}(t,0;x) \right)=\textbf{v}(t,\Pi_\textbf{v}(t,0;x))-\tilde{\textbf{v}}(t,\Pi_{\tilde{\textbf{v}}}(t,0;x)) \]
	Hence one obtains
	\begin{equation*}
		\begin{aligned}
			&\| \Pi_{\mathbf{v}}(t,0;\cdot)-\Pi_{\mathbf{\tilde{v}}}(t,0;\cdot) \|_{L^\infty}\\
			\leq &t\| \textbf{v}(t,\cdot)-\tilde{\textbf{v}}(t,\cdot) \|_{L^\infty}+\int_0^t \|\textbf{v}(t,\cdot) \|_{C^1}\| \Pi_{\mathbf{v}}(l,0;\cdot)-\Pi_{\mathbf{\tilde{v}}}(l,0;\cdot) \|_{L^\infty}\di  l
		\end{aligned}
	\end{equation*}
	By Gronwall inequality, we obtain
	\[ \sup_{t\in [0,\tau]}\| \Pi_{\mathbf{v}}(t,0;\cdot)-\Pi_{\mathbf{\tilde{v}}}(t,0;\cdot) \|_{L^\infty}\leq \tau \sup_{t\in [0,\tau]}\| \textbf{v}(t,\cdot)-\tilde{\textbf{v}}(t,\cdot) \|_{L^\infty} e^{\tau \sup_{t\in [0,\tau]}\|\textbf{v}(t,\cdot) \|_{C^1}}.  \]
\end{proof}

\begin{proof}[Proof of Theorem \ref{THM2.3}]
	We prove this theorem by showing that the contraction mapping theorem applies for $\mathcal T$ as soon as $\tau>0$ is small enough. This will ensure the existence and uniqueness of the local solution.
	To do so, we fix $\tau>0$ that will be chosen later and we consider the Banach space $Z$ defined by $ Z:=X\times Y $ where 
	\begin{equation*}
		X:=C\left( \left[ 0,\tau \right] ,L_{per}^{\infty }\left( \mathbb{R}^{N}\right) \right)^2,\quad Y:=C([0,\tau],C_{per}^1(\R^N)^N)
	\end{equation*}
	endowed with the norm:
	\begin{equation*}
		\left\|\begin{pmatrix}
		\textbf{w}\\\textbf{v}
		\end{pmatrix}\right\|_{Z}=\| \textbf{w} \|_{X}+\| \textbf{v} \|_{Y}.
	\end{equation*}
	We also introduce the closed subset $X_+\subset X$ defined by:
	\begin{equation*}
		X_+=C\left( \left[ 0,\tau \right] ,L_{per,+}^{\infty}\left( \mathbb{%
			R}^{N}\right) \right)^2.
	\end{equation*}
	and define $ Z_+=X_+\times Y $.	Note that due to \eqref{2.10} one has 
	\begin{equation}\label{2.14}
		\mathcal T(Z_+)\subset Z_+.
	\end{equation}
	For each given $\begin{pmatrix}
	\textbf{w}\\ \textbf{v}
	\end{pmatrix} \in X$ and $\kappa>0$ we denote by $\overline{B}_{Z}\left(\begin{pmatrix}
	\textbf{w}\\ \textbf{v}
	\end{pmatrix},\kappa\right)$ the closed ball in $Z$ of center $\begin{pmatrix}
	\textbf{w}\\ \textbf{v}
	\end{pmatrix}$ and radius $\kappa$. Now for any given initial distribution
	$$ \mathbf u_0=(u_1(0,\cdot),u_2(0,\cdot))\in X_+ \text{ and } \mathbf v_0= -\sum_{i=1,2}\nabla K\circ (u_i(0,\cdot)), $$
	and $\kappa>0$ any given constant. 
	We claim that there exists $\widehat{\tau}>0$ such that for each $\tau\in \left(0,\widehat{\tau}\right)$:
	\begin{equation}\label{2.15}
		\mathcal T\left(Z_+\cap \overline{B}_{Z}\left( \begin{pmatrix}\mathbf
		u_0\\ \mathbf v_0
		\end{pmatrix},\kappa\right)\right)\subset Z_+\cap \overline{B}_{Z}\left(\begin{pmatrix}\mathbf
		u_0\\ \mathbf v_0
		\end{pmatrix},\kappa\right).
	\end{equation}
	To prove this claim, for any given $ \begin{pmatrix}
	\mathbf w\\\mathbf v
	\end{pmatrix} \in Z_+\cap \overline{B}_{Z}\left(\begin{pmatrix}\mathbf
	u_0\\ \mathbf v_0
	\end{pmatrix},\kappa\right)$, we estimate each component separately.  
	Recalling the definition of $ \mathcal T $ in \eqref{2.10} one obtains
	\begin{equation}\label{2.16}
		\sup_{t\in [0,\tau]}\|\mathbf{w^1} (t,\cdot)-\mathbf u_0(\cdot)\|_{L^\infty}\leq \|\mathbf u_0 \|_{L^\infty} \theta  e^\theta,
	\end{equation}
	where $ \theta $ is defined as
	\begin{equation*}
		\theta= \sum_{i=1}^2\int_{0}^{\tau } \| h_i\bigl(\textbf{{w}}(l,x)\bigr)-\mathrm{div}\,\mathbf{v}(l,\Pi _{\mathbf{v}}(l,0;x)) \|_{L^\infty} \di  l \leq \tau \left(h_\kappa+\|\mathbf v \|_{Y}\right)\leq \tau (h_\kappa+\kappa+\|\mathbf{v}_0\|_Y),
	\end{equation*}
	and we have set 
	\begin{equation}\label{2.17}
	 h_\kappa:= \sup_{0\leq u_1,u_2 \leq \kappa+\|\mathbf{u}_0\|_{L^\infty}} \sum_{i=1,2}|h_i(u_1,u_2)|.
	\end{equation}
	On the other hand, 
	\begin{equation}\label{2.18}
	\begin{aligned}
	&\sup_{t\in [0,\tau]}\|\mathbf{v^1} (t,\cdot)-\mathbf v_0(\cdot)\|_{C^1}\\
	\leq&\|\mathbf u_0 \|_{L^\infty}\frac{1}{|\mathbb{T}^N|}\sup_{t\in [0,\tau]}\|\int_{\mathbb{T}^N}\nabla K(x-\Pi_{\mathbf{v}}(t,0;z)) \sum_{i=1,2}e^{\int_{0}^{t}h_i(\mathbf w(l,z))\di  l}-\nabla K(x-z)dz \|_{C^1}\\
	\leq&\|\mathbf u_0 \|_{L^\infty} \left\lbrace \left(\| K\|_{C^2}+\| K\|_{C^3}\right)\sup_{t\in[0,\tau]} \| \Pi_{\mathbf{v}}(t,0;z)-z \|_{L^\infty}+\left(\| K\|_{C^1}+\| K\|_{C^2}\right) | e^{\tau h_\kappa} -1|  \right\rbrace\\
	\leq&2 \|\mathbf u_0 \|_{L^\infty}\| K\|_{C^3} \left\lbrace \sup_{t\in[0,\tau]} \| \Pi_{\mathbf{v}}(t,0;z)-\Pi_{\mathbf{v}_0}(t,0;z) \|_{L^\infty}+ \sup_{t\in[0,\tau]} \| \Pi_{\mathbf{v}_0}(t,0;z)-z \|_{L^\infty}+| e^{\tau h_\kappa} -1|  \right\rbrace.
	\end{aligned}
	\end{equation}
	Recalling Lemma \ref{LEM2.4}, we have 
	\[\sup_{t\in[0,\tau]} \| \Pi_{\mathbf{v}}(t,0;z)-\Pi_{\mathbf{v}_0}(t,0;z) \|_{L^\infty} \leq \tau \sup_{t\in [0,\tau]}\| \mathbf{v}(t,\cdot)-\mathbf{v}_0(t,\cdot) \|_{L^\infty} e^{\tau \sup_{t\in [0,\tau]}\|\mathbf{v}(t,\cdot) \|_{C^1}}\leq \tau \kappa e^{\tau (\kappa+\|\mathbf v_0\|_Y)}. \]
	Therefore, equation \eqref{2.18} becomes
	\begin{align*}
	&\sup_{t\in [0,\tau]}\|\mathbf{v^1} (t,\cdot)-\mathbf v_0(\cdot)\|_{C^1}\\
	\leq& 2 \|\mathbf u_0 \|_{L^\infty}\| K\|_{C^3} \left\lbrace \tau \kappa e^{\tau (\kappa+\|\mathbf v_0\|_Y)}+ \sup_{t\in[0,\tau]} \| \Pi_{\mathbf{v}_0}(t,0;z)-z \|_{L^\infty}+| e^{\tau h_\kappa} -1|  \right\rbrace.
	\end{align*}
	Since $ \sup_{t\in[0,\tau]} \| \Pi_{\mathbf{v}_0}(t,0;z)-z\|_{L^\infty} \to 0 $ as $ \tau \to 0 $, incorporating \eqref{2.16}, \eqref{2.18} and \eqref{2.14}, then the above estimations complete the proof of \eqref{2.15} by choosing a $ \hat{\tau} $ small enough.

	We now claim that for any 
	\[ \begin{pmatrix}
	\mathbf w\\\mathbf v
	\end{pmatrix},\;\begin{pmatrix}
	\mathbf{\tilde{w}}\\\mathbf{\tilde{v}}
	\end{pmatrix} \in Z_+\cap \overline{B}_{Z}\left(\begin{pmatrix}\mathbf
	u_0\\ \mathbf v_0
	\end{pmatrix},\kappa\right), \]	
	where
	\[\textbf{w}(t,x)=\textbf{u}(t,\Pi_{\textbf{v}}(t,0;x)),\;\tilde{\textbf{w}}(t,x)=\tilde{\textbf{u}}(t,\Pi_{\tilde{\textbf{v}}}(t,0;x)),\]
	there exists $\tau^*\in \left(0,\widehat{\tau}\right)$ such that for each $\tau\in \left(0,\tau^*\right)$ we can find a $L(\tau)\in (0,1)$ such that
	\begin{equation}\label{2.19}
		\left\|\mathcal T \begin{pmatrix}
		\textbf{w}\\ \textbf{v}
		\end{pmatrix}-\mathcal T \begin{pmatrix}
		\tilde{\textbf{w}}\\\tilde{\textbf{v}}
		\end{pmatrix}\right\|_{Z}\leq L(\tau)\left\|\begin{pmatrix}
		\textbf{w}\\ \textbf{v}
		\end{pmatrix}- \begin{pmatrix}
		\tilde{\textbf{w}}\\\tilde{\textbf{v}}
		\end{pmatrix}\right\|_{Z}.
	\end{equation}
	To prove this claim, as before we estimate each component separately. For any given $\tau\in\left(0,\tau^*\right)$
	\begin{equation*}
	\begin{aligned}
	&\sup_{t\in [0,\tau]}\left\| \mathbf{w^1}(t,x)-\mathbf{\widetilde{w}^1}(t,x)\right\|_{L^\infty}\\
	=&\sum_{i=1}^{2}\|\mathbf u_0 \|_{L^\infty}\sup_{t\in [0,\tau]}\| e^{\int_{0}^{t}h_i({\textbf{w}}(l,x))-\mathrm{div} \,\mathbf{v}(l,\Pi _{\mathbf{v}}(l,0;x)) \di  l}-e^{\int_{0}^{t}h_i({\tilde{\textbf{w}}}(l,x))-\mathrm{div}\,\tilde{\mathbf{v}}(l,\Pi _{\tilde{\mathbf{v}}}(l,0;x)) \di  l } \|_{L^\infty}\\
	\leq & \|\mathbf u_0 \|_{L^\infty} \Big( e^{\tau (\kappa+\|\mathbf{v}_0\|_Y)} \underbrace{\sum_{i=1}^{2}\sup_{t\in [0,\tau]}\|  e^{\int_{0}^{t}h_i({\textbf{w}}(l,x)) \di  l}-e^{\int_{0}^{t}h_i({\tilde{\textbf{w}}}(l,x)) \di  l } \|_{L^\infty}}_{\text{I}} \\ 
	&+e^{\tau h_\kappa} \underbrace{\sup_{t\in [0,\tau]}\|e^{-\int_{0}^{t}\mathrm{div}\,\mathbf{v}(l,\Pi _{\mathbf{v}}(l,0;x)) \di  l}-e^{-\int_{0}^{t}\mathrm{div}\,\tilde{\mathbf{v}}(l,\Pi _{\tilde{\mathbf{v}}}(l,0;x)) \di  l }\|_{L^\infty}}_{\text{II}}\Big).
	\end{aligned}
	\end{equation*}
	\textit{Estimation for I:} We estimate the first term. Since for any $ x,y\in \R $, we  have $ |e^x-e^y|\leq e^{\max\lbrace |x|,|y| \rbrace}|x-y| $. Thus
	\begin{equation}\label{2.20}
	\begin{aligned}
	&\sum_{i=1}^{2}\sup_{t\in [0,\tau]}\|  e^{\int_{0}^{t}h_i({\textbf{w}}(l,x)) \di  l}-e^{\int_{0}^{t}h_i({\tilde{\textbf{w}}}(l,x)) \di  l } \|_{L^\infty}\leq e^{\tau h_\kappa} \sum_{i=1}^{2}\| \int_{0}^{t}h_i\bigl({\textbf{w}}(l,x)\bigr) -h_i\bigl({\tilde{\textbf{w}}}(l,x)\bigr)\di  l \|_{L^\infty}\\
	\leq & \tau e^{\tau h_\kappa}  |\nabla h_{\kappa}| \|\textbf{w}-\tilde{\textbf{w}} \|_{X},
	\end{aligned}
	\end{equation}
	where $ |\nabla h_{\kappa}|=\sum_{i=1}^{2}\sup_{u_1,u_2\in[0,\|\mathbf{u}_0\|_{L^\infty}+\kappa)}|\nabla h_i(u_1,u_2)| $ and\ $ h_\kappa $ is defined in \eqref{2.17}.\\

%	\vspace{0.1cm}
	\noindent	\textit{Estimation for II:} For the second term, 
%	since
%	\[ \|e^{\int_{0}^{\tau}-\mathrm{div}\,\mathbf{v}(l,\Pi _{\mathbf{v}}(l,0;x)) \di l } \|_{L^\infty}\leq  e^{\tau \|\Delta K \|_{L^\infty} e^{\tau h_\kappa}\| \textbf{u}_0 \|_{L^\infty} },\]
%	then 
	we obtain 
	\begin{equation*}
	\begin{aligned}
	&\sup_{t\in [0,\tau]}\|e^{-\int_{0}^{t}\mathrm{div}\,\mathbf{v}(l,\Pi _{\mathbf{v}}(l,0;x)) \di  l}-e^{-\int_{0}^{t}\mathrm{div}\,\tilde{\mathbf{v}}(l,\Pi _{\tilde{\mathbf{v}}}(l,0;x)) \di  l }\|_{L^\infty}\\
	\leq& \tau e^{\tau (\kappa+\|\mathbf{v}_0\|_Y)}  \sup_{t\in [0,\tau]}\| \mathrm{div} \,\mathbf{v}(t,\Pi _{\mathbf{v}}(t,0;x))-\mathrm{div}\,\tilde{\mathbf{v}}(t,\Pi _{\tilde{\mathbf{v}}}(t,0;x)) \|_{L^\infty},
	\end{aligned}
	\end{equation*}
	while due to \eqref{2.8} the last term has the following estimation 
	\begin{equation*}
	\begin{aligned}
	&\sup_{t\in [0,\tau]}\| \mathrm{div}\,\mathbf{v}(t,\Pi _{\mathbf{v}}(t,0;x))-\mathrm{div}\,\tilde{\mathbf{v}}(t,\Pi _{\tilde{\mathbf{v}}}(t,0;x)) \|_{L^\infty}\\
	\leq&\frac{1}{|\T^N|}\sum_{i=1}^{2}\sup_{t\in [0,\tau]}\left\| \int_{\T^N} \Delta K({\Pi _{\mathbf{v}}(t,0;x)}-\Pi _{\mathbf{v}}(t,0;z))e^{\int_{0}^{t}h_i(\textbf{w}(l,z)) \di  l}\right. \\
	&\left. \vphantom{\int_{\T^N}} - \Delta K({\Pi _{\tilde{\mathbf{v}}}(t,0;x)}-\Pi _{\tilde{\mathbf{v}}}(t,0;z))e^{\int_{0}^{t}h_i(\tilde{\textbf{w}}(l,z)) \di  l} dz\right\|_{L^\infty}\|\mathbf u_0\|_{L^\infty}\\
	\leq &\|\mathbf u_0\|_{L^\infty}\left\lbrace\|K \|_{C^2} \sum_{i=1}^{2}\sup_{t\in [0,\tau]}\left\| e^{\int_{0}^{t }h_i({\textbf{w}}(l,x)) \di  l}-e^{\int_{0}^{t}h_i({\tilde{\textbf{w}}}(l,x)) \di  l } \right\|_{L^\infty}\right.\\
	&\left.+2e^{\tau h_\kappa} \| K \|_{C^3}\sup_{t\in [0,\tau]}\| \Pi_{\mathbf{v}}(t,0;x)-\Pi_{\tilde{\mathbf{v}}}(t,0;x) \|_{L^\infty}\right\rbrace.
	\end{aligned}
	\end{equation*}
	where the first part can be estimated by \eqref{2.20}. Recalling Lemma \ref{LEM2.4} and since $ \mathbf v ,\tilde{\mathbf{v}}\in \overline{B}_Y(\mathbf v_0,\kappa) $ we have
	\begin{equation} \label{2.21}
	\sup_{t\in [0,\tau]}\| \Pi_{\mathbf{v}}(t,0;x)-\Pi_{\tilde{\mathbf{v}}}(t,0;x) \|_{L^\infty}\leq  \tau \sup_{t\in [0,\tau]}\| \mathbf{v}(t,\cdot)-\tilde{\mathbf{v}}(t,\cdot) \|_{L^\infty} e^{\tau( \kappa+\|\mathbf{v}_0\|_Y)}. 
	\end{equation}
	 Incorporating the estimation in \eqref{2.20}, we are led to the following estimation
	\begin{equation*}
	\begin{aligned}
	&\left\| \mathbf{w^1}-\mathbf{\widetilde{w}^1}\right\|_{X} 
	\leq  L_1(\tau) \big(\| \textbf{w}-\tilde{\textbf{w}}\|_{X} + \|\mathbf v-\tilde{\mathbf v }\|_{Y}  \big)
	\end{aligned}
	\end{equation*}
	with $ L_1(\tau)\to 0 $ as $ \tau \to 0 $. 
	
	To complete the proof of \eqref{2.19}, let us notice 
	\begin{equation*}
	\begin{aligned}
	\left\| \mathbf{v^1}-\mathbf{\widetilde{v}^1}\right\|_{Y}=&\sup_{t\in [0,\tau]}\left\| \mathbf{v^1}(t,x)-\mathbf{\widetilde{v}^1}(t,x)\right\|_{C^1}\\
	=&\frac{1}{|\mathbb{T}^N|}\sup_{t\in [0,\tau]}\|\int_{\mathbb{T}^N}\nabla K(x-\Pi_{\mathbf{v}}(t,0;z)) \sum_{i=1,2}e^{\int_{0}^{t}h_i(\mathbf w(l,z))\di  l}u_i(0,z)\di  z\\
	&-\int_{\mathbb{T}^N}\nabla K(x-\Pi_{\mathbf{\tilde{v}}}(t,0;z)) \sum_{i=1,2}e^{\int_{0}^{t}h_i(\mathbf {\tilde{w}}(l,z))\di  l}u_i(0,z)dz \|_{C^1}\\
	\leq&\|\mathbf u_0 \|_{L^\infty} \left\lbrace 2e^{\tau h_\kappa}\left(\| K\|_{C^2}+\| K\|_{C^3}\right)\sup_{t\in[0,\tau]} \| \Pi_{\mathbf{v}}(t,0;x)-\Pi_{\tilde{\mathbf{v}}}(t,0;x) \|_{L^\infty}\right.\\
	&\left.+\left(\| K\|_{C^1}+\| K\|_{C^2}\right) \sum_{i=1}^{2}\sup_{t\in [0,\tau]}\|  e^{\int_{0}^{t}h_i({\textbf{w}}(l,x)) \di  l}-e^{\int_{0}^{t}h_i({\tilde{\textbf{w}}}(l,x)) \di  l } \|_{L^\infty}  \right\rbrace\\
	\leq& L_2(\tau) \big(\| \textbf{w}-\tilde{\textbf{w}}\|_{X} + \|\mathbf v-\tilde{\mathbf v }\|_{Y}  \big),
	\end{aligned}
	\end{equation*}
and by using \eqref{2.20} and \eqref{2.21}	we deduce that 
$$ 
\lim_{\tau \to 0}L_2(\tau)=0. 
$$ 
Let $ L(\tau):=L_1(\tau)+L_2(\tau) $ we complete the proof of \eqref{2.19}.
	
	Finally one concludes from \eqref{2.15} and \eqref{2.19} that for $\tau$ small enough, the contraction mapping theorem applies to operator $\mathcal T$. Hence the operator $\mathcal T$ has a unique fixed point in $Z_+\cap \overline{B}_{Z}\left(\begin{pmatrix}\mathbf
	u_0\\ \mathbf v_0
	\end{pmatrix},\kappa\right)$. Recalling \eqref{2.5}, this ensures the existence and uniqueness of the local solution integrated along the characteristic of \eqref{1.1}. The positivity property \eqref{2.11} follows from the same arguments. The semiflow property in Theorem \ref{THM2.3}-{\rm(i)} follows by a standard uniqueness argument.
	
	Next we show that the semiflow is globally defined and the properties {\rm(ii)} and {\rm(iii)} of the semiflow. In fact, since
	\begin{equation}\label{2.22}
	u_i(t,x)=\exp \left( \int_{0}^{t}h_i\bigl(\textbf{u}(l,\Pi_\textbf{v}(l,t;x)))-\mathrm{div} \;\mathbf{v}(l,\Pi _{\mathbf{v}}(l,t;x))\di  l\right) u_i\left(0, \Pi_\textbf{v}(0,t;x)\right),i=1,2.
	\end{equation}%
	Thus recall the definition $ \bar{h} $ in \eqref{2.13} we have
	\begin{equation}\label{2.23}
	\sup_{t\in [0,\tau]}\Vert \textbf{u}(t,\cdot)\|_{L^\infty}\leq e^ {\tau \big(\bar{h} + \|\Delta K \|_{L^\infty} e^{\tau \bar{h}}\| \textbf{u}_0\|_{L^\infty}\big)} \Vert \mathbf u_0\Vert _{\infty
	},\;\forall \tau\geq 0.
	\end{equation}%
	The result {\rm(ii)} follows.
	
Moreover, one deduces from \eqref{2.22} that
	\begin{equation*}
	u_i(t,x)\leq \exp \left( t\bar{h}\right)\exp \left(\int_0^t-\mathrm{div} \;\mathbf{v}(l,\Pi _{\mathbf{v}}(l,t;x))\di  l\right)  u_i\left(0,
	\Pi_{\mathbf{v}}(0,t;x)\right),i=1,2,
	\end{equation*}
	then we integrate over $ \T^N $, using the change of variable $ x=\Pi_\textbf{v}(t,0,z) $ to right hand side, which completes the proof \eqref{2.12}, i.e.,
	\begin{equation}\label{2.24}
	\Vert u_i(t,\cdot)\|_{L^1}\leq e^{ t\bar{h} } \Vert  u_i(0,\cdot)\|_{L^1},\;i=1,2,\forall t\geq 0.
	\end{equation}
In the last part of the proof we study the $ L^1 $ continuity of the semiflow. For any $ 0\leq s\leq t $,
	\begin{equation}\label{2.25}
	\begin{aligned}
	&\| U(t)\mathbf u_0-U(s)\mathbf u_0 \|_{L^1}\leq e^{s\bar{h}} \| U(t-s)\mathbf u_0-\mathbf u_0 \|_{L^1}\\
	=&e^{s\bar{h}}\sum_{i=1}^{2} \| e^{\int_0^{t-s}h_i(\textbf{u}(l,\Pi_\mathbf{v}(l,t-s;\cdot)))-\mathrm{div}\; \mathbf{v}(l,\Pi_\mathbf{v}(l,t-s;\cdot))\di  l}u_i(0,\Pi_\mathbf{v}(0,t-s;\cdot)) -u_i(0,\cdot)\|_{L^1},
	\end{aligned}
	\end{equation}
	since 
	\[ \sum_{i=1}^{2}\| \int_0^{t-s}h_i(\textbf{u}(l,\Pi_\mathbf{v}(l,t-s;\cdot)))-\mathrm{div}\; \mathbf{v}(l,\Pi_\mathbf{v}(l,t-s;\cdot))\di  l\|_{L^\infty} \leq J(t-s), \]
	where 
	\[  J(\tau):=\tau \left(\bar{h}+\|\Delta K \|_{C^0} e^{\tau \bar{h}} \|\textbf{u}_0 \|_{L^\infty}\right).\]
	 Then from \eqref{2.25} we have
	\begin{equation}\label{2.26}
		\begin{aligned}
		&\| U(t)\mathbf u_0-U(s)\mathbf u_0 \|_{L^1}\\
		\leq& e^{s\bar{h}} \|\mathbf u_0(\Pi_\textbf{v}(0,t-s;\cdot))-\mathbf u_0 \|_{L^1}e^{J(t-s)}+e^{s\bar{h}}\| \mathbf u_0\|_{L^1}\|e^{J(t-s)}-1 \|_{L^\infty}\to 0,\quad t\to s.
		\end{aligned}
	\end{equation}
	If $\left\{
	\mathbf u_0^{n}\right\} _{n\in \mathbb{N}}$ is bounded sequence in $%
	L_{per,+}^{\infty }\left( \mathbb{R}^N\right) $ such that $\left\|
	\mathbf u_0^{n}-\mathbf u_0\right\Vert _{1
	}\rightarrow 0\text{ as }n\rightarrow +\infty$, then by \eqref{2.24}, we have
	\[ \| U(t)\mathbf u_0^n-U(t)\mathbf u_0 \|_{L^1} \to 0 ,\quad n\to+\infty,\]
	together with \eqref{2.26}, we have proved the continuity of the semiflow in  {\rm(iii)}.	
	\end{proof}
	
	\begin{theorem}\label{THM2.5}
		Let Assumption \ref{ASS1.1} and Assumption  \ref{ASS1.2} be satisfied. In addition, $\mathbf u_0\in W_{per }^{1}\left( \mathbb{R}%
		^{N}\right)^2 $, then $U(\cdot)\mathbf u_0\in C^{1}\left( \left[ 0,+\infty \right)
		,L_{per }^{1}\left( \mathbb{R}^{N}\right) \right)^2 $. Moreover, if $\mathbf u_0\in C_{per }^{1}\left( \mathbb{R}^{N}\right)^2 $ then $%
		\mathbf u(t,x)=U(t) \mathbf u_0 (x)$ belongs to $C^{1}\left( \left[ 0,+\infty
		\right) \times \mathbb{R}^{N}\right)^2 $ and $u(t,x)$ is a
		classical solution of system \eqref{1.1}-\eqref{1.3}.
	\end{theorem}
	
	\begin{proof}[Sketch of the proof]
		If $ \mathbf u_0\in W_{per}^1(\R^N)^2 $, we claim $ U(\cdot)\mathbf u_0\in C^1([0,\infty),L_{per}^1(\R^N))^2 $. In fact, we define for $ i=1,2, $
		\begin{equation}\label{2.27}
		w_i(t,x)=e^{\int_0^t h_i(\mathbf w(l,x))-\mathrm{div}\,\mathbf{v}(l,\Pi _{\mathbf{v}}(l,0;x))\di  l} u_i(0,x)=:e^{\int_0^t h_i(\mathbf w(l,x))\di  l}B_i(t,x),
		\end{equation}
		where $ B_i(t,x):= e^{\int_0^t-\mathrm{div}\,\mathbf{v}(l,\Pi _{\mathbf{v}}(l,0;x))\di  l}u_i(0,x)$ is $ C([0,\tau],W_{per}^1(\R^N))$ by our assumption. Define the formal derivative $ \tilde{w}_i(t,.) =\nabla_x w_i(t,.) $, solving the following fixed point problem
		\[ \mathcal{F} \begin{pmatrix}
		\tilde{w}_1(t,x)\\\tilde{w}_2(t,x)\\\mathbf{v}
		\end{pmatrix} = \begin{pmatrix}
		\left( \int_0^t \sum_{j=1}^{2}\partial_{u_j}h_1(\mathbf{w}(l,x))\tilde{w}_j(l,x)\di  l\; B_1(t,x)+\nabla_x B_1(t,x) \right)e^{\int_0^t h_1(\mathbf w(l,x))\di  l}\\
		\left( \int_0^t \sum_{j=1}^{2}\partial_{u_j}h_2(\mathbf{w}(l,x))\tilde{w}_j(l,x)\di  l\; B_2(t,x)+\nabla_x B_2(t,x) \right)e^{\int_0^t h_2(\mathbf w(l,x))\di  l}\\
		-\frac{1}{|\T^N|}\int_{\mathbb{T}^{N}}\nabla K({x}-\Pi _{\mathbf{v}}(t,0;z))\sum_{i=1,2}e^{\int_{0}^{t}h_i\bigl(\mathbf{w}(l,z))\bigr)\di  l} u_i\left(0,z\right) \di  z		
		\end{pmatrix}, \]
		on space $ C([0,\tau],L_{per}^\infty(\R^N)^N)^2\times C([0,\tau],C_{per}^1(\R^N)^N) $ where $ \partial_{u_j}h_i(u_1,u_2) $ is the partial derivative of $ h_i $. Similarly, one can show the mapping $ \mathcal F $ is from $ C([0,\tau],L_{per}^\infty(\R^N)^N)^2\times C([0,\tau],C_{per}^1(\R^N)^N)$ to itself  and is a contraction if $ \tau  $ is small. Therefore,  
		\[ \tilde{w}_i(t,x) = \left( \int_0^t \sum_{j=1}^{2}\partial_{u_j}h_i(\mathbf{w}(l,x))\tilde{w}_j(l,x)\di  l\; B_i(t,x)+\nabla_x B_i(t,x) \right)e^{\int_0^t h_i(\mathbf w(l,x))\di  l},i=1,2, \]	 
		on $ [0,\tau] $, since by our assumption
		\[ \sup_{u_1,u_2\geq 0}\partial_{u_j}h_i(u_1,u_2)<\infty,\quad i=1,2,\;j=1,2, \]
		applying Gronwall inequality we have $ \tilde{\mathbf w}\in C([0,\infty),L^1_{per}(\R^N)^N)^2 $ for any positive time.

		Since we have for $ i=1,2, $ $ w_i(t,\Pi_\mathbf{v}(0,t;x))=u_i(t,x), $
		and 
		\[ \partial_t u_i(t,x)= \partial_t w_i(t,\Pi_\mathbf{v}(0,t;x))+ \tilde{w}_i(t,x) \cdot \partial_t\Pi_\mathbf{v}(0,t;x)\in C([0,\infty);L_{per}^1(\R^N)). \]
		If $ \mathbf u_0\in C^1( \R^N)^2 $, then $ B_i(t,x)\in C^{1}\left( \left[ 0,+\infty
		\right) \times \mathbb{R}^{N}\right) $ and by \eqref{2.27} we have $ \mathbf w\in C^1([0,\infty)\times \R^N)^2 $ therefore $ u $ is a classical solution.
	\end{proof}

	\begin{remark}[Conservation law]
		The above computations lead us to the following conservation law: for each
		Borel set $A\subset \mathbb{T}^N$ and each $0\leq s\leq t$: 
		\begin{equation*}
			\int_{\Pi_{\mathbf{v}}(t,s;A)}u_i(t,x)\di x=\int_A \exp\left[\int_s^t
			h_i\left(\mathbf u\left(l,\Pi_{\mathbf{v}}(l,s;z)\right)\right)\di  l\right]u_i(s,z)\di  z,i=1,2.
		\end{equation*}
	\end{remark}

\section{Segregation Property}
From the mono-layer cell populations co-culture experiments, we see the spreading speed of cell propagation is finite. Moreover, once the two cell populations confront each other, they will stop growing. Our next theorem will show that the solution along the characteristics can easily explain the preservation of the segregation. 
\begin{theorem}\label{THM3.1} Suppose $ \mathbf{u}=\mathbf{u}(t,x) $ is the solution of \eqref{1.1}-\eqref{1.3} provided by Theorem \ref{THM2.3}.
	For any initial distribution with $ u_1(0,x)u_2(0,x)=0 $ for all $ x\in \mathbb{T}^N $. Then $  u_1(t,x)u_2(t,x)=0 $ for any $ t>0 $ and $ x\in \mathbb{T}^N $.
\end{theorem}
\begin{proof}
	We argue by contradiction, assume there exist $ t_1 >0, x_1\in \mathbb{T}^N $ such that \[  u_1(t_1,x_1)u_2(t_1,x_1)>0. \] Since $ z\to \Pi_{\mathbf{v}}(t,s;z) $ is invertible from $ \R^N\to \R^N $, then there exists some $ x_0\in \R^N $ such that $ \Pi_{\mathbf{v}}(t_1,0;x_0)=x_1 $. Denote $ x_0=\tilde{x}_0+2\pi k_0 $ for some $ \tilde{x}_0\in \mathbb{T}^N $  and $ k_0\in \Z^N  $, thus by Lemma \ref{LEM2.1} we have 
	\[ 0<u_i(t_1,\Pi_{\mathbf{v}}(t_1,0;x_0))=u_i(t_1,\Pi_{\mathbf{v}}(t_1,0;\tilde{x}_0)+2\pi k_0)
	=u_i(t_1,\Pi_{\mathbf{v}}(t_1,0;\tilde{x}_0)).\]
	Thus, for any $ i=1,2, $
	\[ u_i(t_1,\Pi_{\mathbf{v}}(t_1,0;\tilde{x}_0))=
	\exp \left( \int_{0}^{t_1}h_i\bigl(\textbf{u}(l,\Pi _{%
		\mathbf{v}}(l,0;\tilde{x}_0))-\mathrm{%
		div }\; \mathbf{v}(l,\Pi _{\mathbf{v}}(l,0;\tilde{x}_0))\di l\right) u_i\left(0, \tilde{x}_0\right)>0, \]
	which implies 
	\[ u_i\left(0, \tilde{x}_0\right)>0,\;\forall i=1,2. \]
	This is a contradiction.
\end{proof}
\begin{remark}
	Suppose the dimension $ N=1 $ and $ u_1,u_2 $ are classical solutions, we give an illustration (see Figure \ref{FIG1}) of the segregation for the solutions integrated along the characteristics $ u_i(t,\Pi_{\mathbf{v}}(t,0;x)) $ for $ i=1,2 $. In fact, if there exists for some $ x_0 $ such that $ u_i(0,x_0)=0 $ for $ i=1,2.$ Then from equation \eqref{2.2} we obtain 
	\[ u_1(t,\Pi_{\mathbf{v}}(t,0;x_0))=0= u_2(t,\Pi_{\mathbf{v}}(t,0;x_0)),\;\forall t>0. \]
	Therefore, the characteristics $t\mapsto \Pi_{\mathbf{v}}(t,0;x_0) $ forms a segregation barrier for the two cell populations.
	\end{remark}

	\begin{figure}[H]
		\begin{center}
			\includegraphics[width=0.6\textwidth]{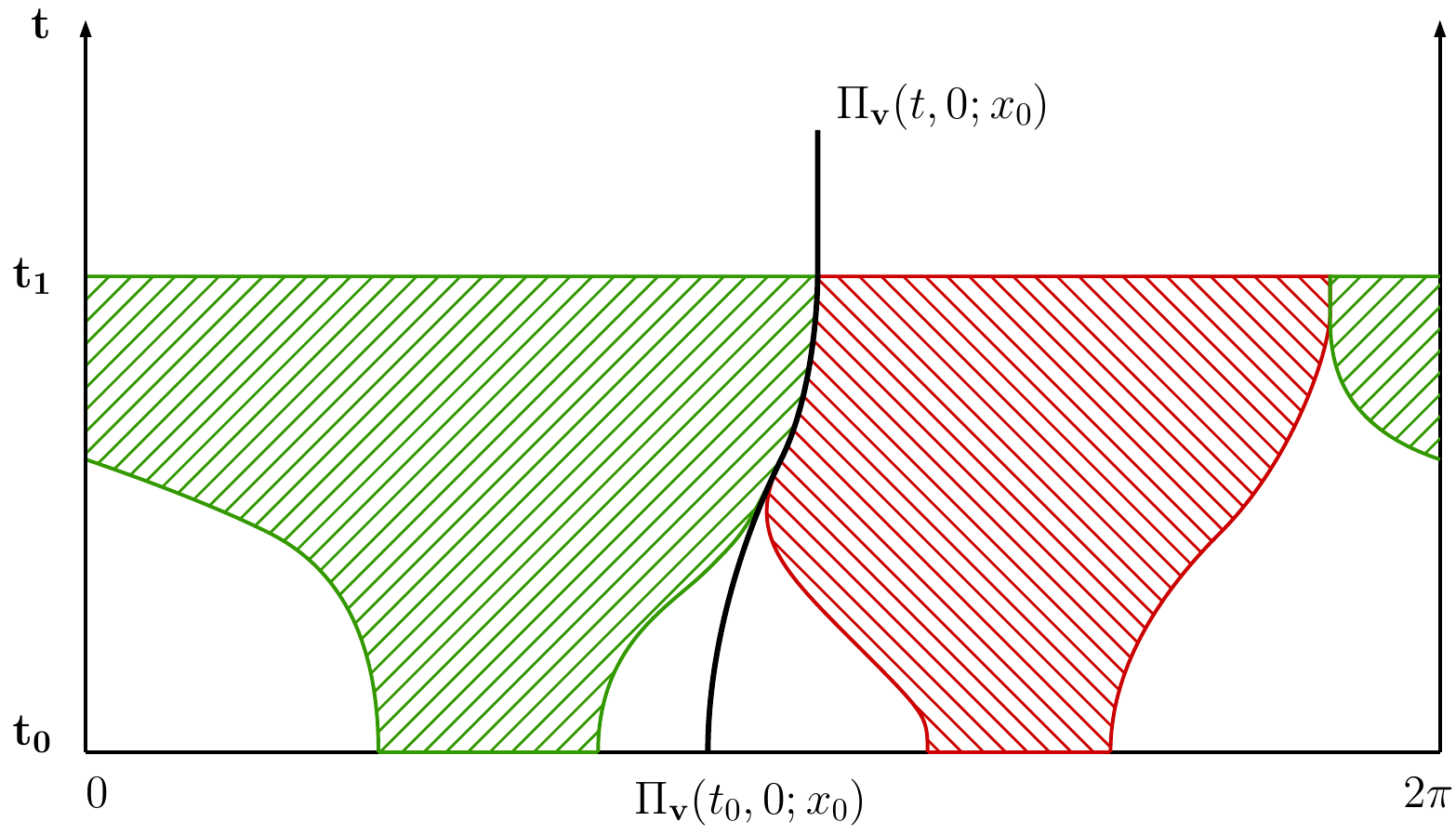}
		\end{center}
		\caption{\textit{The shaded areas represent the supports of two populations (red and green) evolving along the time. Notice that if one starts with two separated supports and choose $ x_0 $ where $ u_i(0,x_0)=0$ for $i=1,2, $ then the characteristics $t\mapsto \Pi_{\mathbf{v}}(t,0;x_0) $ forms a segregation ``wall'' between the two cell populations, which indicates no matter how close they are, they will keep separated. }}
		\label{FIG1}
	\end{figure}

\section{Asymptotic Behavior}
In the rest of the paper, we always assume the initial distributions for the two populations are separated.
\begin{assumption}\label{ASS4.1}
	For initial value $	\mathbf u_0 \in  L_{per ,+}^{\infty }\left( \mathbb{R}^{N}\right)^2  $, we assume 
	$$ 
	u_1(0,x)u_2(0,x)=0 , \forall x\in \mathbb{T}^N. 
	$$
	Furthermore, we suppose $ h_i $ in equation \eqref{1.1} has the following form
	\[ h_i(u_1,u_2)=h_i(u_1+u_2),\,i=1,2,  \]
	with $ h_i(r_i)=0 $ for some $ r_i>0,\,i=1,2, $ and
	\[ h_i(u)>0,\,\forall u\in [0,r_i),\quad h_i(u)<0,\, \forall u>r_i,\quad \limsup_{u\to \infty} h_i(u) <0,\, i=1,2.  \]
	Moreover, $u\longmapsto uh_i(u) $ is a concave function for $ i=1,2. $ 
\end{assumption}
%	\begin{remark}
%		Note that $ h_i(u)=\lambda_i(1-u_i) $ and $ h_i(u)=\dfrac{b_i}{1+\gamma_i u}-\mu_i $ for $ i=1,2  $ satisfy both Assumption \ref{ASS1.1} and \ref{ASS4.1}.
%	\end{remark}
\begin{remark}
	Notice that segregation property in Theorem \ref{THM3.1} implies the following equality:
	\begin{equation}\label{4.1}
		 u_i(t,x)h_i(u_1(t,x)+u_2(t,x))=u_i(t,x)h_i(u_i(t,x)),\,i=1,2,\,\forall (t,x)\in [0,\infty)\times \T^N .
	\end{equation}
\end{remark}
\begin{lemma}\label{LEM4.3}
	Let Assumptions \ref{ASS1.1}, \ref{ASS1.2} and \ref{ASS4.1} be satisfied. Suppose $ \mathbf{u}=\mathbf{u}(t,x) $ is the solution of \eqref{1.1}-\eqref{1.3}. Then we have 
	\begin{itemize}
		\item[\rm{(i)}] $ \sup_{t\geq 0}\| u_i(t,\cdot) \|_{L^1} \leq \max\lbrace \| u_i(0,\cdot)\|_{L^1}, |\mathbb T^N|  \rbrace ,\,i=1,2. $
		\item[\rm{(ii)}] $ \mathbf{v}(t,x):= (\nabla K \circ\, (u_1+u_2)(t,\cdot))(x)\text{ satisfies } \mathbf{v}\in L^{\infty}((0,\infty),W_{per}^{1,\infty}(\R^N))^N $ and 
		\[ \| \mathbf{v}(t,\cdot) \|_{C^1} \leq 2 \| K \|_{C^2} \max\lbrace \| u_1(0,\cdot)\|_{L^1}, \|u_2(0,\cdot)\|_{L^1}, |\mathbb T^N|  \rbrace.  \]
	\end{itemize}
\end{lemma}
\begin{proof}
	By the equation \eqref{4.1} due to segregation, the equation \eqref{1.1} can be rewritten as
	\begin{equation}\label{4.2}
	\partial _{t}u_i +\mathrm{div }\bigl(u_i\mathbf{v}\bigr)=
	u_ih_i(u_i),\;i=1,2.
	\end{equation}
	By Assumption \ref{ASS4.1} the function $ f_i(u)=uh_i(u) $ is concave for each $ i $, integrating \eqref{4.2} over $ \T^N $ and using Jensen's inequality, we have for classical solution
	\[ \frac{\di }{\di t} \|u_i(t,\cdot)\|_{1}=\|f(u_i(t,\cdot)) \|_{1}\leq f\left(\|u_i(t,\cdot)\|_{L^1}\right). \]
	Then the results follows using the usual ordinary differential arguments with Assumption \ref{ASS4.1}, where we can prove 
	\[ \sup_{t\geq 0}\| u_i(t,\cdot) \|_{L^1}\leq \max\lbrace \| u_i(0,\cdot)\|_{L^1}, |\mathbb T^N|  \rbrace, \; i=1,2.  \]
	Let $ \mathbf{u}_0\in L_{per ,+}^{\infty }\left( \mathbb{R}^{N}\right)^2 $ be given and $ \mathbf{u}  $ be the corresponding solution integrated along the characteristics. Consider a sequence $ \lbrace   \mathbf{u}_0^n\rbrace_{n\geq 0}  $ in $ C_{per,+}^1(\R^N)^2 $ such that $ \|  \mathbf{u}_0^n-\mathbf{u}_0 \|_{L^1} \rightarrow 0 $ as $ n\to+\infty $. Then denote $ \mathbf{u}^n $ the solutions corresponding to $ \mathbf{u}_0^n $, from Theorem \ref{THM2.3} we have $ \|  \mathbf{u}^n(t,\cdot)-\mathbf{u}(t,\cdot) \|_{L^1}\to 0 $ and $ \mathbf u(t,\cdot)\in L_{per ,+}^{\infty }\left( \mathbb{R}^{N}\right)^2 $. Therefore, by using the Lebesgue convergence theorem, the result (i) follows. Then result (ii) is a direct consequence of (i).
\end{proof}

\subsection{Energy Functional}
In order to prove that our energy functional is decreasing we will make the following assumption. 
\begin{assumption}
	\label{ASS4.4} The Fourier's coefficients of function $K$ on $\mathbb{T}^N$
		denoted by $\left\{c_n[K]\right\}_{n\in\mathbb{Z}^N}$ satisfy $%
		c_n[K]>0,\;\;\forall n\in\mathbb{Z}^N\setminus\{0\}$. Here the Fourier
		coefficients are defined by 
		\begin{equation*}
			c_n[K]=|\mathbb{T}^N|^{-1}\int_{\mathbb{T}^N} e^{-in\cdot x}K(x) \di  x,\;\;\forall
			n\in \mathbb{Z}^N.
		\end{equation*}

\end{assumption}
\begin{remark}\label{REM4.5}
	If $ \rho $ in system \eqref{1.1} satisfies that the Fourier transformation $ \widehat{\rho}(\xi)>0 $ for all $ \xi \in \R^N$, then for the kernel $ K $, we have $ c_n[K]>0 $ for all $ n\in \Z^N $. This implies that Assumption \ref{ASS4.4} is satisfied. 
\end{remark}
We construct the functional for $ u_i,i=1,2, $ as
\[ E_i[u_i(t,\cdot)]=\frac{1}{|\mathbb{T}^N|}\int_{\mathbb{T}^N}G_i(u_i(t,x))\di x, \] 
where $ G_i:[0,\infty)\to [0,\infty) $ is defined by 
\begin{equation}\label{4.3.0}
G_i(u):=u\ln \left(\frac{u}{r_i}\right)-u+r_i.
\end{equation}
Notice that $ G_i'(u)=\ln(u/r_i) $ for $ u>0 $ and we define the energy functional as
\begin{equation}\label{4.3}
E[(u_1,u_2)(t,\cdot)]:=\sum_{i=1,2}E_i[u_i(t,\cdot)].
\end{equation}
\begin{theorem}\label{THM4.6}
Let Assumptions \ref{ASS1.1}, \ref{ASS1.2}, \ref{ASS4.1} and \ref{ASS4.4} be satisfied. Suppose $ \mathbf{u}=\mathbf{u}(t,x) $ is the solution of \eqref{1.1}-\eqref{1.3}. Then for any $ t,\tau>0 $ set $ u:=u_1+u_2 $ we have 
	\begin{equation}\label{4.4}
	\begin{aligned}
	&E[(u_1,u_2)(t+\tau,\cdot)]-E[(u_1,u_2)(t,\cdot)]\\
	&=-\int_{t}^{t+\tau}\sum_{k\in\Z^N}|k|^2 c_k[K] \left|c_k[u(s,\cdot)]\right|^2\di s-\frac{1}{|\mathbb{T}^N|}\int_{t}^{t+\tau}\int_{\mathbb{T}^N}\sum_{i=1,2}u_i\left|h_i(u_i) \ln \left(\frac{u_i}{r_i}\right)\right|\di x\di s.
	\end{aligned}
	\end{equation}
\end{theorem}
\begin{proof}
For any $ \delta>0 $, as before we first suppose $ \mathbf{u}=(u_1,u_2) $ to be the classical solution. Setting $ u=u_1+u_2\geq 0 $, recall the property in \eqref{4.1} we have
\begin{equation*}
	\begin{aligned}
	&\frac{\di }{\di t} E_i[(u_i+\delta)(t,\cdot)] =\frac{1}{|\mathbb{T}^N|}\int_{\mathbb{T}^N}\ln \left(\frac{u_i+\delta}{r_i}\right)\partial_tu_i\di x\\
	=&\frac{1}{|\mathbb{T}^N|}\int_{\mathbb{T}^N}\ln \left(\frac{u_i+\delta}{r_i}\right)\Big[\mathrm{div}\;\left[u_i\nabla (K\circ\, u)\right]+u_ih_i(u_i)\Big]\di x\\
	=&\frac{1}{|\mathbb{T}^N|}\int_{\mathbb{T}^N}\frac{u_i^2}{u_i+\delta}\Delta( K\circ\, u)+u_i\nabla K\circ u\cdot\nabla \left(\frac{u_i}{u_i+\delta}\right) \di x+\frac{1}{|\mathbb{T}^N|}\int_{\mathbb{T}^N}u_ih_i(u_i)\ln \left(\frac{u_i+\delta}{r_i}\right)\di x.
	\end{aligned}
\end{equation*}
Therefore, for any $ t,\tau>0 $ we obtain
\begin{equation*}
\begin{aligned}
	&E_i[(u_i+\delta)(t+\tau,\cdot)]-E_i[(u_i+\delta)(t,\cdot)]\\
	=&\frac{1}{|\mathbb{T}^N|}\int_{t}^{t+\tau}\int_{\mathbb{T}^N}\frac{u_i^2}{u_i+\delta}\Delta( K\circ\, u)+u_i\nabla K\circ u\cdot\nabla \left(\frac{u_i}{u_i+\delta}\right) \di x\di s+\frac{1}{|\mathbb{T}^N|}\int_{t}^{t+\tau}\int_{\mathbb{T}^N}u_ih_i(u_i) \ln \left(\frac{u_i+\delta}{r_i}\right)\di x\di s.
\end{aligned}
\end{equation*}
Now by letting $ \delta\to 0 $ we see that
\begin{equation*}
\begin{aligned}
&E_i[u_i(t+\tau,\cdot)]-E_i[u_i(t,\cdot)]\\
=&\frac{1}{|\mathbb{T}^N|}\int_{t}^{t+\tau}\int_{\mathbb{T}^N}u_i\Delta( K\circ\, u) \di x\di s+\frac{1}{|\mathbb{T}^N|}\int_{t}^{t+\tau}\int_{\mathbb{T}^N}u_ih_i(u_i) \ln \left(\frac{u_i}{r_i}\right)\di x\di s.
\end{aligned}
\end{equation*}
By summing the two functional $ E_i,i=1,2, $ we obtain
\begin{equation*}
\begin{aligned}
&E[(u_1,u_2)(t+\tau,\cdot)]-E[(u_1,u_2)(t,\cdot)]\\
=&\frac{1}{|\mathbb{T}^N|}\int_{t}^{t+\tau}\int_{\mathbb{T}^N}u\Delta( K\circ\, u) \di x\di s+\frac{1}{|\mathbb{T}^N|}\int_{t}^{t+\tau}\int_{\mathbb{T}^N}\sum_{i=1,2}u_ih_i(u_i) \ln \left(\frac{u_i}{r_i}\right)\di x\di s.
\end{aligned}
\end{equation*}
On the other hand, for each $ \phi \in L_{per}^2(\R^N) $ one has $ \phi(x)=\sum_{k\in\Z^N}\overline{c_k[\phi]}e^{i n\cdot x} $ almost everywhere, thus
\[ \frac{1}{|\mathbb{T}^N|}\int_{\mathbb{T}^N}\phi \Delta( K\circ\,\, \phi) \di x= \sum_{k\in\Z^N} \frac{1}{|\mathbb{T}^N|}\int_{\mathbb{T}^N}\overline{c_k[\phi]}e^{i n\cdot x} \Delta( K\circ\,\, \phi) \di x=\sum_{k\in\Z^N}\overline{c_k[\phi]} c_k[\Delta K\circ\,\, \phi]=-\sum_{k\in\Z^N}|k|^2 c_k[K] c_k[\phi]^2. \]
Therefore, by the above calculation and by the fact that $ h_i(u)\ln(u/r_i)<0 ,i=1,2 $, we have 
\begin{equation*}
\begin{aligned}
&E[u(t+\tau,\cdot)]-E[u(t,\cdot)]\\
&=-\int_{t}^{t+\tau}\sum_{k\in\Z^N}|k|^2 c_k[K] \left|c_k[u(s,\cdot)]\right|^2\di s-\frac{1}{|\mathbb{T}^N|}\int_{t}^{t+\tau}\int_{\mathbb{T}^N}\sum_{i=1,2}u_i\left|h_i(u_i) \ln \left(\frac{u_i}{r_i}\right)\right|\di x\di s.
\end{aligned}
\end{equation*}
 The usual limiting procedure as in Lemma \ref{LEM4.3} allows us to extend the estimation to solutions integrated along the characteristics.
\end{proof}
\begin{remark}
	By the above theorem, we can see the energy functional $ E $ is non-negative and is decreasing along the trajectories of $ \eqref{1.1} $, by letting $ t\to+\infty  $ we deduce from \eqref{4.4} that
	\begin{equation}\label{4.5}
	\lim_{t\to+\infty}\int_{t}^{t+\tau}\sum_{k\in\Z^N}|k|^2 c_k[K] \left|c_k[u(s,\cdot)]\right|^2\di s=0,
	\end{equation}
	and
	\begin{equation}\label{4.6}
	\lim_{t\to+\infty}\int_{t}^{t+\tau}\int_{\mathbb{T}^N}u_i\left|h_i(u_i) \ln \left(\frac{u_i}{r_i}\right)\right|\di x\di s=0,\;i=1,2.
	\end{equation}	
\end{remark}

Before we prove the $ L^\infty $ boundedness of the solution  for all $ t \geq 0 $, i.e., $ \sup_{t\geq 0}\| u_i(t,\cdot)\|_{L^\infty} <\infty, $ for $ i=1,2, $
we need following lemmas.
\begin{lemma} \label{LEM4.8}
	Let Assumptions \ref{ASS1.1}, \ref{ASS1.2}, \ref{ASS4.1} and \ref{ASS4.4} be satisfied. Suppose $ \mathbf{u}=\mathbf{u}(t,x) $ is the solution of \eqref{1.1}-\eqref{1.3}. Then for any $ k\in \Z^N $ and for each $ i=1,2, $ the mapping
	\[t\longmapsto c_k[u_i(t,\cdot)]  \]
	is a $ C^1 $ function. Here $ c_k[u_i(t,\cdot)],\,k\in \Z^N $ are the Fourier coefficients. Moreover,
	\[ \sup_{t\geq 0}\left|\frac{\di }{\di t}c_k[u_i(t,\cdot)]  \right|<\infty . \]
\end{lemma}
\begin{proof}
	For any $ k\in \Z^N $, suppose $ \mathbf{u}=(u_1,u_2) $ is the classical solution then we have
	\begin{align*}
	\frac{\di }{\di  t} c_k[u_i(t,\cdot)]&=\frac{1}{|\T^N|}\int_{\T^N} e^{-ik\cdot x}\left[- \mathrm{div}\,(u_i \mathbf{v})+u_ih_i(u_i)\right]\di x\\
	&=\frac{1}{|\T^N|}\int_{\T^N} u_i \nabla\left(e^{-ik\cdot x}\right)\cdot\mathbf{v}+ e^{-ik\cdot x}u_ih_i(u_i)\di x.
	\end{align*}
	Therefore, by using the Jensen's inequality for $ f_i(u)=uh_i(u) $ again, we derive
	\begin{equation*}
	\begin{aligned}
	\left| \frac{\di }{\di  t} c_k[u_i(t,\cdot)] \right|\leq |k|\|u_i(t,\cdot) \|_{1} \|\mathbf{v}(t,\cdot)\|_{C^0}+f(\|u_i(t,\cdot)\|_{L^1}),
	\end{aligned}
	\end{equation*}
	the result follows by Lemma \ref{LEM4.3}. The case of the solution integrated along the characteristics can be proved by applying a classical regularization procedure.
\end{proof}

The regularity condition for kernel $ K $ defined in Assumption \ref{ASS1.2} serves mainly for the following result. 
\begin{lemma}\label{LEM4.9} Let Assumptions \ref{ASS1.1}, \ref{ASS1.2}, \ref{ASS4.1} and \ref{ASS4.4} be satisfied. Suppose $ \mathbf{u}=\mathbf{u}(t,x) $ is the solution of \eqref{1.1}-\eqref{1.3} and define $ u:=u_1+u_2 $. Then for $ \mathbf{v}(t,x)= \left(\nabla K\circ\, u(t,\cdot)\right)(x)$ we have
	\[ \lim_{t\to+\infty} \| \mathrm{div}\, \mathbf{v}(t,\cdot)\|_{C^0} =0.\]
\end{lemma}
\begin{proof}
		By Assumption \ref{ASS1.2} $ K\in C^m_{per}(\R^N) $ with $ m\geq \frac{N+5}{2} $, therefore by Temam \cite[page. 50]{Temam1997}
		\begin{equation}\label{4.7}
		\sum_{k\in\Z^N}(1+|k|^2)^s c_k[K]^2<\infty. 
		\end{equation}
		Moreover, we can deduce from \eqref{4.5} that for each $ k\in \Z^N \backslash \lbrace 0 \rbrace $
		\[ \lim_{t\to+\infty}\int_{t}^{t+\tau} \left|c_k[u(s,\cdot)]\right|^2\di s=\lim_{t\to+\infty}\int_{0}^{\tau} |c_k[u(s+t,\cdot)]|^2\di s=0. \]
		The last equality together with the results in Lemma \ref{LEM4.8}, we can deduce
		\begin{equation}\label{4.8}
		\lim_{t\to+\infty} c_k[u(t,\cdot)]=0,\quad  k\in \Z^N \backslash \lbrace 0 \rbrace.
		\end{equation}
		Since for any finite $ t $, $ u(t,\cdot)$ is $ L_{per}^\infty $ bounded by \eqref{2.23}, thus $ L_{per}^2 $ bounded, we deduce
		\begin{equation*}
		\begin{aligned}
		\mathrm{div}\,\mathbf{v}(t,x)=&-\frac{1}{|\T^N|} \int_{\T^N} \Delta  K(x-y)u(t,y)\di y\\
		=&-\frac{1}{|\T^N|} \int_{\T^N} \Delta  K(x-y)\sum_{k\in\Z^N} e^{-ik\cdot y}c_k[u(t,\cdot)] \di y\\
		=&-\frac{1}{|\T^N|} \int_{\T^N} \sum_{k\in\Z^N}\Delta  K(z)e^{ik\cdot (z-x)} c_k[u(t,\cdot)] \di z\\
		=&\sum_{k\in \Z^N}|k|^2c_k[K]c_k[u(t,\cdot)]e^{-ik\cdot x},
		\end{aligned}
		\end{equation*}
		where the last series converges due to \eqref{4.7}. In fact, by Lemma \ref{LEM4.3}, we can find a constant $M>0$ such that for each $k \in  \Z^N$ we have 
		$$
		|c_k[u(t,\cdot)]| < \|u(t,\cdot) \|_{L^1} \leq M,\, \forall t \geq 0. 
		$$
		Therefore,
		\begin{equation*}
		\begin{aligned}
		\|\mathrm{div}\,\mathbf{v}(t,x)\|_{C^0} =&\left\|\sum_{k\in \Z^N}|k|^{2}c_k[K]c_k[u(t,\cdot)]e^{-ik\cdot x}\right\|_{C^0} \\
		\leq& M \sum_{k\in \Z^N}|k|^2c_k[K]=M \sum_{k\in \Z^N \backslash \lbrace 0\rbrace}|k|^{-\frac{N+1}{2}}|k|^{2+\frac{N+1}{2}}c_k[K]\\
		\leq& M \left(\sum_{k\in \Z^N \backslash \lbrace 0\rbrace} \frac{1}{|k|^{N+1}}\right)^\frac{1}{2} \left(\sum_{k\in \Z^N \backslash \lbrace 0\rbrace} |k|^{N+5}c_k[K]^2\right)^\frac{1}{2}, 
		\end{aligned}
		\end{equation*}
		and due to  \eqref{4.7}, this last series converges. Therefore, by Lebesgue dominated convergence theorem we have
		\[ \limsup_{t \to + \infty}	 \|\mathrm{div}\,\mathbf{v}(t,x)\|_{C^0}\leq \limsup_{t \to + \infty} \sum_{k\in \Z^N}|k|^{2}c_k[K]|c_k[u(t,\cdot)]|=0. \]
\end{proof}

As a consequence of Lemma \ref{LEM4.9}, we obtain Theorem \ref{THM4.11} and Corollary \ref{COR4.12} which are the main results of this section.
\begin{theorem}\label{THM4.11}
	 Let Assumptions \ref{ASS1.1}, \ref{ASS1.2}, \ref{ASS4.1} and \ref{ASS4.4} be satisfied. Suppose $ \mathbf{u}=\mathbf{u}(t,x) $ is the solution of \eqref{1.1}-\eqref{1.3}. Then we have for each $ i=1,2, $ 
	\begin{equation*}
	\sup_{t\geq 0}\| u_i(t,\cdot)\|_{L^\infty} <+\infty,
	\end{equation*}
	and more precisely we have
	$$ 
	\limsup_{t\to+\infty} \|u_i(t,\cdot)\|_{L^\infty} \leq r_i.
	$$
	Moreover, for any $ x\in \R^N $ such that $ u_i(0,x)>0 $. Then the solution evolving along the characteristics converges point-wisely to the positive equilibrium $ r_i $ for $ i=1,2. $ That is, for any $ x\in \mathcal{U}_i $ where $  \mathcal{U}_i=\lbrace x\in \R^N : u_i(0,x)>0 \rbrace$
	\[ \lim_{t\to \infty} u_i(t,\Pi _{\mathbf{v}}(t,0;x)) = r_i. \]
	Or equivalently, for any $ x\in \R^N $ we have
	\[ u_i(t,\Pi _{\mathbf{v}}(t,0;x))\xrightarrow{p.w.} r_i \mathbbm{1}_{\mathcal{U}_i}(x),\;\text{ as  }t\to \infty. \] 
\end{theorem}

\begin{remark}\label{REM4.10}
	Notice from the above theorem, we obtain the following $ L^2 $ uniform boundedness of the solution $ u=u_1+u_2 $, that is
	\[ \sup_{t\geq 0}\| u(t,\cdot) \|_{L^2}<\infty. \]
Moreover for any sequence $ \lbrace t_n\rbrace_{n\geq 0} $ which tends to infinity, since for Fourier coefficients, one has
	\[ \lim_{n\to\infty}c_k[u(t_n,\cdot)]=0,\quad \forall k\in \Z^N \backslash \lbrace 0 \rbrace, \]
	therefore Banach-Alaoglu-Bourbaki theorem together with the fact that $L^2$ is an Hilbert space we deduce that there exists a subsequence $ \lbrace t_{n_l} \rbrace_{l\geq 0}  $ such that
\[ 
u(t_{n_l},\cdot) \rightharpoonup c \text{ in } L^2 
\]
where $c$ stands for a constant distribution which depends on the subsequence and the convergence is the weak convergence in $ L^2 $. With the above argument we can deduce 
\begin{equation}\label{4.10.0}
\lim_{t\to \infty} \|\mathbf{v}(t,.) \|_{C^0}=0.
\end{equation}
In fact, for any sequence $ \lbrace t_n\rbrace_{n\geq 0} $ with $ t_n\to \infty $ as $ n\to \infty $,  we can find a subsequence such that
\[ \mathbf{v}(t_{n_l},x)=\int_{\T^N}\nabla K(x-y) u(t_{n_l},y)dy \to c\int_{\T^N}\nabla K(x-y)dy=0  \]
where the last equation is by the periodic boundary condition. Thus, equation \eqref{4.10.0} follows.
\end{remark}
\begin{proof}
	Suppose $ \mathbf{u}=(u_1,u_2) $ the classical solution. The usual limiting procedure allows us to extend the estimation to solutions integrated along the characteristics. We recall the notation in \eqref{2.5} where $ w_i(t,x):=u_i(t,\Pi _{\mathbf{v}}(t,0;x)),\,i=1,2, $ and for any $ x\in \R^N $ we have
	\begin{equation*}
	\begin{aligned}
	\frac{\di  w_i(t,x)}{\di t}=&w_i(t,x) \left[-\mathrm{%
		div }\; \mathbf{v}(t,\Pi _{\mathbf{v}}(t,0;x)) +h_i((w_1+w_2)(t,x)) \right]\\
	=&w_i(t,x) \left[-\mathrm{%
		div }\; \mathbf{v}(t,\Pi _{\mathbf{v}}(t,0;x)) +h_i(w_i(t,x)) \right],
	\end{aligned}
	\end{equation*}
	where the second equation is due to segregation. Now we apply the comparison principle to the solution along the characteristics with the solution of the following ordinary differential equation. For any $ \tau>0 $, let $ \overline{w}_i(t) $ to be the solution of the following Cauchy problem
	\begin{equation*}
	\begin{cases}
	\dfrac{\di \overline{w}_i(t)}{\di t}&=\overline{w}_i(t)\left[\displaystyle\sup_{t\geq \tau}\|\mathrm{div }\; \mathbf{v}(t,\cdot)\|_{C^0}+h_i(\overline{w}_i(t)) \right],\quad t>\tau\\
	\overline{w}_i(\tau)&=\|w_i(\tau,\cdot)\|_{L^\infty}.
	\end{cases}
	\end{equation*}
	Then we see that for any $ \tau>0 $,
	\[ \limsup_{t\to+\infty} \overline{w}_i(t)\leq \overline{\Phi}_i(\tau):=\inf\lbrace z> r_i: \sup_{t\geq \tau}\|\mathrm{div }\; \mathbf{v}(t,\cdot)\|_{C^0}+h_i(y)\leq  0,\;\forall y\geq z \rbrace. \]
	If the set is empty, then $ \overline{\Phi}_i(\tau)=+\infty $.
	By comparison principle, for any $ \tau>0 $  we have
	\[ \limsup_{t\to+\infty} \|w_i(t,\cdot)\|_{L^\infty} \leq \limsup_{t\to+\infty} \overline{w}_i(t)\leq \overline{\Phi}_i(\tau), \]
	while due to Assumption \ref{ASS4.1} where $ h_i(u)<0 $ for any $ u>r_i $ and $ \limsup_{u\to \infty}h_i(u)<0 $ and Lemma \ref{LEM4.9}, we have $ \lim_{\tau\to+\infty}\overline{\Phi}_i(\tau)=r_i $ thus we have
	\begin{equation}\label{4.9}
	\limsup_{t\to+\infty}\,\| u_i(t,\Pi _{\mathbf{v}}(t,0;\cdot))\|_{L^\infty}\leq r_i.
	\end{equation}
	Since $ x\mapsto  \Pi _{\mathbf{v}}(t,0;x)$ is invertible on $ \R^N $, we have
	\[ \limsup_{t\to+\infty} \|u_i(t,\cdot) \|_{L^\infty}\leq r_i. \]
	Together with the $ L^\infty $ estimation of $ \mathbf u $ in finite time in \eqref{2.23} , we see that 
	\begin{equation*}
	\sup_{t\geq 0}\| u_i(t,\cdot)\|_{L^\infty} <\infty.
	\end{equation*}
	Now we prove the second part of the theorem. For any fixed $ x\in \R^N $ with $ u_i(0,x)>0 $, we can see from the definition of the solution integrated along the characteristics \eqref{2.9} that
	\[ w_i(t,x)=u_i(t,\Pi _{\mathbf{v}}(t,0;x))>0,\;\forall t>0. \]
	For any $ \tau>0 $, define the solution $ \underline{w}_i(t) $ to be the solution of the following Cauchy problem 
	\begin{equation*}
	\begin{cases}
	\dfrac{\di  \underline{w}_i(t)}{\di t}&=\underline{w}_i(t)\left[-\displaystyle\sup_{t\geq \tau}\|\mathrm{div }\; \mathbf{v}(t,\cdot)\|_{C^0}+h_i(\underline{w}_i(t)) \right]\\
	\underline{w}_i(\tau)&=w_i(\tau,x)>0.
	\end{cases}
	\end{equation*}
	Then we see that for any $ \tau>0 $,
	\[ \liminf_{t\to+\infty} \underline{w}_i(t)\geq \underline{\Phi}_i(\tau):=\sup\lbrace z>0: -\sup_{t\geq \tau}\|\mathrm{div }\; \mathbf{v}(t,\cdot)\|_{C^0}+h_i(y)\geq 0,\forall y\leq z\rbrace. \]
	If the set is empty, then $ \underline{\Phi}_i(\tau)=-\infty $. As before we use the comparison principle, for any $ \tau>0 $ and any $ x\in \lbrace x\in \R^N : u_i(0,x)>0 \rbrace $ we have
	\[ \liminf_{t\to+\infty} w_i(t,x) \geq \liminf_{t\to+\infty} \underline{w}_i(t)\geq \underline{\Phi}_i(\tau). \]
	Due to Assumption \ref{ASS4.1} where $ h_i(u)>0 $ for any $ u\in [0,r_i) $, we have $ \lim_{\tau\to+\infty}\underline{\Phi}_i(\tau)=r_i $ thus we have for any $ x\in \lbrace x\in \R^N : u_i(0,x)>0 \rbrace $,
	\[ \liminf_{t\to+\infty}\, u_i(t,\Pi _{\mathbf{v}}(t,0;x))\geq r_i, \]
	together with \eqref{4.9} the result follows.
\end{proof}
Next corollary is a consequence of Theorem \ref{THM4.11}.
\begin{corollary}\label{COR4.12}
	Let Assumptions \ref{ASS1.1}, \ref{ASS1.2}, \ref{ASS4.1} and \ref{ASS4.4} be satisfied. Suppose $ \mathbf{u}=\mathbf{u}(t,x) $ is the solution of \eqref{1.1}-\eqref{1.3}. If for some constant $ \delta >0$, $ u(0,x)=\sum_{i=1,2}u_i(0,x)\geq \delta>0 $ for $a.e.\, x\in \T^N  $. Moreover, we assume $ r_1=r_2=:r $ in Assumption \ref{ASS4.1}. Then
	\[  \lim_{t\to \infty}\|u(t,\cdot)-r\|_{L^\infty}=0. \]
\end{corollary}
\begin{proof}
	Here again we only prove the convergence when $ \mathbf{u}=(u_1,u_2) $ is the classical solution. We use the same notations as in Theorem \ref{THM4.11} and define
	\begin{equation*}
	w(t,x):=w_1(t,x)+w_2(t,x).
	\end{equation*}
	Due to estimation \eqref{4.9} in Theorem \ref{THM4.11} and segregation property, we have
	\begin{equation}\label{4.10}
	\limsup_{t\to+\infty} \sup_{x\in \R^N}w(t,x) \leq r.
	\end{equation}
	Moreover, we can obtain
	\begin{equation*}
	\frac{\di w(t,x)}{\di t}=-w(t,x) \mathrm{div }\; \mathbf{v}(t,\Pi _{\mathbf{v}}(t,0;x)) +\sum_{i=1}^{2}w_ih_i(w_i).
	\end{equation*}
%	Since the initial distributions are separated by Assumption \ref{ASS4.1}, we deduce $ \mathcal{U}_1\cap \mathcal{U}_2=\varnothing $ and $ \R^N=\mathcal{U}_1\cup \mathcal{U}_2 $.  
	In order to use comparison principle, we set $ \underline{h}(u)=\min_{u\geq 0} \lbrace h_1(u),h_2(u) \rbrace $ and
	by the separation property in Theorem \ref{THM3.1} we have
	\begin{align*}
	w_1h_1(w_1)+w_2h(w_2)&\geq w_1 \underline{h}(w_1)+w_2 \underline{h}(w_2)=(w_1 +w_2) \underline{h}(w_1+w_2).
	\end{align*}
	Hence,
	\begin{equation*}
	\frac{\di  w(t,x)}{\di t}\geq w(t,x)\left[-\sup_{t\geq \tau}\|\mathrm{div }\; \mathbf{v}(t,\cdot)\|_{C^0} +\underline{h}(w(t,x))\right],\;t\geq \tau.
	\end{equation*}
	For any $ \tau>0 $, we have
	$ \inf_{x\in \R^N}w(\tau,x)>0 $.  In fact, by our assumption, $ u(0,x)\geq \delta >0$ on $ \T^N $, thus $ u(0,x)\geq \delta >0$ on $ \R^N $ and by equation \eqref{2.9} we have $ w(\tau,x)>0 $ for any $ x\in \R^N $  and since $ w(t,x+2\pi)=w(t,x) $ for any $ x\in \R^N $, we have $ \inf_{x\in \R^N}w(\tau,x)\geq \tilde{\delta}>0 $ for some positive $ \tilde{\delta} $.
	Thus, for any $ \tau>0 $,  we define $ \underline{w}(t) $ to be the solution of the following ordinary differential equation
	 \begin{equation*}
	 \begin{cases}
	 \dfrac{\di  \underline{w}(t)}{\di t}&=\underline{w}(t)\left[-\displaystyle\sup_{t\geq \tau}\|\mathrm{div }\; \mathbf{v}(t,\cdot)\|_{C^0}+\underline{h}(\underline{w}(t)) \right],\\
	 \underline{w}(\tau)&=\inf_{x\in \R^N}w(\tau,x)>0.
	 \end{cases}
	 \end{equation*}
	 By the similar argument as in Theorem \ref{THM4.11}, we can see that
	 \[ \liminf_{t\to+\infty} \inf_{x\in \R^N}w(t,x) \geq \liminf_{t\to+\infty} \underline{w}(t)\geq r. \]
	 Together with \eqref{4.10}, we have 
	 \[ \lim_{t\to \infty}\|w(t,\cdot)-r\|_{L^\infty}=0. \]
	 Since for any $ t>0 $, the mapping $t\mapsto \Pi _{\mathbf{v}}(t,0;\cdot) $ is a bijection, we have
	 \[ \|w(t,\cdot)-r\|_{L^\infty}=\|u(t,\Pi _{\mathbf{v}}(t,0;\cdot))-r\|_{L^\infty}=\|u(t,\cdot)-r \|_{L^\infty}. \]
	 Thus, we obtain \[ \lim_{t\to \infty}\|u(t,\cdot)-r\|_{L^\infty}=0. \]
\end{proof}
\begin{remark}
	Note that the result in the corollary, we only assume the roots of two different reaction functions $ h_1,h_2 $ to be the same while we obtain convergence in $ L^\infty $. 
\end{remark}

\section{Young Measure}
In order to introduce the notion of Young measures we introduce it informally first. The basic idea of Young measure is to replace the map $(t,x) \to u(t,x)=u_1(t,x)+u_2(t,x)$ by the map 
$$
(t,x) \to \delta_{u(t,x)}
$$
from $\left[0, \infty \right) \times \T^N $ into a probability space. Namely, for fixed $t$ and $x$,  the Dirac mass $\delta_{u(t,x)}$ is regarded as an element of the dual space the continuous functions $C([0, \gamma], \mathbb{R})$ (where $ \displaystyle \gamma:=\Vert u \Vert_{L^\infty(\left[0, \infty \right) \times \T^N)}$) by using the following mapping
$$
f \longmapsto \int_{[0, \gamma]} f(\lambda)\delta_{u(t,x)}(\di \lambda)=  f(u(t,x)).
$$
This means that the map  $(t,x) \to \delta_{u(t,x)}$ is identified to an element of $L^\infty \left( \left[0, \infty \right) \times \T^N, C([0, \gamma], \mathbb{R})^\star \right) $.  The goal of this procedure is to use the weak star topology, by considering Young measure as an element the dual space of
$$
L^1 \left( \left[0, \infty \right) \times \T^N, C([0, \gamma], \mathbb{R}) \right). 
$$
The space of Young measures in our specific context is nothing but  
$L^\infty \left( \left[0, \infty \right) \times \T^N, \mathbb{P}\left(\left[0, \gamma \right] \right) \right)$ (where $\mathbb{P}\left(\left[0, \gamma \right] \right)$ is the space of probabilities on $\left[0, \gamma \right]$) endowed with the weak star topology. 

In Corollary \ref{COR4.12}, we have the $ L^\infty $ convergence of the solution $ u(=u_1+u_2) $ when the initial distribution is strictly positive. Then one would like to know about the convergence of the solution when the initial distribution admits zero values. To answer this question, we prove the following result. 

	\begin{theorem}\label{THM5.5}
		Let Assumptions \ref{ASS1.1}, \ref{ASS1.2}, \ref{ASS4.1} and \ref{ASS4.4} be satisfied. Suppose $ \mathbf{u}=\mathbf{u}(t,x) $ is the solution of \eqref{1.1}-\eqref{1.3} provided by Theorem \ref{THM2.3}. 
		Let us denote by $\gamma $ as in \eqref{5.1}. Furthermore, suppose we have 
		\[  r_1=r_2=r,  \]
		in Assumption \ref{ASS4.1} and define 
		\begin{equation*}
		E_{\infty }:=\lim_{t\rightarrow \infty }E[(u_1,u_2)(t,\cdot)],
		\end{equation*} 
		where $E[(u_1,u_2)(t,\cdot)]$ is the energy functional defined in \eqref{4.3}.
		
		\noindent  Then for each  $i=1,2$ and each $t\geq 0
		$ the Dirac measure $\delta _{(u_1+u_2)(t,x)}$ belongs to the space of Young measures $Y\left( \mathbb{T}^{N};[0,\gamma ]\right) $ which means
		\[ (u_1+u_2)(t,x) \in [0,\gamma ],\, \forall  t \geq 0 \text{ and almost every } x \in \mathbb{R}^N, \]
		\[ \int_{A\times [0,\gamma]} \eta(\lambda)\delta_{(u_1+u_2)(t,x)}(\di \lambda)\di x=\int_A\eta((u_1+u_2)(t,x))\di x,\;\forall\,A\in \mathcal{B}(\T^N),\; \forall \eta \in C([0,\gamma], \mathbb{R}). \]
		Moreover, we obtain
		\[ 		r\leq E_{\infty }\leq 2r \]
		and 
		\begin{equation*}
		\lim_{t\rightarrow \infty }\delta _{(u_1+u_2)(t,x)}=(E_{\infty }/r-1)\delta _{0}+(2-E_{\infty }/r) \delta_{r},
		\end{equation*}%
		in the sense of the narrow convergence topology of $ Y(\T^N;[0,\gamma])$. This means that for each continuous function $\eta:[0, \gamma] \to \mathbb{R}$ (indeed bounded function) and for any $ A\in \mathcal B(\T^N) $
		\begin{equation*}
		\lim_{t\rightarrow \infty } \int_A\eta((u_1+u_2)(t,x))\di x =\int_A (E_{\infty }/r-1)\eta(0)+(2-E_{\infty }/r) \eta(r)\di x . 
		\end{equation*}%
	\end{theorem}
	\begin{remark}\label{REM5.6}
		Under the same assumptions as in Theorem \ref{THM5.5}, let $ \lbrace t_n \rbrace_{n\geq 0} $ be any sequence tending to $ \infty $ as $ n\to \infty $. Then the sequence $ \lbrace (u_1+u_2)(t_n,\cdot) \rbrace_{n\geq 0}\subset L^{\infty}_{per}(\R^N) $ is relative compact in $ L^1_{per}(\R^N) $ if and only if
		\[ E_{\infty}=r\quad \text{or}\quad E_{\infty}=2r. \]
		This is a direct consequence of Young measure properties (see \cite[Corollary 3.1.5]{Castaing2004}), which says if the sequence of Young measures $ \lbrace \delta_{(u_1+u_2)(t_n,x)} \rbrace_{n\geq 0}$ converges in the narrow sense to a Young measure $ \nu(x,\cdot) $ and $ \nu(x,\cdot) $ is a Dirac measure $ \delta_{\phi(x)}(\cdot) $ for almost all $ x\in \T^N $. Then given that $ (u_1+u_2)(t,.) $ is uniformly bounded, we deduce
		\[ (u_1+u_2)(t_n,x)\xrightarrow{L^1} \phi(x),\quad n\to \infty. \] 
		In our case, when $ E_{\infty}=r $ (resp. $ =2r $), then 
		\[ (u_1+u_2)(t_n,x)\xrightarrow{L^1} r\text{ (resp. 0)},\quad n\to \infty. \] 
	\end{remark}
	\begin{remark}
		When $ E_\infty $ lies strictly in the interval $(r,2r) $, then $ \delta _{(u_1+u_2)(t,x)} $ converges to two Dirac measures as $ t\to \infty $. To illustrate the notion of narrow convergence to two Dirac measures, one may consider the following example. For each $ n \in \mathbb{N}$, 
		\[ u_n(x)=\begin{cases}
		1,& x\in \Delta x\,[j,j+p),\\
		0,& x\in \Delta x\,[j+p,j+1).
		\end{cases},\quad j=0,1,...,n,\quad p\in (0,1),\;\Delta x=\frac{2\pi}{n+1}. \]
		Then one can prove that 
		\[ \lim_{n\to \infty} \delta_{u_n(x)} = p\delta_{1}+(1-p) \delta_0 \]
		in the sense of narrow convergence.
		Indeed, for any $ \eta \in C_b([0,1]) $ one has
		\begin{align*}
		&\int_{[0,2\pi]} \varphi(x)\int_{[0,1]} \eta (\lambda)\delta_{u_n(x)}(d\lambda)dx =\int_{[0,2\pi]} \varphi(x)\eta(u_n(x))dx\\
		=& \sum_{j=0}^{n}\int_{\Delta x[j,j+p)} \varphi(x) \eta (1) dx +\int_{\Delta x[j+p,j+1)} \varphi(x) \eta (0) dx.
		\end{align*}
		and the result follows when $n \to \infty$. 
	\end{remark}
Since we will consider the convergence of $ t \to \delta_{u_i(t,x)}$ (with values in the probability space) as $t$ goes to infinity in the sense of the narrow convergence topology, we need to introduce the notion of Young measure and the notion of narrow convergence topology in a general sense.   
\begin{definition}[Young measure]
	Let $(\mathcal{S},d)$ be a separable metric space and let $\mathbb{P}(\mathcal{S})$ be the set
	of probability measures on $(\mathcal{S},d)$. Let $(\Omega,\mathcal{A},\mu)$ be a finite measure space endowed with $\sigma -$algebra $\mathcal{A}$ (in practice $\mu$ will be a Lebesgue measure in our case). A map $\nu :\Omega \rightarrow 
	\mathbb{P}(\mathcal{S})$ (i.e. the map $\nu$ maps each $x \in \Omega$ to a probability $B \to \nu(x,B)$ on $\mathcal{S}$) is said to be a \textbf{Young measure} if for each Borel set 
	$B\in \mathcal{B}(\mathcal{S})$ the function $x \mapsto \nu(x,B)$ is measurable from $(\Omega ,\mathcal{A})$ into $[0,1]$. The set of all Young measures from 
	$(\Omega ,\mathcal{A})$ into $\mathcal{S}$ is denoted by $Y\left( \Omega ,\mathcal{A}
	;\mathcal{S}\right) $.
\end{definition}
\begin{definition}[Narrow convergence topology]
		The set $ Y(\Omega,\mathcal{A};\mathcal{S}) $ is endowed with narrow convergence topology which is the weakest topology on 	$ Y(\Omega,\mathcal{A};\mathcal{S}) $ such that all the functionals from $Y(\Omega,\mathcal{A};\mathcal{S})$ into $\mathbb{R}$ defined by 
		$$  
		\nu \longmapsto \int_A\int_{\mathcal{S}} \eta (\lambda)\nu(x,\di \lambda)\mu(\di x) 		
		$$
		is continuous whenever $ A\in \mathcal A $ and $ \eta \in C_b(\mathcal{S};\R) $.
	\end{definition}
	\begin{remark}
		Note that a sequence $ \lbrace \nu^n \rbrace_{n\in \mathbb{N}} \subset Y(\Omega,\mathcal{A};\mathcal{S}) $ narrowly converges to $ \nu \in Y(\Omega,\mathcal{A};\mathcal{S}) $ if and only if for any $ \eta \in C_b(\mathcal{S};\R) $ and $ A\in \mathcal A $
		\[ \lim_{n\to \infty} \int_A\int_{\mathcal{S}} \eta (\lambda)\nu^n(x,\di \lambda)\mu(\di x) =\int_A\int_{\mathcal{S}} \eta (\lambda)\nu(x,\di \lambda)\mu(\di x). \]
		For the sake of simplicity, we use $ Y(\Omega;\mathcal{S}) $ to denote $ Y(\Omega,\mathcal{A};\mathcal{S}) $ if $ \mathcal A=\mathcal B(\Omega) $.
	\end{remark}
	In order to to incorporate the time $t$, we introduce the local narrow convergence topology.
	\begin{definition}[Local narrow convergence topology]
		Let $(\mathcal{S},d)$ be a separable metric space and let $(\Omega,\mathcal{A},\mu)$ be a
		finite measure space (in practice $\mu$ will be a Lebesgue measure in our case). The set $ Y\left(\mathbb{R}\times \Omega,\mathcal{B}(	\mathbb{R})\otimes \mathcal{A}; \mathcal{S} \right)$ is endowed with the local narrow convergence topology denoted by $\mathcal T_{loc}$ which is defined as the weakest topology on $Y\left(\mathbb{R}\times \Omega,\mathcal{B}(\mathbb{R})\otimes \mathcal{A};\mathcal{S}\right)$ such that all the functionals from $Y(\Omega,\mathcal{A};\mathcal{S})$ into $\mathbb{R}$ defined by  
		\begin{equation*}
		\nu \longmapsto \int_{I\times A}\left(\int_\mathcal S
		\eta(\lambda)\nu(t, x,\di \lambda)\right)\left(\di  t\otimes \mu(\di x)\right),
		\end{equation*}
		is continuous for each bounded interval $I\subset \mathbb{R}$, $A\in \mathcal{A}$ and $\eta\in C_b(\mathcal{S};\mathbb{R})$.
	\end{definition}
	For our case, we consider $\Omega =\T^N$, $\mathcal A= \mathcal B\left(\T^N\right)$ is the Borel $ \sigma $-- algebra and $\mu$ is the Lebegues measure.   By setting   
	\begin{equation}\label{5.1}
	\gamma:=\sup_{t\geq 0} \sum_{i=1,2}\| u_i(t,\cdot) \|_{L^\infty}<\infty, 
	\end{equation}
	the set $\mathcal{S}=[0,\gamma]$ corresponds to the range of $u_1+u_2$ which is of course an Euclidean metric space.
To simplify the notations,  we set 
\[ 
Y(\T^N;[0,\gamma]):=Y(\T^N,\mathcal{B}(\T^N);[0,\gamma]). 
\]

We define $Y_{loc}\left(\R\times\T^N;[0,\gamma]\right)$ to be the topological space $Y\left(\R\times\T^N;[0,\gamma]\right)$ endowed with the local narrow convergence topology $\mathcal T_{loc}$.

%	The proof of Theorem \ref{THM5.5} is based on several lemmas. Before starting, we need some supplementary definitions. 
	
	Furthermore, let us consider the probability space $\mathbb{P}\left( 
	\mathbb{T}^{N}\times \lbrack 0,\gamma ]\right) $ and let us recall that the
	usual weak $\ast -$topology on $\mathbb{P}\left( \mathbb{T}^{N}\times
	\lbrack 0,\gamma ]\right) $ is metrizable by using the so-called bounded dual
	Lipschitz metric (Wasserstein metric $ W_p $ when $ p=1 $) defined for each $ \mu, \nu \in \mathbb{P}\left( \mathbb{T}^{N}\times \lbrack 0,\gamma ]\right)$ by 
	\begin{equation*}
	\Theta\left( \mu ,\nu \right) =\sup \left\{ \left\vert \int_{\mathbb{T}
		^{N}\times \lbrack 0,\gamma ]}f(x,\lambda)\, (\mu-\nu) (\di x, \di \lambda) \right\vert \;f\in \mathrm{Lip}\;\left( \mathbb{T}
	^{N}\times \lbrack 0,\gamma ]\right) ,\;\Vert f\Vert _{\mathrm{Lip}}\leq
	1\right\} .
	\end{equation*}%
	Recall that the Lipschitz norm for metric space $ (X,d) $ is defined as follows 
	\[ \Vert f\Vert _{\mathrm{Lip}}=\sup_{x\in X}|f(x)|+\sup_{(x,y)\in
		X^{2},\;x\neq y}\frac{|f(x)-f(y)|}{d(x,y)},\;\;\forall f\in \mathrm{Lip}(X).\]
	We refer to  Dudley \cite[Theorem 18]{Dudley1966} for the equivalence between the weak $\ast -$topology on $\mathbb{P}\left( \mathbb{T}^{N}\times \lbrack 0,\gamma ]\right) $ and the topology induced by $\Theta \left( .,.\right) $.
	In the sequel the probability space $\mathbb{P}\left( \mathbb{T}^{N}\times \lbrack 0,\gamma ]\right)$ is always endowed with the metric topology induced by $\Theta$ without further precision.\\
	Let $\{t_{n}\}_{n\geq 0}$ be a given increasing
	sequence tending to $\infty $ as $n\rightarrow \infty $. Using the above
	definition, we can prove the following lemma. 

\begin{lemma}
	\label{LEM5.8} Let Assumptions \ref{ASS1.1}, \ref{ASS1.2}, \ref{ASS4.1} and \ref{ASS4.4} be satisfied. Let $T>0$ and $i=1,2$ be
	given. The sequence of maps 
	$\left\lbrace  t\longmapsto \mu_{i,t}^n  \right\rbrace_{n \in \mathbb{N}}$
	from $[-T,T]$ to $\mathbb{P}\left(\mathbb{T}^N\times [0,\gamma]\right)$ (endowed with the above metric $\Theta$) and defined by 
	\begin{equation*}
	\int_{\mathbb{T}^N\times [0,\gamma]} g(x,y)\mu_{i,t}^n(\di x,\di y)=|\mathbb{T}%
	^N|^{-1}\int_{\mathbb{T}^N}g\left(x,u_i(t+t_n,x)\right)\di x, \forall g \in C\left(\mathbb{T}^N\times [0,\gamma];\mathbb{R}\right),
	\end{equation*}
	is relatively compact in $C\left([-T,T]; \mathbb{P}\left(\mathbb{T}%
	^N\times [0,\gamma]\right)\right)$.
\end{lemma}
\begin{remark}
In the following, we will use the notation 
	\begin{equation*}
	\mu_{i,t}^n(\di x,\di y)=|\mathbb{T}^N|^{-1} \di x\otimes \delta_{u_i(t+t_n,x)}(\di y). 
	\end{equation*}
	\end{remark}
	\begin{proof}  Let us first consider a classical solution. 
		For each $  g\in C^1(\T^N\times\R) $
		\begin{equation*}
		\int_{\T^N}g(x,u_i(t,x)) \di x -\int_{\T^N}g(x,u_i(s,x)) \di x =\int_s^t \frac{\di }{\di l} \int_{\T^N} g(x,u_i(l,x))\di x \di l.
		\end{equation*}
		Since $u_i$ is bounded, we have
		\begin{equation}\label{5.2}
		\begin{aligned}
		\frac{\di }{\di t} \int_{\T^N} g(x,u_i(t,x))\di x
		=&\int_{\T^N} \partial_u g(x,u_i(t,x))\partial_tu_i(t,x)\di x\\
		=&\int_{\T^N} \partial_u g(x,u_i(t,x))\left(-\mathrm{div}(u_i \mathbf{v})+u_ih_i(u_i)\right)\di x\\
		=&\int_{\T^N} u_i \nabla_x \left[ \partial_u  g(x,u_i(t,x)) \right]\cdot \mathbf{v}+\partial_u g(x,u_i(t,x)) u_ih_i(u_i)\di x,
		\end{aligned}
		\end{equation}
		where the last equality is obtained by applying the Green's formula together with periodic boundary condition. 
		We can see that
		\[  u_i(t,x) \nabla_x \left[ \partial_u  g(x,u_i(t,x)) \right]= \nabla_x \left[u_i(t,x)\partial_u g(x,u_i(t,x))-g(x,u_i(t,x))\right]+\mathbf{p}(x,u_i(t,x)) \]
		where $\mathbf{p}(x,u)=\nabla_x g(x,u)$. 
		 
By substituting the last formula into \eqref{5.2} and by using again the periodicity we derive that
\begin{equation} \label{5.1.0}
\begin{aligned}
\frac{\di }{\di t} \int_{\T^N} g(x,u_i(t,x))\di x=&-\int_{\T^N} \left[u_i(t,x) \partial_u g(x,u_i(t,x))-g(x,u_i(t,x))\right] \mathrm{div}\; \mathbf{v}(t,x)\di x\\
&+\int_{\T^N} \mathbf{p}(x,u_i(t,x)) \cdot \mathbf{v}(t,x)\di x\\
&+\int_{\T^N} \partial_u g(x,u_i(t,x)) u_i(t,x)h_i(u_i(t,x))\di x.
\end{aligned}
\end{equation}
The formula \eqref{5.2} extends to the solution integrated along the characteristics by usual density arguments. 
Incorporating the estimation of $ \sup_{t\geq 0}\| u(t,\cdot)\|_{L^\infty} $ in Theorem \ref{THM4.11},  the estimation of $ \mathbf{v} $ in Lemma \ref{LEM4.3}, and the above equality   \eqref{5.1.0}, we deduce that there exists a constant $ M >0 $ such that 
		\[\left|\int_{\R^N}g(x,u_i(t,x)) \di x -\int_{\R^N}g(x,u_i(s,x)) \di x\right|\leq M \| g\|_{\mathrm{Lip}(\T^N\times [0,\gamma])}|t-s| . \]
		 From the definition of the metric on $ \Theta\left( \mu ,\nu \right) $, we can see that
		 \begin{equation*} 
		 \Theta\left(\mu_{i,t}^n,\mu_{i,s}^n\right)\leq M|t-s|.
		 \end{equation*}
		 From this we observe that each map $t \to \mu_{i,t}^n$ is continuous from $[-T,T]$ to $\mathbb{P}\left(\mathbb{T}^N\times [0,\gamma]\right)$. By Prohorov's compactness theorem \cite[Theorem 5.1]{Billingsley1999}, the space $\mathbb{P}\left(\mathbb{T}^N\times [0,\gamma]\right)$ endowed with the metric $ \Theta$ is a compact metric space. Therefore we can apply Arzela-Ascoli theorem and the result follows. 
	\end{proof}		

Since $ u $ is uniformly bounded, one can deduce the following compact result for Young measures (see \cite[Theorem 9.15]{Serre1996}).
%	\begin{lemma}
%		\label{LEM5.9} The sequence $\left\{\delta_{u_i(t_n,.)}\right\}_{n\geq 0}$ is
%		relatively compact for the narrow convergence topology of $Y\left(\mathbb{T}
%		^N;[0,\gamma]\right)$.
%	\end{lemma}
%	Next the following proposition is also a direct consequence of Young measures
%	properties.
	\begin{lemma}
		\label{LEM5.10} The sequence $\left\{\delta_{u_i(t+t_n,x)}\right\}_{n\geq 0}$ is
		relatively compact for the local narrow convergence topology of $Y_{loc}\left(\R
		\times\T^N;[0,\gamma]\right)$.
	\end{lemma}
	
	Using the above Lemma \ref{LEM5.8} and Lemma \ref{LEM5.10}, up to a subsequence, one can assume that there exists a Young measure $\nu\equiv \nu_{i,t}(x,.)\in Y\left(\mathbb{R}%
	\times\mathbb{T}^N;[0,\gamma]\right)$ such that 
	\begin{equation}\label{5.4}
	\lim_{n\to\infty} \delta_{u_i(t+t_n,x)}=\nu_{i,t}(x,.)\text{ in the topology of }%
	Y_{loc}\left(\mathbb{R}\times\mathbb{T}^N;[0,\gamma]\right).
	\end{equation}
	and
	\begin{equation}  \label{5.4bis}
	\lim_{n\to\infty} \mu_{i,t}^n= \mu_{i,t}^{\infty}
	\end{equation}
	where the limit holds to the locally uniform continuous topology of $C\left(\mathbb{R}; \mathbb{P}\left(\mathbb{T}^N\times
	[0,\gamma]\right)\right)$. Here we would like to recall that the limits $\mu_{i,t}^{\infty}$ and $\nu_{i,t}(x,.)$ depend on the choice of subsequence. \\
	Next, by definition one has for each continuous function $f\in C\left(\mathbb{T}^N\times [0,\gamma];\mathbb{R}\right)$ and each $n\geq 0$:
	\begin{equation*}
	\int_{\mathbb{T}^N\times [0,\gamma]} f(x,y)\mu_{i,t}^n(\di x,\di y)=|\mathbb{T}%
	^N|^{-1}\int_{\mathbb{T}^N}\int_{[0,\gamma]}f\left(x,y\right)\delta_{u_i(t+t_n,x)}(\di y)\; \di x.
	\end{equation*}
	From \eqref{5.4} and \eqref{5.4bis}, passing to the limit $n\to\infty$ yields to
	\begin{equation*}
	\int_{\mathbb{T}^N\times [0,\gamma]} f(x,y)\mu_{i,t}^{\infty}(\di x,\di y)=|\mathbb{T}%
	^N|^{-1}\int_{\mathbb{T}^N}\int_{[0,\gamma]}f\left(x,y\right)\nu_{i,t}(x,\di y)\;\di x.
	\end{equation*}
	This rewrites as
	\begin{equation*} 
	\mu_{i,t}^{\infty}(\di x,\di y)=|\mathbb{T}^N|^{-1} \di x\otimes \nu_{i,t}(x,\di y).
	\end{equation*}	
%	Note that
%	\[\mu_{i,t}^n(\di x,\di y)= |\mathbb{T}^N|^{-1}\di x\otimes \delta_{u_i(t+t_n,x)}(\di y). \]
%	From \eqref{5.4} and \eqref{5.4bis}, one has 
%	\begin{equation}  \label{5.5}
%	\mu_{i,t}^{\infty}(\di x,\di y)=|\mathbb{T}^N|^{-1}\di x\otimes \nu_{i,t}(x,\di y)\text{ a.e. $(t,x)\in\mathbb{R}%
%		\times\mathbb{T}^N$}.
%	\end{equation}
%	Here let us also recall that $\nu_{i,t}(x,.)$ depends on the choice of subsequence.
	
	The aim of the following lemmas is to identify the family of measures $\nu_{i,t}(x,.)$.
	Our next result describes the support of  $\nu_{i,t}(x,.)$.
	\begin{lemma}\label{LEM5.11}
		Under the same assumptions of Lemma \ref{LEM5.8}, for $ i=1,2, $ there exist measurable maps $a_i:\mathbb{R}\times\mathbb{T}^N\to 
		\mathbb{R}$ such that $0\leq a_i(t,x)\leq 1$ $a.e.$ $(t,x)\in\mathbb{R}\times%
		\mathbb{T}^N$ and 
		\begin{equation*}
		\nu_{i,t}(x,.)=\left(1-a_i(t,x)\right)\delta_0(.)+a_i(t,x)\delta_{r_i}(.),\;\;a.e.\;(t,x)\in%
		\mathbb{R}\times\mathbb{T}^N.
		\end{equation*}
	\end{lemma}
	\begin{proof}
		Let us reconsider $ F_i(u):=u\left| h_i(u)\ln({u}/{r_i})\right| $ for $ u\in [0,\infty) $ and recall that from equation \eqref{4.6} we have for any $ \tau>0 $
		\begin{equation*}
		\lim_{t\to+\infty}\int_{t}^{t+\tau}\int_{\mathbb{T}^N}F_i(u_i(s,x))\di x\di s=0,\;i=1,2.
		\end{equation*}	
		Therefore, for $ i=1,2$ and from equation \eqref{5.4bis}
		\begin{align*}
		0=&\lim_{n\to \infty} \int_{0}^{\tau}\int_{\mathbb{T}^N}F_i(u_i(t+t_n,x))\di x\di t\\
		=&\lim_{n\to \infty}|\T^N|\int_{0}^{\tau}\int_{\mathbb{T}^N\times[0,\gamma]}F_i(\lambda)\mu_{i,t}^n(\di x,\di \lambda)\di t\\
		=&\int_{0}^{\tau}\int_{\mathbb{T}^N\times[0,\gamma]}F_i(\lambda)\nu_{i,t}(x,\di \lambda)\di x \di t.
		\end{align*}
		Since the map $u\mapsto F_i(u)$ is non-negative and only vanishes at $u=0$
		and $u=r_i$ one obtains that 
		\begin{equation*}
		\mathrm{supp}\;\nu_{i,t}(x,.)\subset \{0\}\cup \{r_i\},\;\;a.e.\;(t,x)\in\mathbb{R}%
		\times\mathbb{T}^N.
		\end{equation*}
		The above characterization of the support allows us to rewrite
		\begin{equation*}
		\nu_{i,t}(x,.)=\nu_{i,t}\left(x,\left\{0\right\}\right)\delta_0(.)+\nu_{i,t}\left(x,\left\{r_i\right\}\right)\delta_{r_i}(.),\;a.e.\;(t,x)\in\mathbb{R}\times\mathbb{T}^N.
		\end{equation*}
		Finally set $a_i(t,x)\equiv \nu_{i,t}(x,
		\left\{r_i\right\})$. Recalling that $(t,x)\mapsto \nu_{i,t}(x,.)$ is measurable with value as a probability measure, thus $ \nu_{i,t}\left(x,\left\{0\right\}\right)=1- \nu_{i,t}\left(x,\left\{r_i\right\}\right)$ and $ (t,x)\mapsto a(t,x)$ is measurable,  the result follows.
	\end{proof}
	
	Our next result shows the measurable function $ a_i(t,x) $ is independent of the variable $ t $.
	\begin{lemma}\label{LEM5.12} 
		Under the same assumptions of Lemma \ref{LEM5.8}, there exists a measurable map $c_i:\T^N\to \T^N$ such
		that $a_i\equiv a^i(t,x)$ provided by Lemma \ref{LEM5.11} is independent of $ t $ and satisfies for any $  t\in \R $, 
		\[ 
		a_i(t,x)\equiv c_i(x),\  a.e.\ x \in\mathbb{T}^N,\;i=1,2. \]
		Moreover, for any $ t\in \R $,
	\[ \nu_{i,t}(x,.)= (1-c_i(x))\delta_0(.)+c_i(x)\delta_{r_i}(.),\;\ a.e. \ x\in\T^N,\;i=1,2, \]
		for some measurable functions $c_i:\T^N\to \T^N,\,i=1,2$. 

Furthermore, we have
	\begin{equation}\label{5.6}
	\lim_{n\to \infty}\delta_{u_i(t+t_n,x)}= (1-c_i(x))\delta_0+c_i(x)\delta_{r_i},
	\end{equation}
	in the sense of the narrow convergence and where the limit depends on the choice of subsequence.
	\end{lemma}
	\begin{proof}
		Suppose $ \mathbf{u}=(u_1,u_2) $ the classical solution. For any $ \lbrace  t_n\rbrace_{n\geq 0}$ with $ t_n\to \infty $ as $ n\to 0 $ and any $ \phi\in C_c^1(\T^N) $, 
		 \[ \int_{ \T^N}  \phi(x)  \partial_t u_i(t+t_n,x)\di x +\int_{\T^N} \phi (x) \mathrm{div}(u(t+t_n,x) \textbf{v}(t+t_n,x))\di x=\int_{\T^N}\phi(x) u_i(t+t_n,x)h_i(u_i(t+t_n,x))\di x. \]
		Since $ \phi $ has compact support, we have
		 \[ \int_{ \T^N}  \phi(x)  \partial_t u_i(t+t_n,x)\di x= \int_{\T^N} \nabla\phi (x) \cdot \textbf{v}(t+t_n,x) u(t+t_n,x) \di x +\int_{\T^N}\phi(x) u_i(t+t_n,x)h_i(u_i(t+t_n,x))\di x. \]
		 Let $T\in \R$ and $ \delta >0 $ be given. Integrating the both sides over $ (T,T+\delta)
		 $ leads to 
		 \begin{equation}\label{5.6.0}
		 \begin{aligned}
		 -\int_{\T^N}\phi(x)\big(u_i(T+\delta+t_n,x)-u_i(T+t_n,x)\big)\di x=&\int_{T}^{T+\delta} \int_{\T^N} \nabla\phi (x) \cdot \textbf{v}(t+t_n,x) u(t+t_n,x) \di x\di t\\
		  &+\int_{T}^{T+\delta}\int_{\T^N}\phi(x) u_i(t+t_n,x)h_i(u_i(t+t_n,x))\di x \di t.
		 \end{aligned}
		 \end{equation}
		Equation \eqref{5.6.0} remains true for general mild by using density argument and applying Theorem \ref{THM2.3}-(iii). 
		
		 For the right-hand-side of \eqref{5.6.0}, by \eqref{4.10.0} in Remark \ref{REM4.10} that $ \lim_{t\to \infty} \| \textbf{v}(t,.) \|_{C^0}=0 $ we have for the first term
		 \begin{equation}\label{5.6.1}
		 \begin{aligned}
		 &\lim_{n\to \infty}\left|\int_{T}^{T+\delta}\int_{\T^N} \nabla\phi (x) \cdot \textbf{v}(t+t_n,x) u(t+t_n,x) \di x\di t\right| \\
		 \leq& \lim_{n\to \infty} \delta |\T^N| \|\phi  \|_{C^1} \sup_{t\geq 0}\| u_i(t,.) \|_{L^\infty} \| \textbf{v}(t+t_n,.)\|_{C^0}=0.
		 \end{aligned}
		 \end{equation}
		 and the second term
		 \begin{equation*}
		 \int_{T}^{T+\delta} \int_{\T^N}\phi(x) u_i(t+t_n,x)h_i(u_i(t+t_n,x))\di x \di t=\int_{T}^{T+\delta}\int_{\mathbb{T}^N} \phi(x)\left[\int_{[0,\gamma]}\lambda h_i\left(\lambda
		 \right)\delta_{u_i(t+t_n,x)}(\di \lambda)\right]\di x\di t.
		 \end{equation*}
		  Letting $n\to\infty$, we have
		 \begin{equation}\label{5.6.2}
		 \begin{aligned}
		 &\lim_{n\to \infty}\int_{T}^{T+\delta}\int_{\T^N}\phi(x) u_i(t+t_n,x)h_i(u_i(t+t_n,x))\di x \di t\\
		 =& \int_{T}^{T+\delta}\int_{%
		 	\mathbb{T}^N}\phi (x)\left[\int_{[0,\gamma]}\lambda h_i\left(\lambda\right)\left[\left(1-a_i(t,x)\right)%
		 \delta_0+a_i(t,x)\delta_{r_i}\right](\di \lambda)\right]\di x\di t= 0
		 \end{aligned}
		 \end{equation}
		 Therefore, by \eqref{5.6.1} and \eqref{5.6.2} we deduce the left-hand-side of \eqref{5.6.0}
		 \begin{equation*}
		 \lim_{n\to \infty }-\int_{\T^N}\phi(x)\big(u_i(T+\delta+t_n,x)-u_i(T+t_n,x)\big)\di x=-r_i\int_{\T^N}\phi(x)\big(a_i(T+\delta,x)-a_i(T,x)\big)\di x=0.
		 \end{equation*}
		 Hence we have
		 \[ \int_{\T^N}\phi(x)\big(a_i(T+\delta,x)-a_i(T,x)\big)\di x=0,\;\forall \phi(x)\in C_c^1(\T^N). \]
		 Since $ T \in \R $ and $ \delta>0 $ is arbitrary, we deduce for any $ t\in \R $
		 \begin{equation}\label{5.8.1}
		  a_i(t,x)=c_i(x),\;a.e. \; x\in \T^N.
		 \end{equation}	 
		 The last part of the lemma now follows by the above equation \eqref{5.8.1}, \eqref{5.4} and Lemma \ref{LEM5.12}. 
		\end{proof}

	Next, we study the narrow convergence of the measure $ \delta_{(u_1+u_2)(t+t_n,x)} $ as $ n\to \infty $.

		\begin{corollary}
			Let $\{t_{n}\}_{n\geq 0}$ be a given increasing
			sequence tending to $\infty $ as $n\rightarrow \infty $. Then, up to a subsequence, we have two measurable functions $ c_i(x)\in [0,1]$ for $ i=1,2, $ such that for any $ t\geq 0 $
			\begin{equation*}
			\lim_{n\to \infty} \delta_{(u_1+u_2)(t+t_n,x)}=\left(1-\sum_{i=1,2}c_i(x)\right)\delta_{0}+\sum_{i=1,2}c_i(x)\delta_{r_i} \;\text{in the sense of narrow convergence.}
			\end{equation*}
			\end{corollary}
			\begin{proof}
				From segregation property in Theorem \ref{THM3.1}, we have for any $ \eta\in C([0,\gamma]) $ that
				\begin{equation*}
				\eta \left(u_1(t,x)+u_2(t,x)\right)+\eta(0)=\eta (u_1(t,x))+\eta (u_2(t,x)),\quad \forall (t,x)\in \R_+\times \T^N,
				\end{equation*}
				which is equivalent to say that 
				$$
				\delta_0+\delta_{(u_1+u_2)(t,x)}=\delta_{u_1(t,x)}+\delta_{u_2(t,x)}.
				$$
				Therefore, for any $ \varphi \in L^1(\T^N) $, we have
				\begin{equation*}
				\begin{aligned}
				&\lim_{n\to \infty} \int_{\T^N}\varphi(x)\int_{[0,\gamma]} \eta(\lambda)\left(\delta_0+\delta_{(u_1+u_2)(t+t_n,x)}\right)(\di \lambda)\di x\\
				=&\lim_{n\to \infty} \int_{\T^N}\varphi(x)\int_{[0,\gamma]} \eta(\lambda)\left(\delta_{u_1(t+t_n,x)}+\delta_{u_2(t+t_n,x)}\right)(\di \lambda)\di x\\
				=&\int_{\T^N}\varphi(x)\int_{[0,\gamma]} \eta(\lambda)\left(\left(2-\sum_{i=1,2}c_i(x)\right)\delta_{0}+\sum_{i=1,2}c_i(x)\delta_{r_i}\right)(\di \lambda)\di x
				\end{aligned}
				\end{equation*}
				By simplifying the term $\delta_0$ from each side, we deduce that 
				\begin{equation}\label{5.7}
				\lim_{n\to \infty} \delta_{(u_1+u_2)(t+t_n,x)}=\left(1-\sum_{i=1,2}c_i(x)\right)\delta_{0}+\sum_{i=1,2}c_i(x)\delta_{r_i} 
				\end{equation}
				in the sense of the narrow convergence topology of $ Y(\T^N;[0,\gamma])$. Here we recall that the limit depends on the choice of subsequence.
			\end{proof}
	\begin{lemma}\label{LEM5.13}
		Under the same assumptions as in Lemma \ref{LEM5.8}, the  following equality holds true: 
		\[ r_1c_1(x)+r_2c_2(x)\equiv r_1+r_2-E_\infty,\;a.e.\ x\in \T^N, \]
		where $ E_\infty :=  \lim_{t\to \infty}E\left[(u_1,u_2)(t,\cdot)\right] $ in \eqref{4.3}.
	\end{lemma}
	\begin{proof}
		Recall equation \eqref{4.3.0} where we have $ G_i(0)=r_i,G(r_i)=0 $, we can see that
		\begin{equation}\label{5.9}
		\begin{aligned}
		\lim_{n\to \infty}E_i\left[u_i(t+t_n,\cdot)\right]=&\lim_{n\to \infty}\frac{1}{|\T^N|} \int_{\T^N}G_i(u_i(t+t_n,x)) \di   x\\
		=&\lim_{n\to \infty}\frac{1}{|\T^N|} \int_{\T^N\times [0,\gamma]} G_i(\lambda) \delta_{u_i(t+t_n,x)}(\di \lambda)\di  x\\
		=&\lim_{n\to \infty}\frac{1}{|\T^N|} \int_{\T^N\times [0,\gamma]} G_i(0)(1-c_i(x)) +G_i(r_i) c_i(x)\di  x\\
		=&r_i-\frac{1}{|\T^N|} \int_{\T^N} r_ic_i(x) \di  x.
		\end{aligned}
		\end{equation}
		Meanwhile, from \eqref{4.8} the Fourier coefficients satisfy
		\[ \lim_{t\to \infty}c_k\left[(u_1+u_2)(t,\cdot)\right]=0, \quad \forall k\in \Z^N \backslash \lbrace 0 \rbrace.\]
		On the other hand, we have for all $ k\in \Z^N \backslash \lbrace 0 \rbrace $
		\begin{align*}
		\lim_{n\to \infty}c_k\left[(u_1+u_2)(t+t_n,\cdot)\right]=&\lim_{n\to \infty}\frac{1}{|\T^N|} \int_{\T^N} e^{-ikx}(u_1+u_2)(t+t_n,x) \di x\\
		=&\lim_{n\to \infty}\frac{1}{|\T^N|} \int_{\T^N\times [0\times\gamma]} e^{-ikx} \lambda \left(\delta_{u_1(t+t_n,x)}+\delta_{u_1(t+t_n,x)}\right)(\di \lambda)\di  x\\
		=&\frac{1}{|\T^N|} \int_{\T^N} e^{-ikx} (r_1c_1(x)+r_2c_2(x))\di  x.
		\end{align*}
		Since $ c_1,c_2\in L^\infty(\T^N)\subset L^2(\T^N) $ and $ \lbrace e^{-ikx}\rbrace_{k\in\Z} $ is a basis of $ L ^2(\T^N) $.
		This implies that $ r_1c_1(x)+r_2c_2(x)$ is a constant function. Recall that
		\[ E_\infty= \lim_{n\to \infty}\sum_{i=1,2}E_i\left[u_i(t+t_n,\cdot)\right]=r_1+r_2-\frac{1}{|\T^N|} \int_{\T^N} \sum_{i=1,2}r_ic_i(x)dx,\]
		thus the result follows.
	\end{proof}
	
%	The following lemma proves that $ E_{\infty}\in [r,2r] $.
	\begin{lemma}[Segregation at $ t=\infty $]\label{LEM5.14}
		Under the same assumptions as in Lemma \ref{LEM5.8}, the following equation holds 
		\[  c_1(x)c_2(x)=0,\quad a.e.,\;x\in \T^N. \]
		Moreover when $ r_1=r_2=r $, then 
		\[ r\leq E_\infty \leq 2r. \]
	\end{lemma}
	\begin{proof}
		By using the segregation property in Theorem \ref{THM3.1}, we can see that, for any $ \eta \in C_b([0,\gamma]) $,
		\[ \eta \bigg(\left(u_1(t,x)+u_2(t,x)\right)^2 \bigg)=\eta \bigg( u_1^2(t,x)+u_2^2(t,x)\bigg),\;\forall\, t\in \R_+,\;a.e.\,x\in  \T^N. \]
		Therefore, for any Borel set $ A\in \mathcal B(\T^N) $, we deduce the following equation
		\begin{equation}\label{5.8}
		\int_{A\times [0,\gamma]}\eta (\lambda^2) \delta_{(u_1+u_2)(t+t_n,x)}(\di\lambda) \di x=\int_{A\times [0,\gamma]^2}\eta (\lambda_1^2+\lambda_2^2) \delta_{u_1(t+t_n,x)}(\di \lambda_1)\delta_{u_2(t+t_n,x)}(\di \lambda_2) \di x.
		\end{equation}
		By equation \eqref{5.6} and \eqref{5.7}, we let $ n\to \infty $, then for the Left-Hand-Side (L.H.S.) of equation \eqref{5.8}
		\begin{align*}
		\lim_{n\to \infty} \mathrm{L.H.S.} =&\int_{A\times [0,\gamma]} \eta(\lambda^2) \left[\left(1-\sum_{i=1,2}c_i(x)\right)\delta_{0}(\di \lambda) +\sum_{i=1,2}c_i(x) \delta_{r_i}(\di \lambda) \right] \\
		=&\int_{A}  \eta (0)\left(1-\sum_{i=1,2}c_i(x)\right) +\sum_{i=1,2}\eta(r_i^2)c_i(x)\di x.
		\end{align*}
		Then for the Right-Hand-Side (R.H.S.) of equation \eqref{5.8}
		\begin{align*}
		&\lim_{n\to \infty} \mathrm{R.H.S.}\\ =&\int_{A\times [0,\gamma]^2} \eta(\lambda_1^2+\lambda_2^2) \left[\left(1-c_1(x)\right)\delta_{0}(d\lambda_1) +c_1(x)\delta_{r_1}(d\lambda_1) \right]\otimes \left[\left(1-c_2(x)\right)\delta_{0}(d\lambda_2) +c_2(x)\delta_{r_2}(d\lambda_2) \right] dx\\
		=&\int_{A\times [0,\gamma]}  \eta (\lambda_2^2)\left(1-c_1(x)\right) +\eta(r_1^2+\lambda_2^2)c_1(x)\left[\left(1-c_2(x)\right)\delta_{0}(d\lambda_2) +c_2(x)\delta_{r_2}(d\lambda_2) \right] dx\\
		=&\int_{A}\Big( \eta(0)\prod_{i=1,2}(1-c_i(x)) +\eta(r_1^2)c_1(x)(1-c_2(x))\\
		&+\eta(r_2^2)c_2(x)(1-c_1(x))+\eta(r_1^2+r_2^2)c_1(x)c_2(x)\Big)dx.
		\end{align*}
		Comparing the two limits and noticing that $ A\in\mathcal{B}(\T^N) $ is arbitrary, we conclude that
		\[ c_1(x)c_2(x)\bigg[\eta(0)+\eta(r_1^2+r_2^2)-\eta(r_1^2)-\eta(r_2^2)\bigg]=0,\;\text{for a.e. } x\in \T^N. \]
		Furthermore, since $ \eta \in C_b([0,\gamma]) $ is any given function, we can choose  an $ \eta  $ such that 
		\[  \eta(0)+\eta(r_1^2+r_2^2)-\eta(r_1^2)-\eta(r_2^2)\neq 0, \]
		thus 
		\begin{equation}\label{5.11}
		c_1(x)c_2(x)=0,\quad a.e.,\;x\in \T^N. 
		\end{equation}
		Since by Lemma \ref{LEM5.11} and \ref{LEM5.12}, one has $ 0\leq c_i(x)\leq 1 $ for any $ x\in \T^N $. Hence, one can deduce from Lemma \ref{LEM5.13} \[ 0\leq E_\infty \leq r_1+r_2 . \]
		Moreover, one can deduce from \eqref{5.11} that 
		\[ \min\lbrace r_1,r_2 \rbrace\leq E_\infty \leq r_1+r_2. \]
		If we assume  $ r_1=r_2=r $, then 
		\[ r\leq E_\infty \leq 2r. \]
	\end{proof}
%		
%	Now we can finish the proof of Theorem \ref{THM5.5}.
	
	\begin{proof}[Proof of Theorem \ref{THM5.5}]
%		We give a resume of the proof.
		By Lemma \ref{LEM5.10}, the sequence $ \lbrace  \delta_{u_i(t+t_n,x)}\rbrace_{n\geq 0} $ is relatively compact in $Y_{loc}\left(\R
		\times\T^N;[0,\gamma]\right)$ with locally narrow topology, thus, up to a sequence, we have 
		\[ \lim_{n\to\infty} \delta_{u_i(t+t_n,x)}=\nu_{i,t}(x,.)\text{ in the topology of }%
		Y_{loc}\left(\mathbb{R}\times\mathbb{T}^N;[0,\gamma]\right).\]
		The key arguments of the proof lies in the two consequences of the decreasing energy functional, namely, equation \eqref{4.5} and  equation \eqref{4.6}.
		Lemma \ref{LEM5.11} is a consequence of the first equation \eqref{4.5} by which we can determine the support of $ \nu_{i,t}(x,.)  $, i.e., there exists measurable functions $ a_i(t,x) $ such that
			\[ \nu_{i,t}(x,.)= (1-a_i(t,x))\delta_0(.)+a_i(t,x)\delta_{r_i}(.),\;\ a.e. \ x\in\T^N,\;i=1,2. \] 
		Moreover, Lemma \ref{LEM5.8} and Lemma \ref{LEM5.12} enable us to write $ a_i(t,x)\equiv c_i(x),\,i=1,2. $ Thus, we have 
		\begin{equation*}
		\lim_{n\to \infty}\delta_{u_i(t+t_n,x)}= (1-c_i(x))\delta_0+c_i(x)\delta_{r_i} \text{ in the topology of }%
		Y_{loc}\left(\mathbb{R}\times\mathbb{T}^N;[0,\gamma]\right)
		\end{equation*}
		Applying the segregation property, we have 
		\[ \delta_0+\delta_{(u_1+u_2)(t,x)}=\delta_{u_1(t,x)}+\delta_{u_2(t,x)}, \]
		hence by Corollary \ref{COR4.12},
		\begin{equation}\label{5.12}
		\lim_{n\to \infty} \delta_{(u_1+u_2)(t+t_n,x)}=\left(1-\sum_{i=1,2}c_i(x)\right)\delta_{0}+\sum_{i=1,2}c_i(x)\delta_{r_i} 
		\end{equation}
		If in addition, we assume that  $ r_1=r_2=r $, we apply Lemma \ref{LEM5.13} where we used the decaying of Fourier coefficients in equation \eqref{4.6}, which yields
		\[ \sum_{i=1}^{2}c_i(x)=2-\frac{E_\infty}{r}. \] 
		together with equation \eqref{5.12} we obtain 
		\[ \lim_{n\to \infty} \delta_{(u_1+u_2)(t+t_n,x)}=(E_{\infty }/r-1)\delta _{0}+(2-E_{\infty }/r) \delta_{r},  \]
		in the sense of the narrow convergence topology of $ Y(\T^N;[0,\gamma])$ and by  Lemma \ref{LEM5.14} we have $ E_\infty\in [r,2r] $. Now the limit does not depend on $ t $ and the choice of the subsequence. Since $ \lbrace t_n\rbrace_{n\geq 0} $ is any given sequence that tends to infinity and $\left( \mathbb{T}^{N},\mathcal{B}(\mathbb{T}%
		^{N})\right) $ is a countably generated $\sigma-$algebra then the topology $ Y(\T^N;[0,\gamma])$ is metrizable (see for instance \cite[Theorem 1]{Valadier1990} or the monograph \cite{Castaing2004}), therefore we can conclude that
		\[ \lim_{t\to \infty} \delta_{(u_1+u_2)(t,x)}=(E_{\infty }/r-1)\delta _{0}+(2-E_{\infty }/r) \delta_{r}.  \]
	\end{proof}
\section{Discussion and Numerical Simulations}
In this section we will study the  system \eqref{1.1} numerically for the one dimensional case. Our original motivation is coming from two species of cell growing in a petri dish. The two dimensional case will be considered in some future work. 

Here we will focus on the coexistence and the exclusion principle for two species. From Theorem \ref{THM5.5}, we deduce
\begin{equation*}
\lim_{t\rightarrow \infty }\delta _{(u_1+u_2)(t,x)}=(E_{\infty }/r-1)\delta _{0}+(2-E_{\infty }/r) \delta_{r},\text{ in the sense of narrow convergence.}
\end{equation*}
Therefore the limit $ E_{\infty}:= \lim_{t\to \infty} E[(u_1,u_2)(t,\cdot)] $ is an important index to determine whether solution $ u_1+u_2 $ converges to a Young measure in the sense of narrowly convergence or to a constant in $ L^1 $ norm (see Remark \ref{REM5.6}). To that aim, we trace the curve $ t\longmapsto E[(u_1,u_2)(t,\cdot)] $ in numerical simulations, which has been analytically proved decreasing in Theorem \ref{THM4.6}. Moreover, we also plot the curve $ t\longmapsto E_i[u_i(t,\cdot)],\;i=1,2, $ respectively. This will help us to understand  the limit for each species $ u_i $.

In the numerical simulations, we focus on the convergence of the energy functional which implies the convergence of the total number for each species. In fact, by using \eqref{5.6} we obtain 
$$
\lim_{t\to \infty} \dfrac{1}{|\T|}\int_{\T}u_i(t,x)dx=\lim_{t\to \infty} \dfrac{1}{|\T|}\int_{\T} \int_{[0,\gamma]} \lambda \delta_{u_i(t,x)}(d\lambda)dx= \dfrac{r_i}{|\T|}\int_{\T} \int_{[0,\gamma]} c_i(x) dx.
$$
Hence by using \eqref{5.9} one has
\begin{equation}\label{6.1.0}
	\lim_{t\to\infty} E_i[u_i(t,\cdot)]  =r_i\left(1-\frac{1}{|\T|}\int_{\T}c_i(x)dx\right)=  r_i - \lim_{t\to \infty} \dfrac{1}{|\T|}\int_{\T}u_i(t,x)dx.
\end{equation}
That is to say that the energy functional inform us about the asymptotic number of individuals for each species.

Our numerical simulations will be decomposed as follows. 
\bigskip

	\noindent \textbf{Coexistence:} If $ r_1= r_2 =r$, then $ c_1(x),c_2(x)\in (0,1),\;a.e.,\;x\in \T^N $. For each species, the following limits exist
	\[ \lim_{t\to\infty}\| u_i(t,\cdot)\|_{L^1}= r\int_{\T^N}c_i(x)dx\in (0,r),\;i=1,2. \]
We will see that the relative location of each species influences the asymptotic number in each species. Moreover, we have \[ (u_1+u_2)(t,x)\xrightarrow{L^1} r,\;t\to \infty. \]
	
	\bigskip
	\noindent \textbf{Exclusion Principle:} If $ r_1 > r_2 $ (resp. $ r_1<r_2 $) then $ c_1(x)=1,\,c_2(x)=0$ (resp. $  c_1(x)=0,\,c_2(x)=1 $) $\;a.e.,\;x\in \T^N $ , which implies 
	\[ u_1(t,x)\xrightarrow{L^1}r_1,\quad u_2(t,x)\xrightarrow{L^1}0,\; \mathrm{(}resp.\; u_1(t,x)\xrightarrow{L^1}0,\quad u_2(t,x)\xrightarrow{L^1}r_2 \mathrm{)},  \]
	and
	\[ (u_1+u_2)(t,x)\xrightarrow{L^1}  \max\lbrace r_1,r_2 \rbrace,\;t\to \infty. \]
	
\bigskip
\subsection{The case $ r_1=r_2 $ implies coexistence}
The goal of this first part is to confirm numerically Theorem \ref{THM5.5}. 
	It is interesting to notice that in Theorem \ref{THM5.5}, we only assume the equilibrium of the corresponding ODE system for each species to be the same without imposing any other condition on $ h $, which means that the dynamics for these two species can be different. Hence, we will use the following two different reaction functions for different species 
	\begin{equation}\label{6.1.1}
	u_1h_1(u_1+u_2)=u_1\bigg(\frac{b_1}{1+\gamma(u_1+u_2)}-\mu\bigg),\quad u_2h_2(u_1+u_2)=b_2u_2\bigg(1-\frac{u_1+u_2}{K}\bigg).
	\end{equation}

One can verify that $ h_i $ satisfies Assumption \ref{ASS1.1} and Assumption \ref{ASS4.1} with their roots (i.e., $ h_i(r_i)=0,\,i=1,2 $) as 
\[ r_1:=\frac{b_1-\mu}{\gamma\mu},\quad r_2=K. \]
Our kernel $ \rho $ in the simulation is chosen as  
\begin{equation}\label{6.1}
\rho(\mathbf{x})=e^{-\pi |\mathbf{x}|^2},\quad \mathbf{x}\in \R^N,
\end{equation}
which is the Gaussian kernel and we consider the dimension $ N=1 $ in this section. Therefore, due to Remark \ref{REM1.3} and Remark \ref{REM4.5}, Assumption \ref{ASS1.2} and Assumption \ref{ASS4.4} are satisfied. 

We set the initial distribution for two species to be of compact supports and separated. From Theorem \ref{THM3.1}, we shall observe the segregation property of two species as time evolves. Our parameters in system \eqref{1.1} are given as
\begin{equation}\label{6.2}
	b_1=b_2=1.2,\;\mu=1,\;\gamma=1, \; K=0.2,
\end{equation}
therefore one can calculate 
\[r_1=r_2=0.2.  \]
Now we trace the curve $ t\longmapsto E[(u_1,u_2)(t,\cdot)] $ in numerical simulation, which has been analytically proved decreasing in Theorem \ref{THM5.5}. We also plot the curve $ t\longmapsto E_i[u_i(t,\cdot)],\;i=1,2, $ respectively. Moreover, we plot the variation of the mean value of the total number of individuals for each species, that is 
\[ t\longmapsto \frac{1}{2\pi}\int_{0}^{2\pi}u_i(t,x)dx,\quad i=1,2. \]
\begin{figure}[H]
	\begin{center}
		\includegraphics[width=0.5\textwidth]{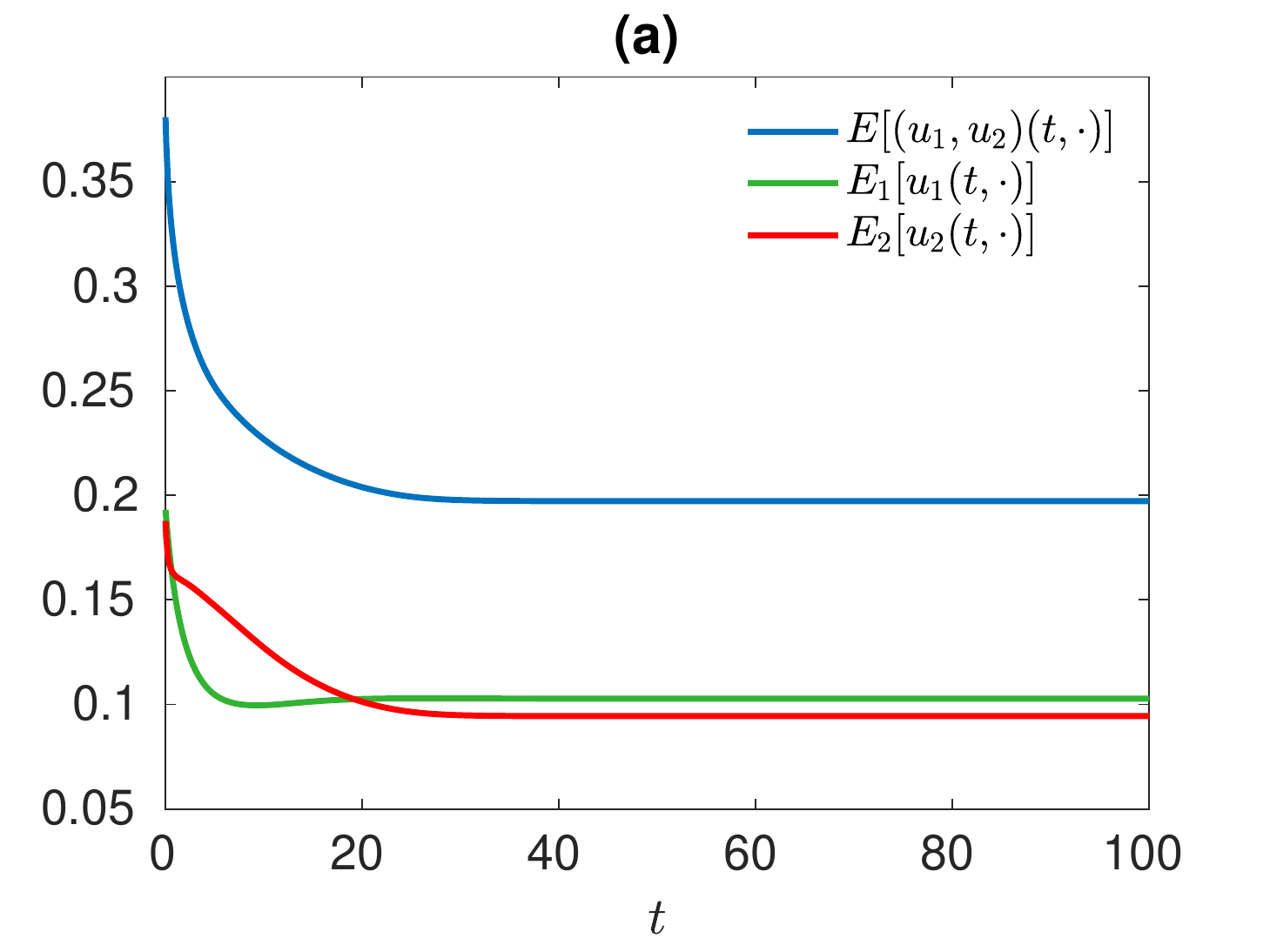}\includegraphics[width=0.5\textwidth]{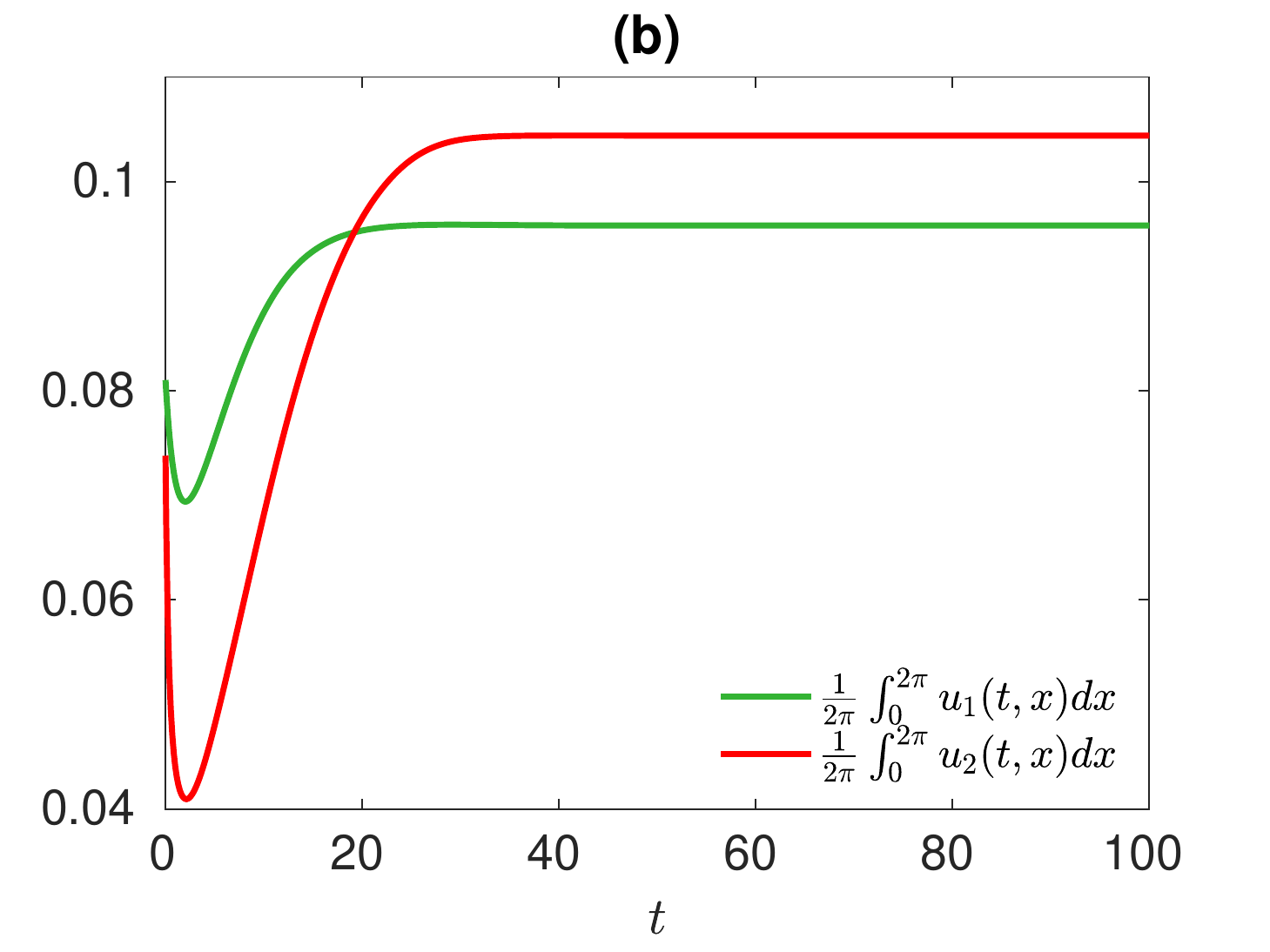}
	\end{center}
	\caption{\textit{Figure (a) is the plot of the curves $ t\longmapsto E_i[u_i(t,\cdot)],\;i=1,2, $ (green and red curves respectively) and $ t\longmapsto E[u_1,u_2)(t,\cdot)] $ (blue curve) under system \eqref{1.1} with reaction functions as \eqref{6.1.1} and kernel $ \rho $ as Gaussian in \eqref{6.1}.  We set our parameters as in \eqref{6.2}. Thus, one has $ r_1=r_2=0.2 $. We trace the curve $ t\longmapsto E[(u_1,u_2)(t,\cdot)] $ which is decreasing. We can see that the curve $ t\longmapsto E_1[u_1(t,\cdot)] $ is not monotone decreasing while $ t\longmapsto E_2[u_2(t,\cdot)] $ is monotone decreasing and they converge when $ t \to \infty$.  Figure (b) is the plot of total number of individuals for each species.}}
	\label{Figure2}
\end{figure}
From Figure \ref{Figure2}, we can see that the limit $ E_\infty $ exists and equals to $ r=0.2 $. From Theorem \ref{THM5.5} and Remark \ref{REM5.6}, the limit $ E_\infty=r $ implies
\[ (u_1+u_2)(t,x) \xrightarrow{L^1} r, \quad t\to \infty. \] 
Moreover, from the simulation we note that each limit $ E_{i,\infty}:=\lim_{t\to \infty}E_i[u_i(t,\cdot)] $ exists for $ i=1,2. $ 
From \eqref{6.1.0} we have 
\begin{equation}\label{6.3}
E_{i,\infty}=r\left(1-\frac{1}{|\T|}\int_{\T}c_i(x)dx\right),\quad i=1,2.
\end{equation}
By our simulation, we can see that $ E_{1,\infty},E_{2,\infty}\in (0,r) $ while $ E_{1,\infty}+E_{2,\infty}=r $, together with  equation \eqref{6.3} we can deduce $ c_1(x),c_2(x)\in(0,1),\; c_1(x)+c_2(x)=1$. Notice that $ c_1(x),c_2(x)\in (0,1)  $ implies the limits
\[ \lim_{n\to \infty} \delta_{u_i(t_n,x)}=(1-c_i(x))\delta_0+c_i(x)\delta_r,\quad i=1,2, \]
is not a single Dirac measure. Therefore, using Young measure and the weak compactness in $ Y(\T;[0,\gamma]) $ help us to understand the limit of the solution.

Now we plot the evolution of the two populations under system \eqref{1.1} in Figure \ref{Figure3}.
		\begin{figure}[H]
			\begin{center}
			\includegraphics[width=0.33\textwidth]{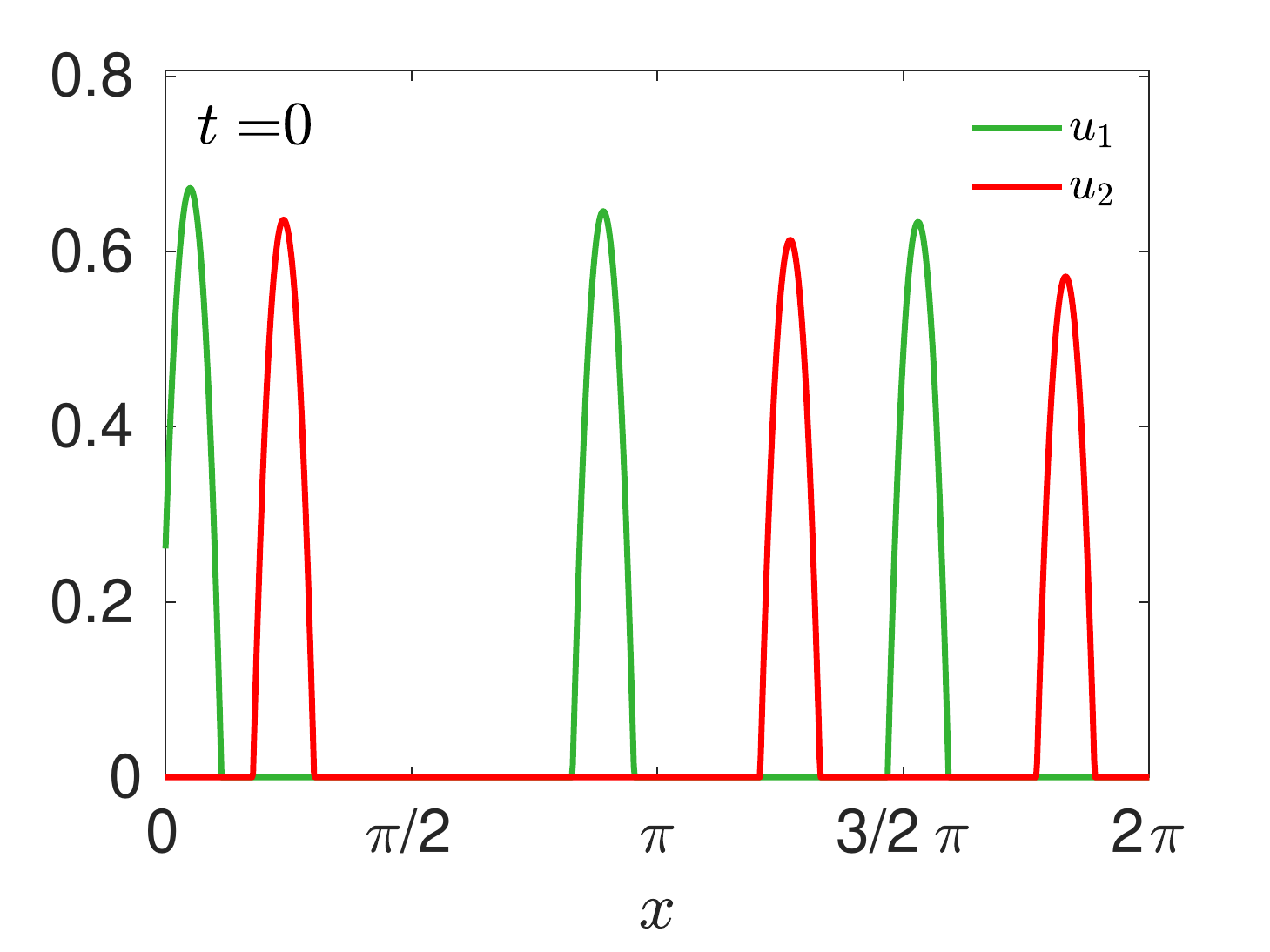}\includegraphics[width=0.33\textwidth]{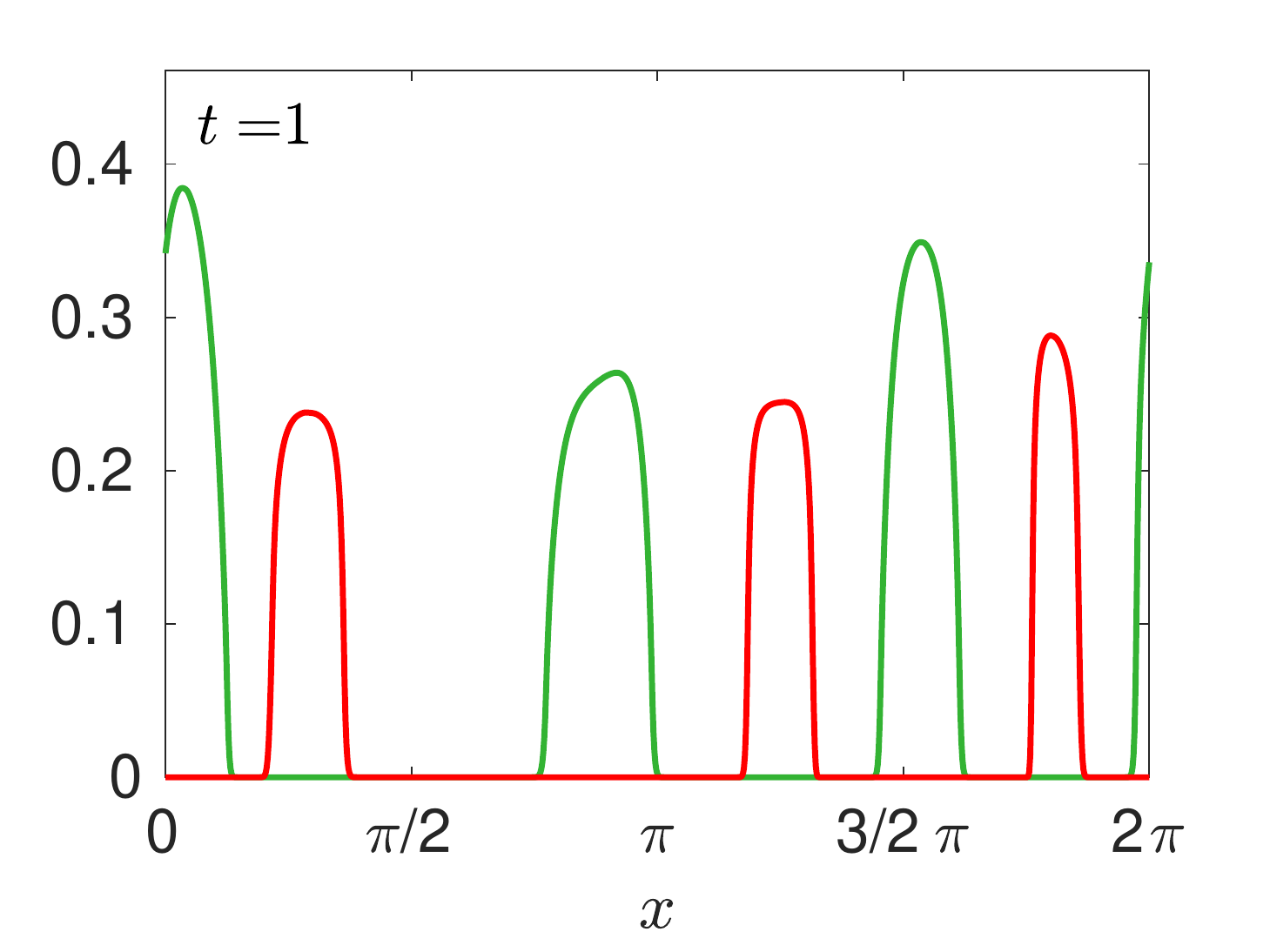}\includegraphics[width=0.33\textwidth]{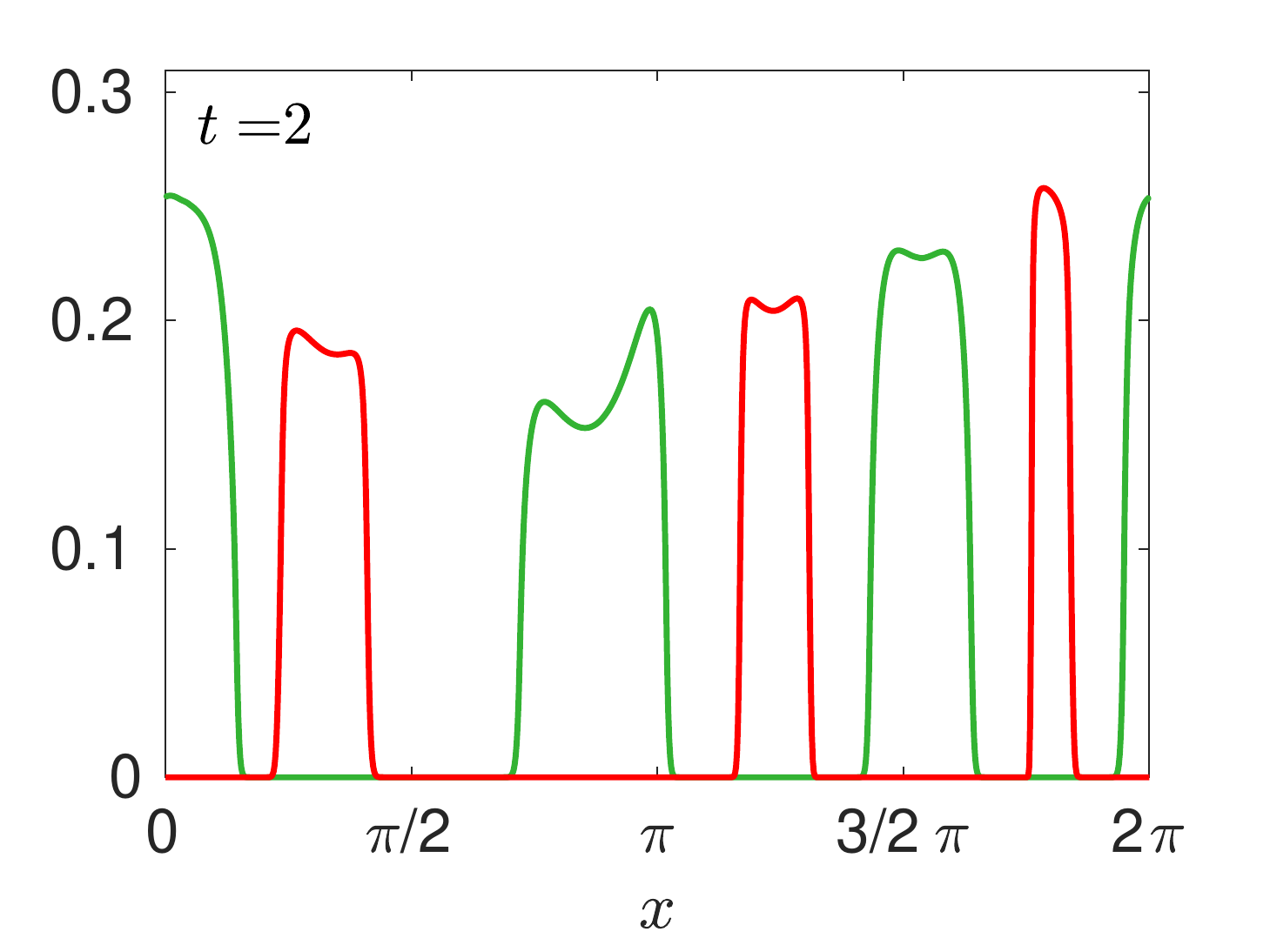}
			\includegraphics[width=0.33\textwidth]{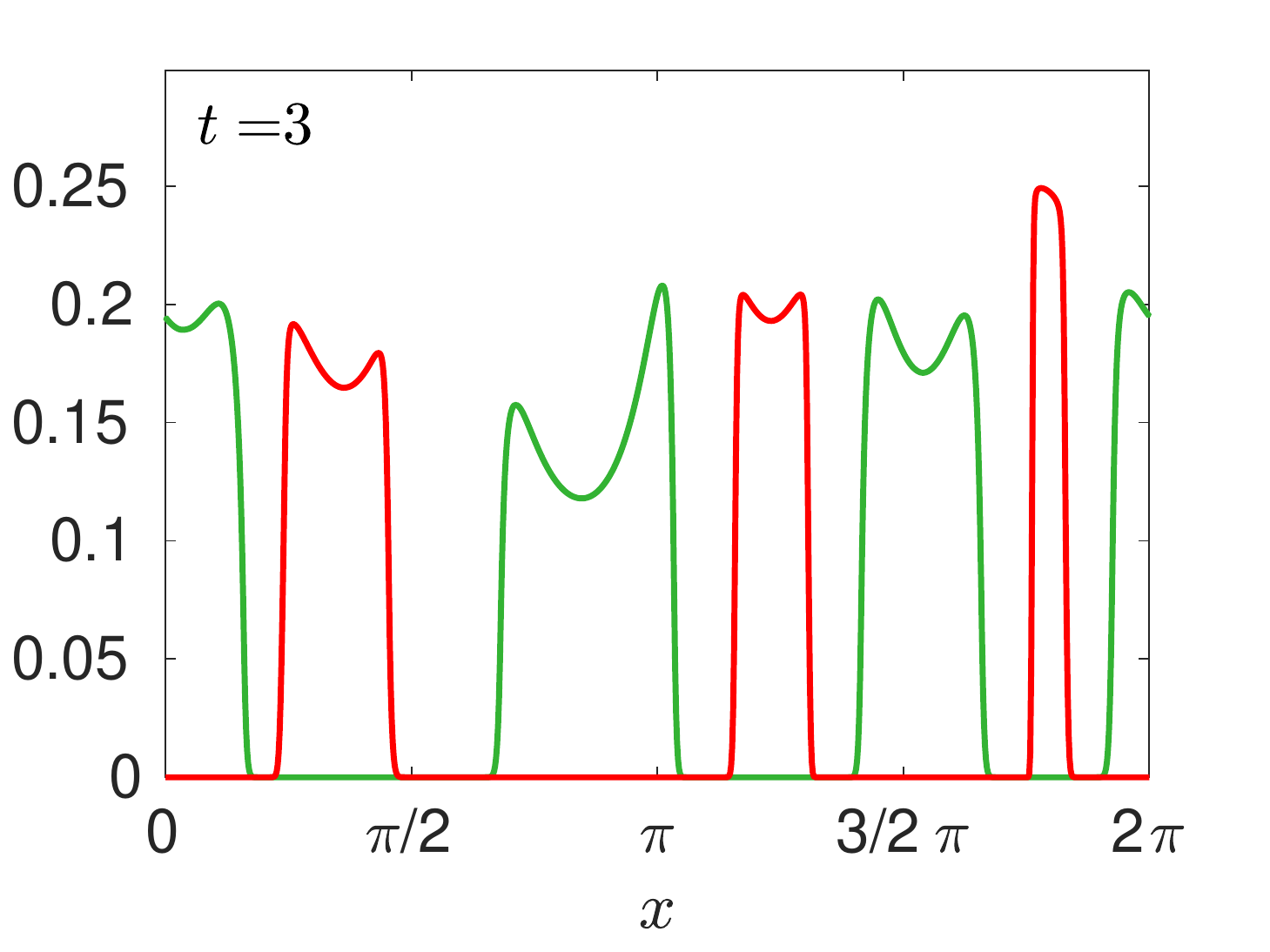}\includegraphics[width=0.33\textwidth]{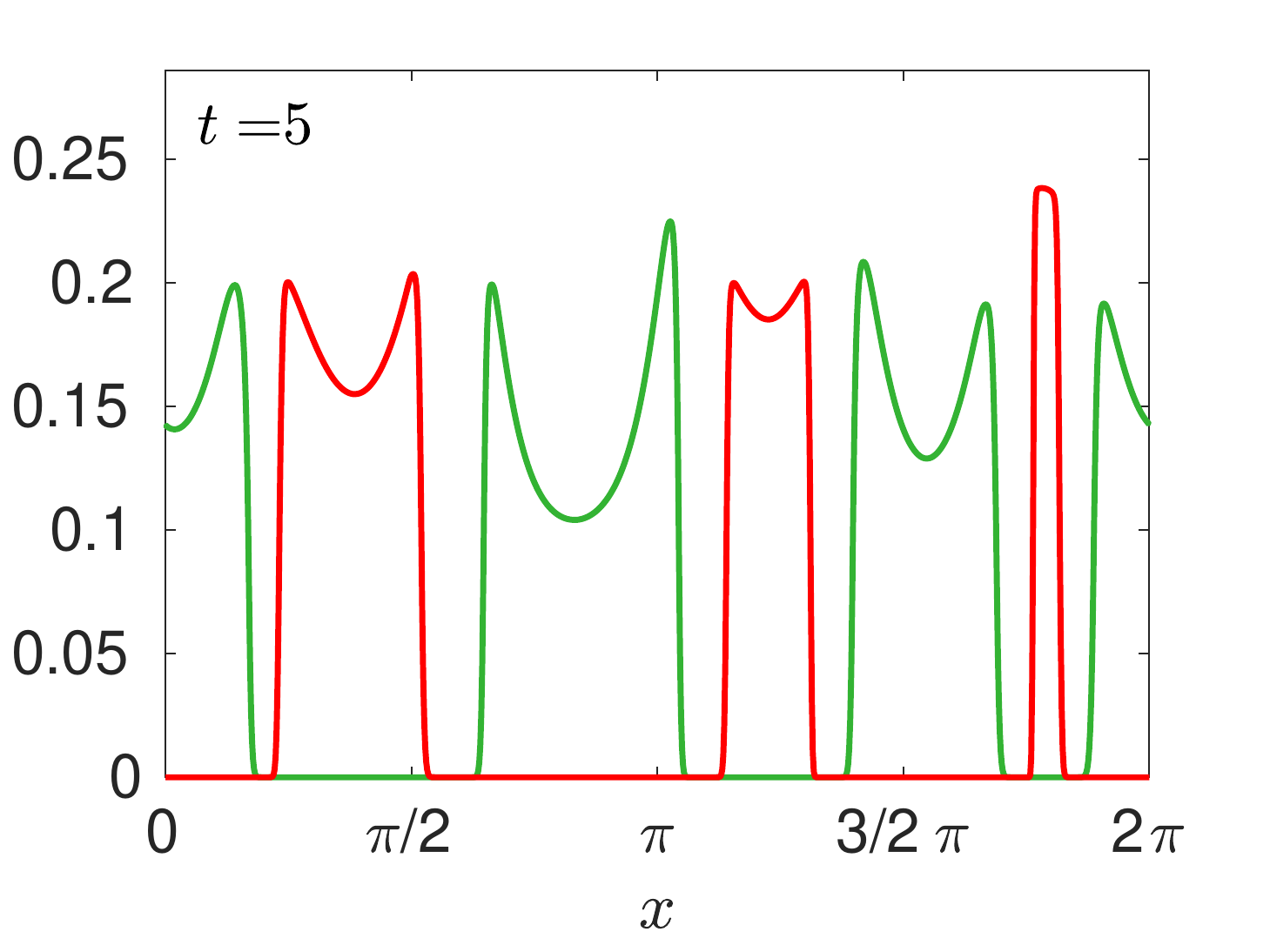}\includegraphics[width=0.33\textwidth]{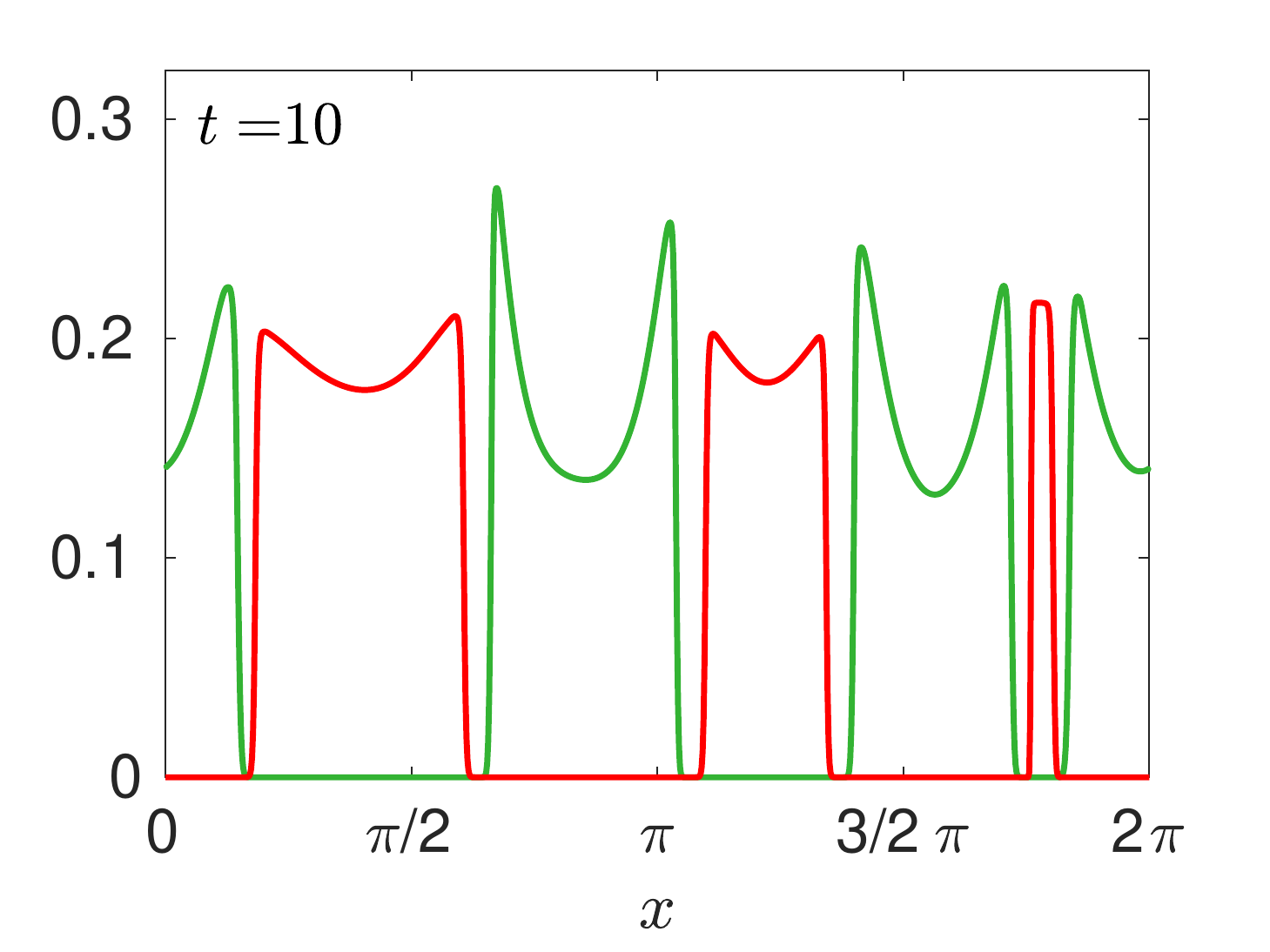}
			\includegraphics[width=0.33\textwidth]{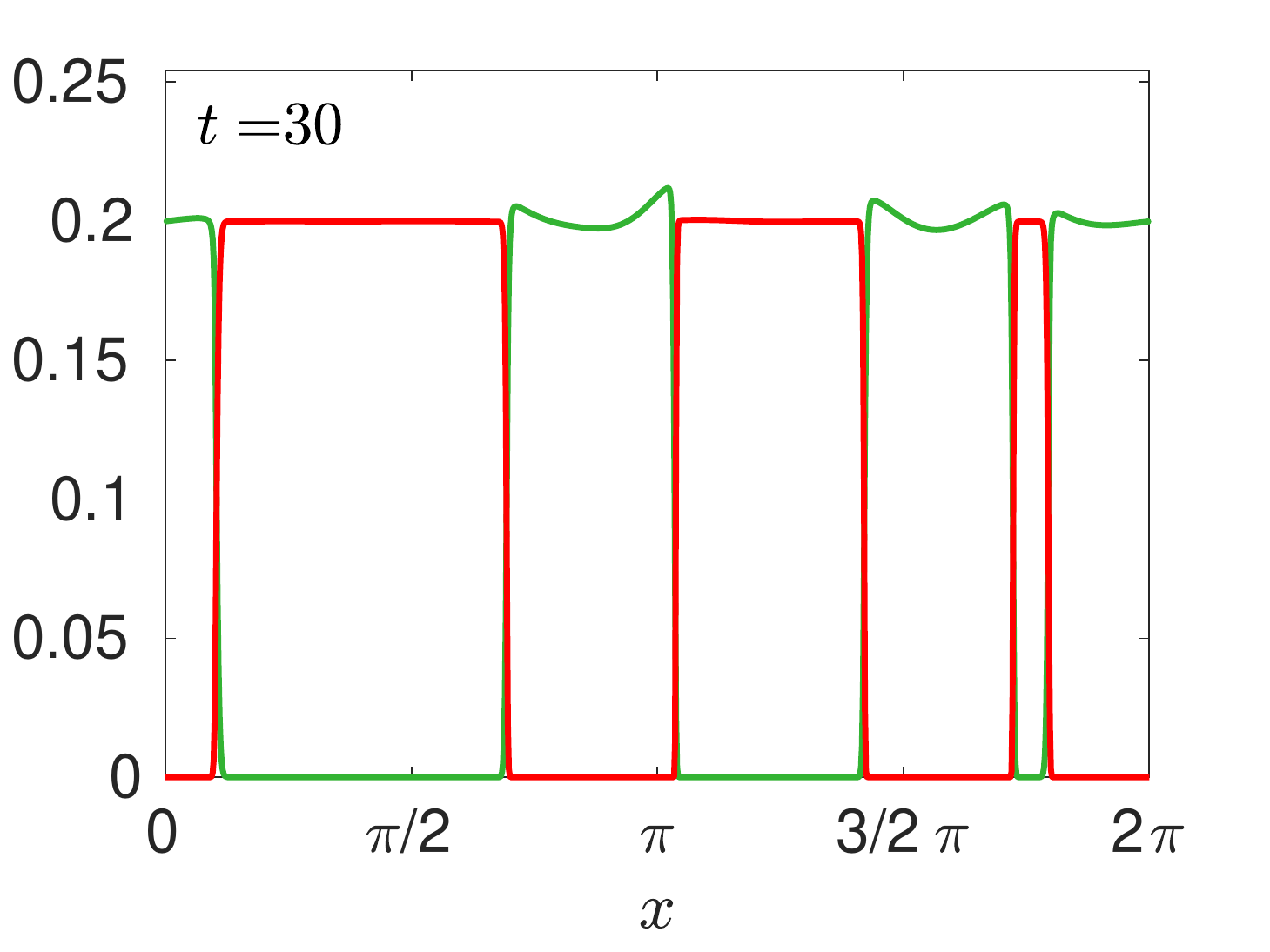}\includegraphics[width=0.33\textwidth]{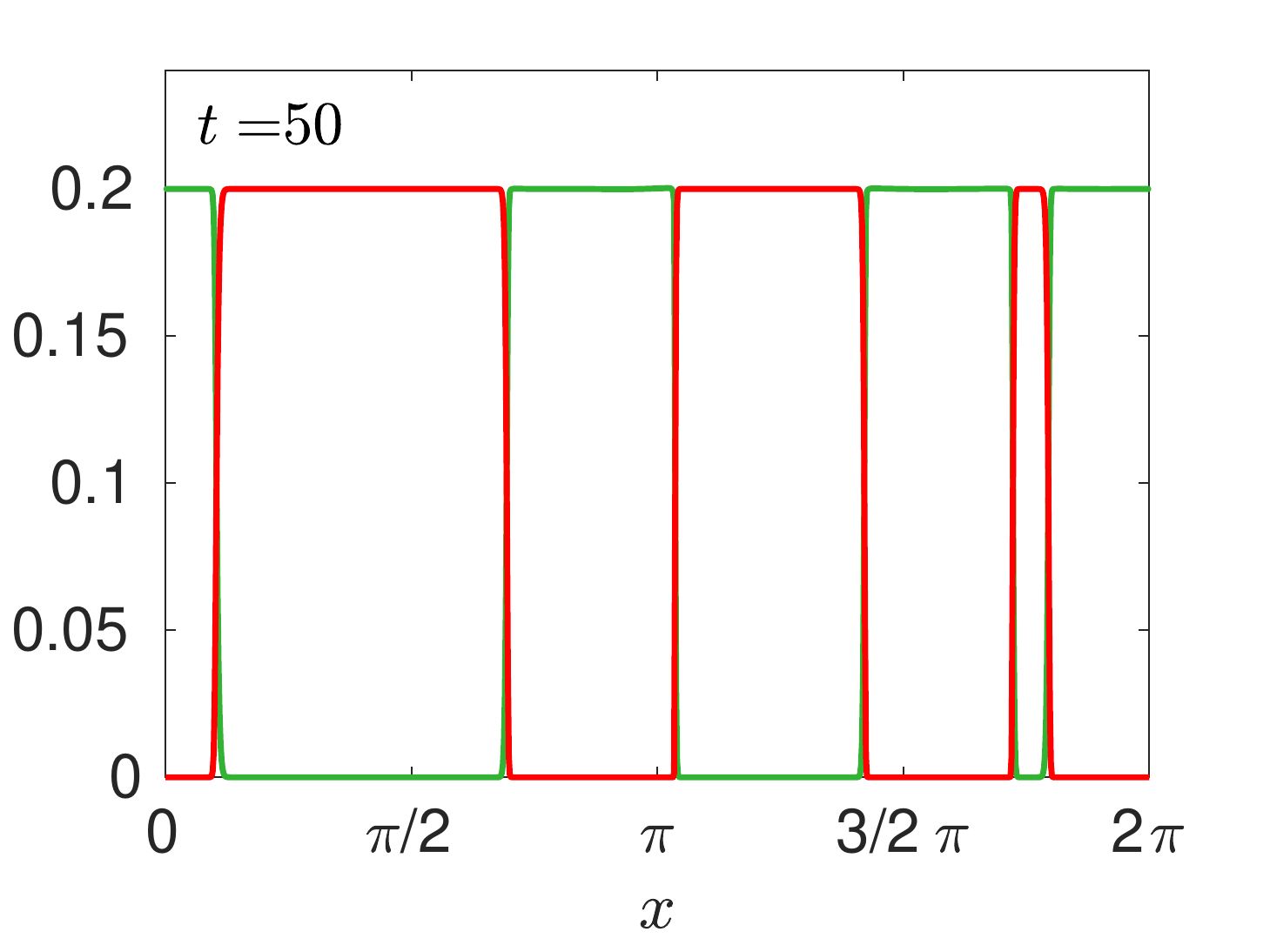}\includegraphics[width=0.33\textwidth]{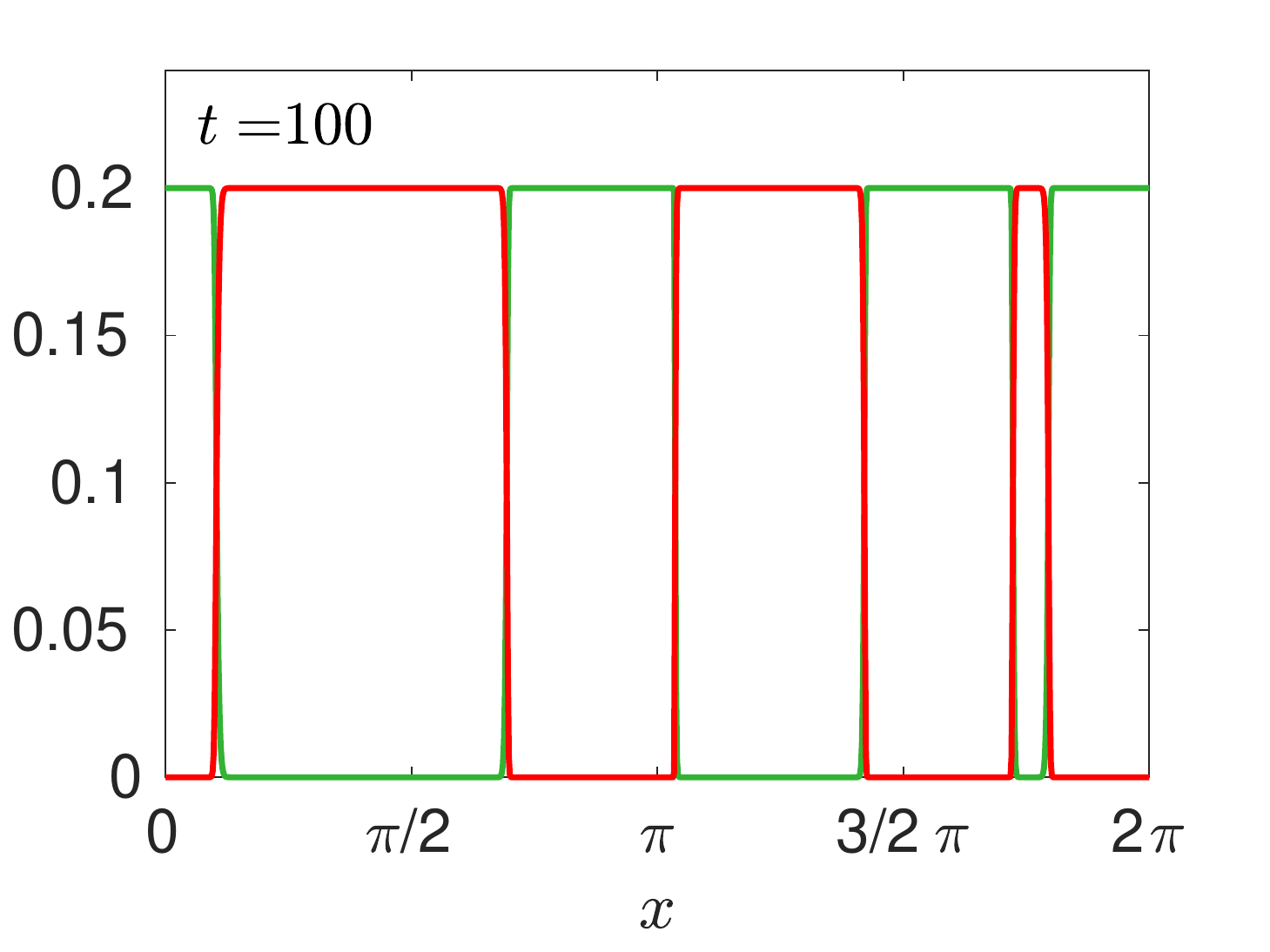}
			\end{center}
			\caption{\textit{The simulation of system \eqref{1.1} with reaction functions as \eqref{6.1.1} and kernel $ \rho $ as Gaussian in \eqref{6.1}. The green curves represent species $ u_1 $, the red represents species $ u_2 $. We set our parameters as in \eqref{6.2}. Thus, one has $ r_1=r_2=0.2 $.  As we have $ r_1=r_2 $, we see the coexistence of the two species and the segregation property and after $ t=100  $ the distributions of the two species stay the same.  }}
			\label{Figure3}
		\end{figure}
		
For the asymptotic behavior of the population, we can see from Figure \ref{Figure3} that the sum of two species $ u_1+u_2 $ reaches a steady state at $ t=100 $. From the pattern at each moment $ t $, we can see two species keep segregated in stead of being mixed. This result is due to our nonlocal advection term which ensures that the propagation speed is finite, which captures the "islets"  phenomenon in the real biological experiments in dimension two (see Figure 1. in \cite{Pasquier2012}).

\subsection{Initial location matters}
	Consider two different initial distributions $ \mathbf{u_0}=(u_1(0,x),u_2(0,x) )$ and $ \mathbf{\tilde{u}_0}=(\tilde{u}_1(0,x),\tilde{u}_2(0,x)) $ and assume their $ L^1 $ norms are the same,  that is
	\[ \int_{\T}u_i(0,x)dx= \int_{\T}\tilde{u}_i(0,x)dx,\quad i=1,2. \] 
	Under the same set of parameters, one may ask
	whether the limits 
	\begin{equation}\label{6.4}
	U_{i,\infty}:=\lim_{t\to \infty} \dfrac{1}{|\T|}\int_{\T}u_i(t,x)dx,\quad \tilde{U}_{i,\infty}:=\lim_{t\to \infty} \dfrac{1}{|\T|}\int_{\T}\tilde{u}_i(t,x)dx, \quad i=1,2.
	\end{equation}
for each species $ i=1,2 $ will be the same or not.

	In the real biological experiments, this situation corresponds to the case where the researchers use the same quantity of cells for each species for two separate petri dishes. Supposing the intrinsic mechanisms of cell population for these two groups are the same, the only difference is the initial cell distributions in two petri dishes. We are interested in whether the final total mass for each population are the same.
	
	Before our simulation, we give two different initial distributions $\mathbf{u_0}=(u_1(0,x),u_2(0,x) ), \mathbf{\tilde{u}_0}=(\tilde{u}_1(0,x),\tilde{u}_2(0,x)) $ as follows.
\begin{figure}[H]
	\begin{center}
		\includegraphics[width=0.5\textwidth]{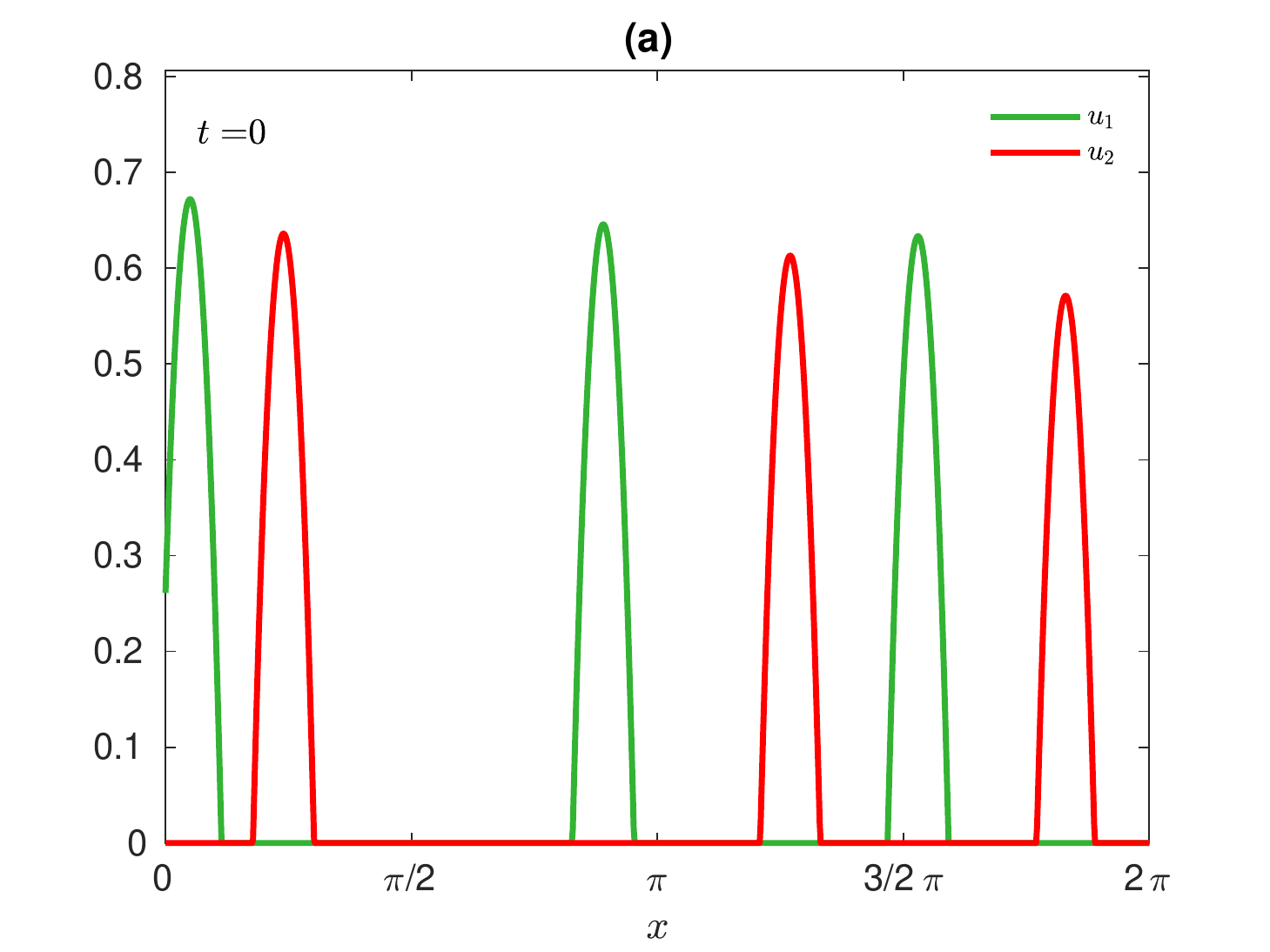}\includegraphics[width=0.5\textwidth]{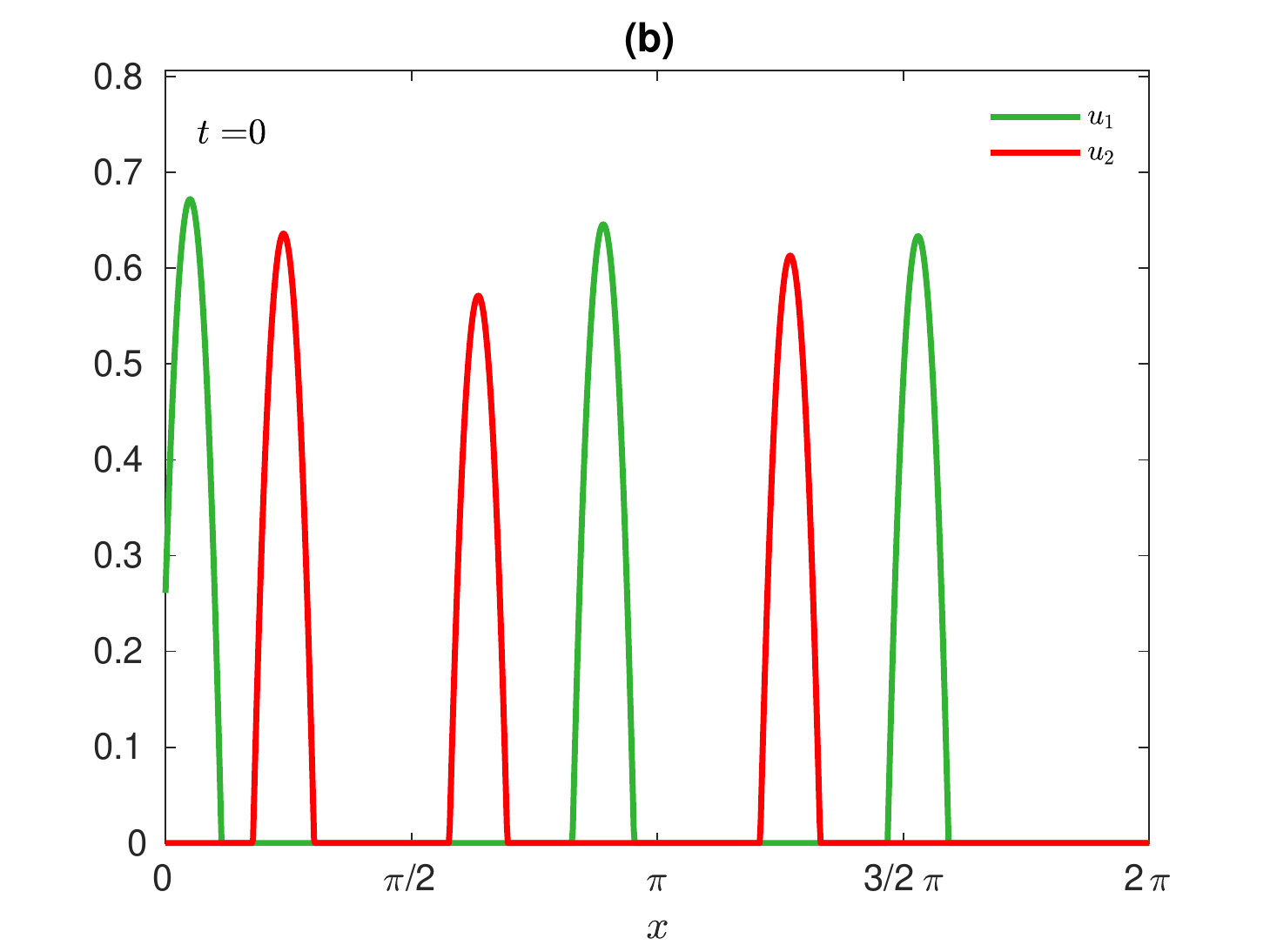}
	\end{center}
	\caption{\textit{Figures (a) and (b) correspond respectively to the initial distributions $\mathbf{u_0} $ and  $ \mathbf{\tilde{u}_0}$. In Figure (a), we shift the population of $ u_2 $ (red curve) at position in between $ 3/2\pi $ and $ 2\pi $ to  position in between $ \pi/2 $ to $ \pi $. Therefore, the number individuals for each species  is conserved.  }}
	\label{Figure4}
\end{figure}	
In Figure \ref{Figure6}, we plot the energy functional and the number of individuals corresponding to each initial distribution in Figure \ref{Figure4}. 
	\begin{figure}[H]
		\begin{center}
			\includegraphics[width=0.5\textwidth]{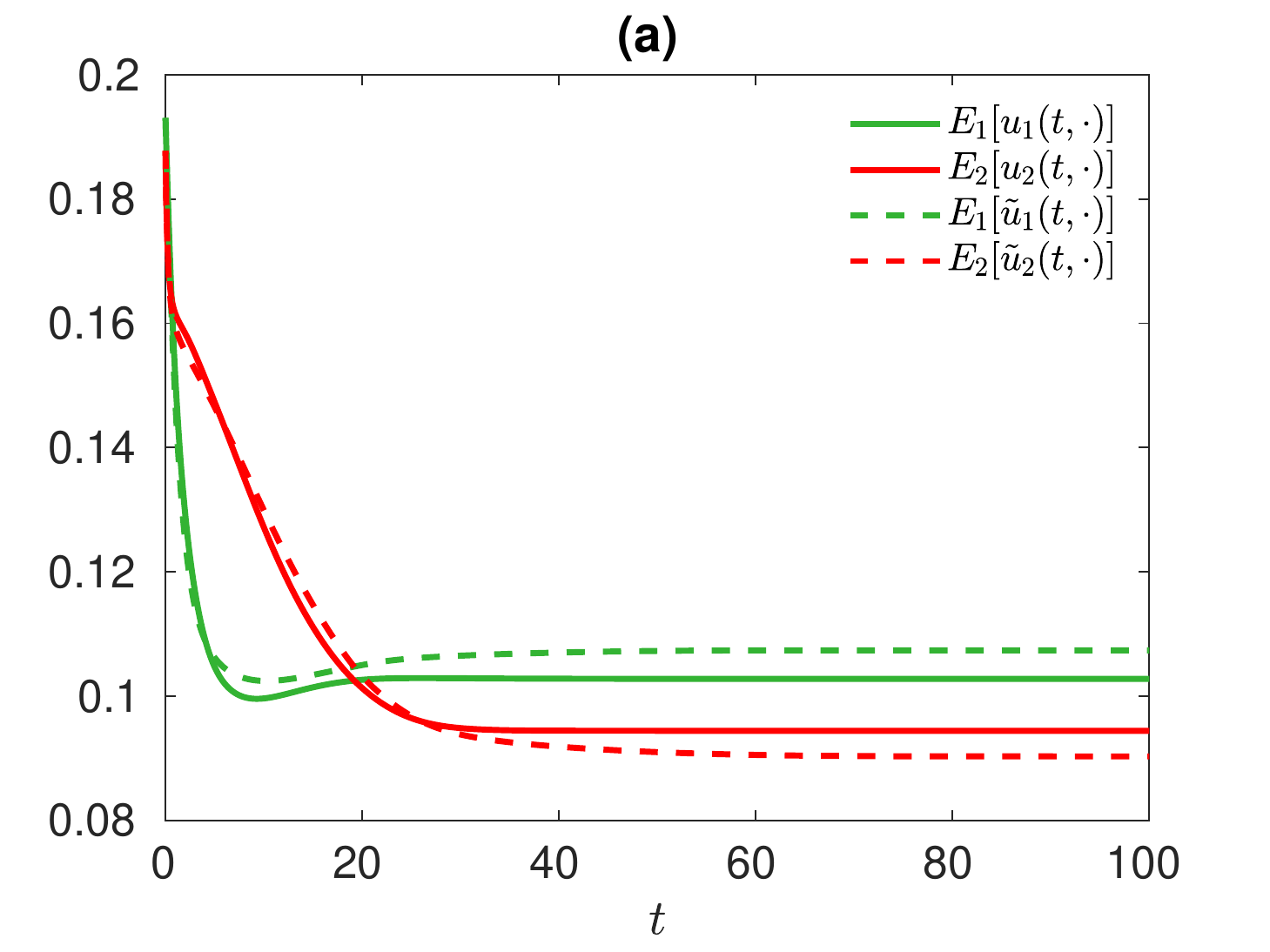}\includegraphics[width=0.5\textwidth]{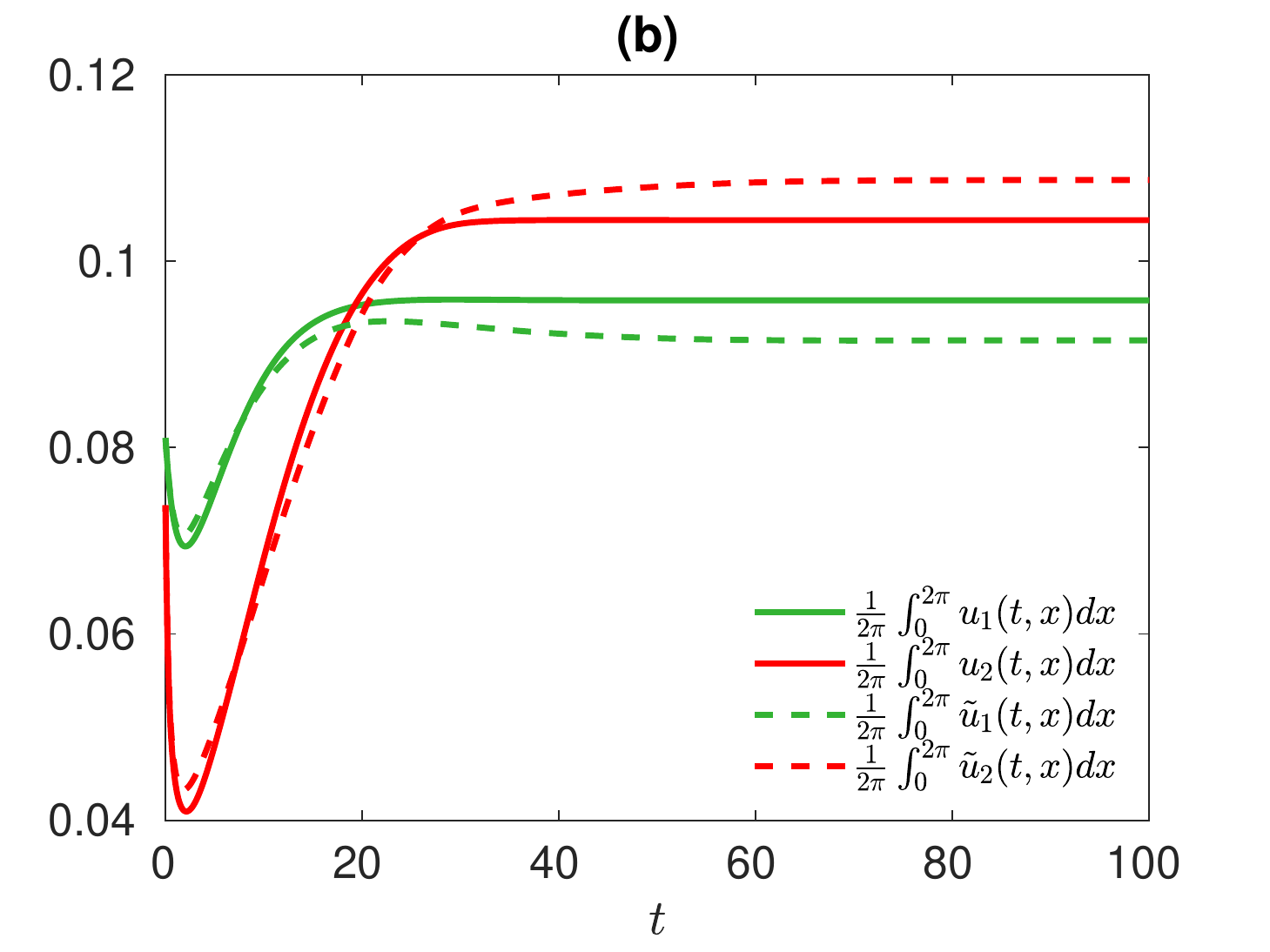}
		\end{center}
		\caption{\textit{In this figure we plot the energy functional $ t\mapsto E_i[u_i(t,.)],\;t\mapsto E_i[\tilde{u}_i(t,.)],\;i=1,2 $ (see Figure (a)) and the mean value of individuals $ t\mapsto \frac{1}{2\pi}\int_{0}^{2\pi}u_i(t,x)dx,\;t\mapsto \frac{1}{2\pi}\int_{0}^{2\pi}\tilde{u}_i(t,x)dx,\;i=1,2 $ (see Figure (b)) corresponding to two sets of different initial distributions in Figure \ref{Figure4}. 
	The green curves represent the species $ u_1 $ while the red curves represent the species $ u_2 $. The dashed lines correspond to the simulation with initial distribution as in Figure \ref{Figure4} (a)  and solide lines correspond to initial distribution as in Figure \ref{Figure4} (b). The parameters are the same as in \eqref{6.2}.}}
		\label{Figure6}
	\end{figure}
	Since the limits $ U_{i,\infty} $ and $ \tilde{U}_{i,\infty} $ have a significant difference from Figure \ref{Figure6} (b), thus we conclude the final total mass depends on the position of the initial value. 	
	Now we give the evolution of the two populations under system \eqref{1.1}.
\begin{figure}[H]
	\begin{center}
		\includegraphics[width=0.33\textwidth]{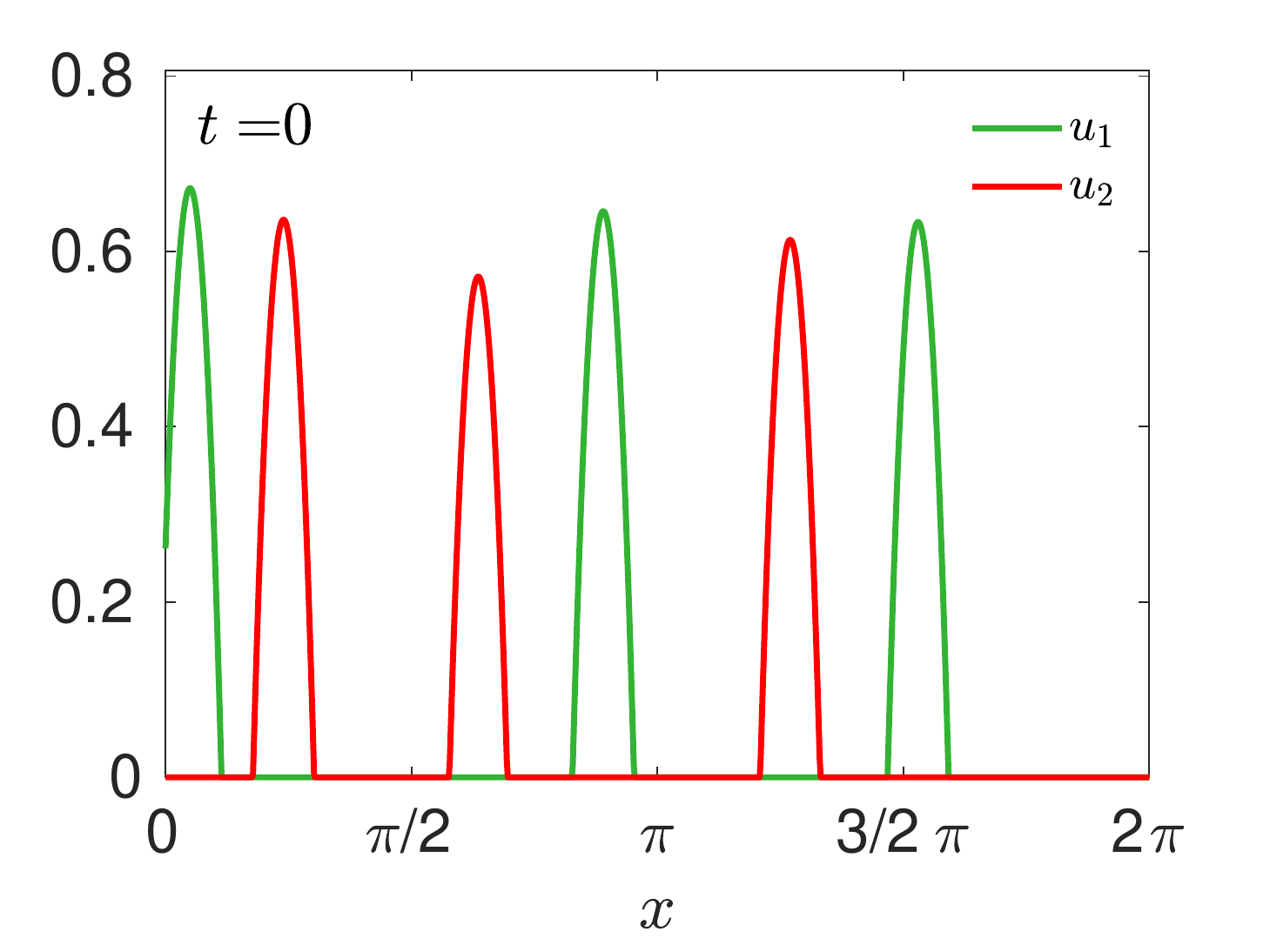}\includegraphics[width=0.33\textwidth]{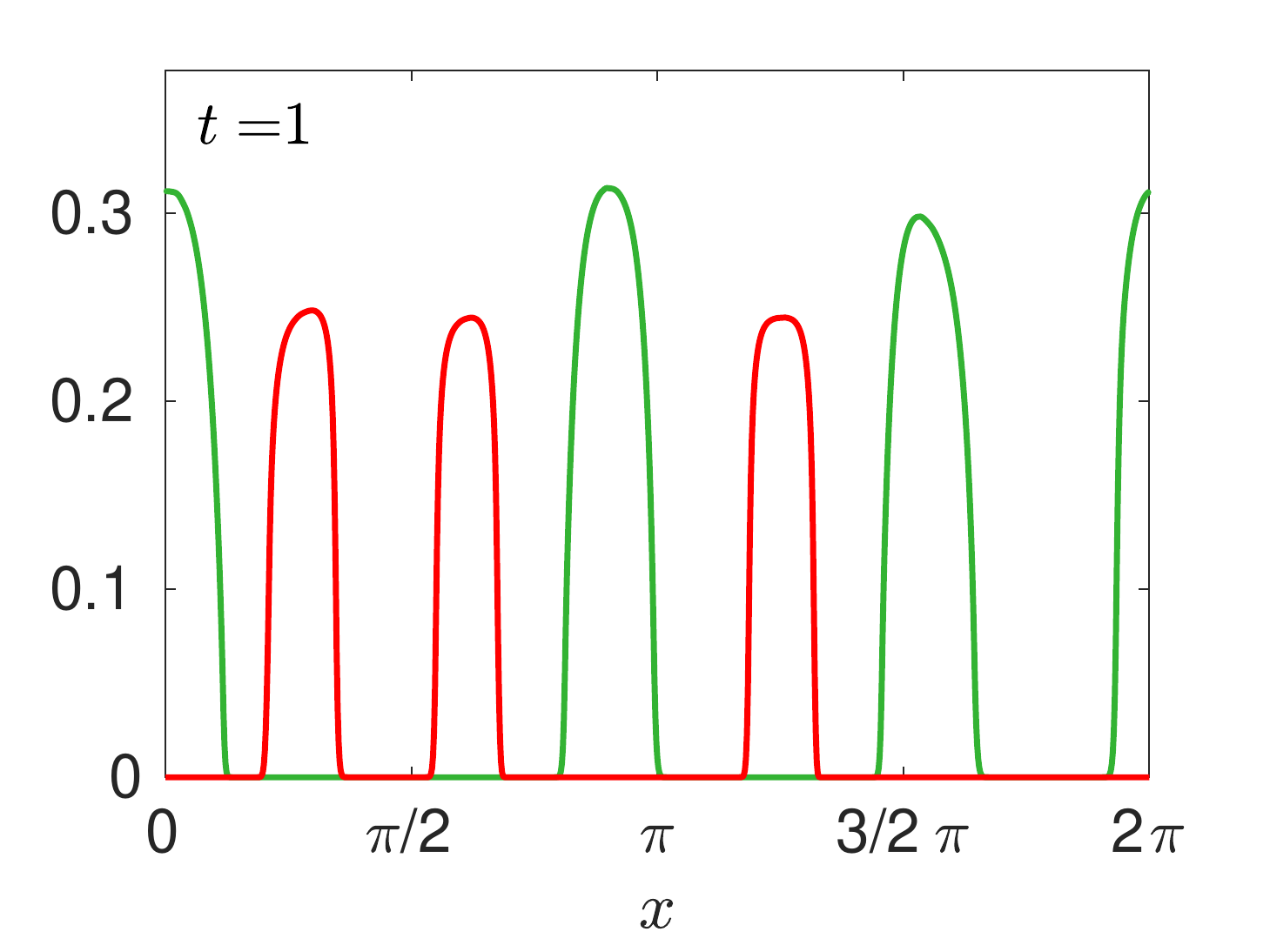}\includegraphics[width=0.33\textwidth]{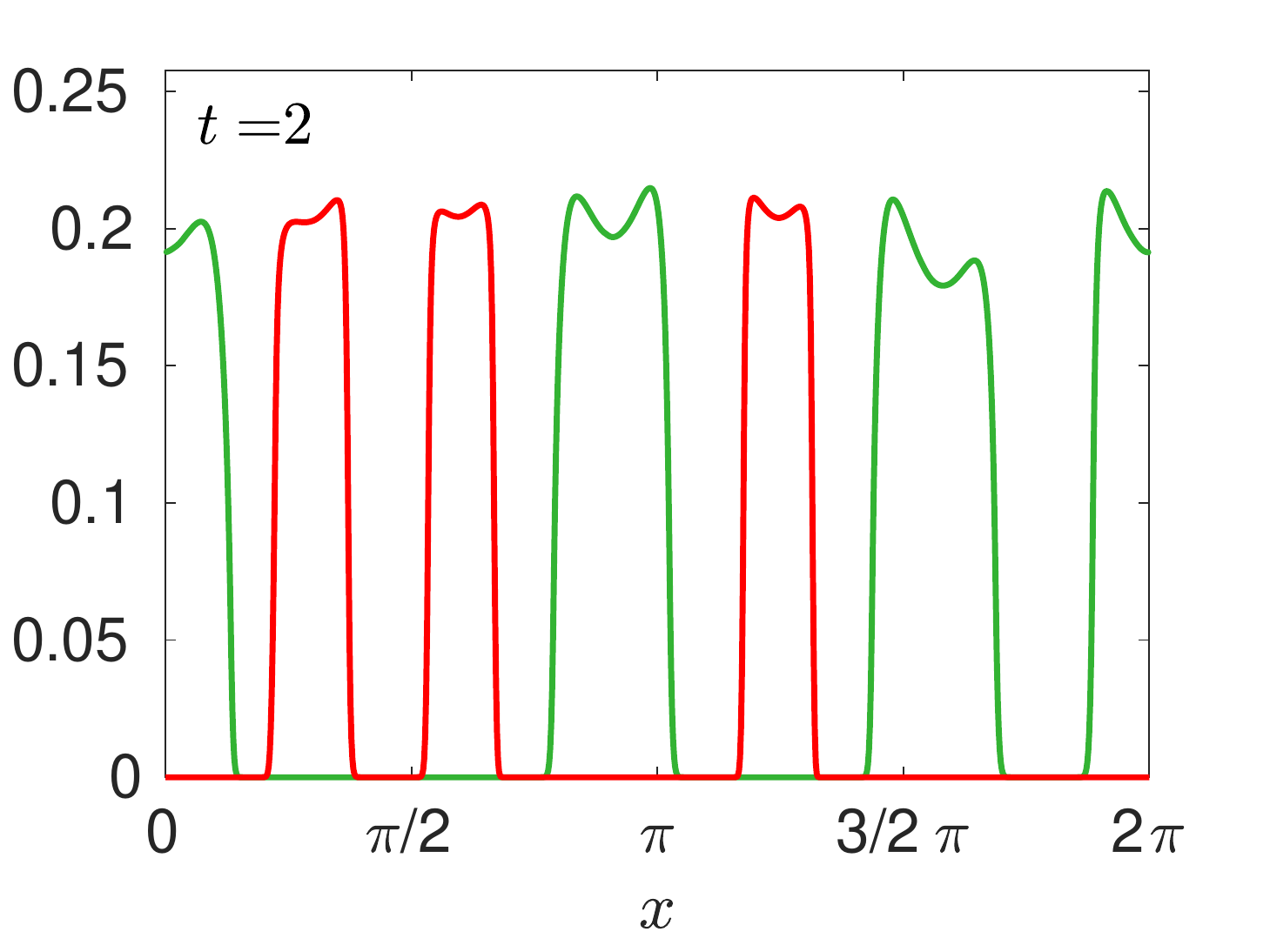}
		\includegraphics[width=0.33\textwidth]{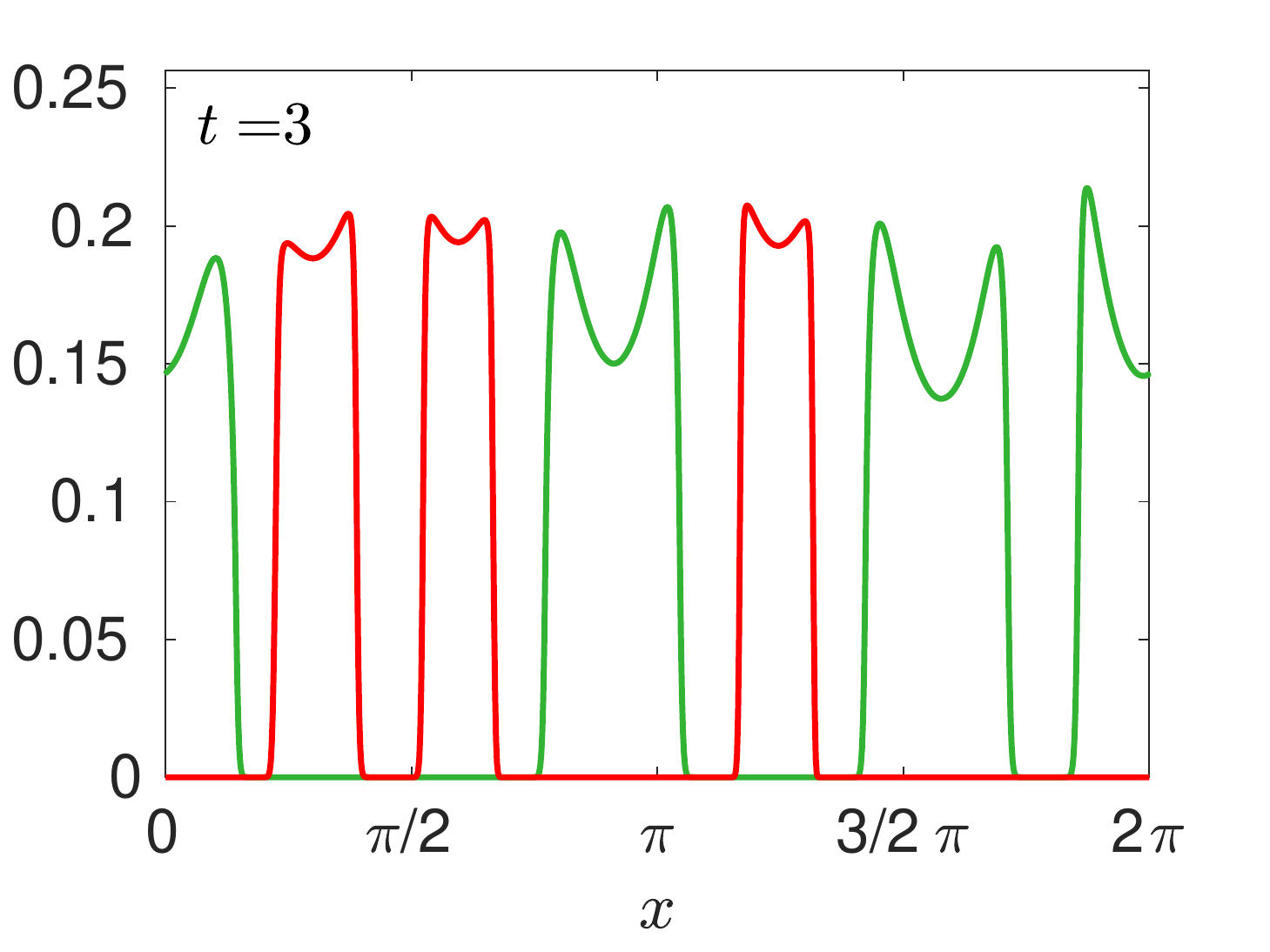}\includegraphics[width=0.33\textwidth]{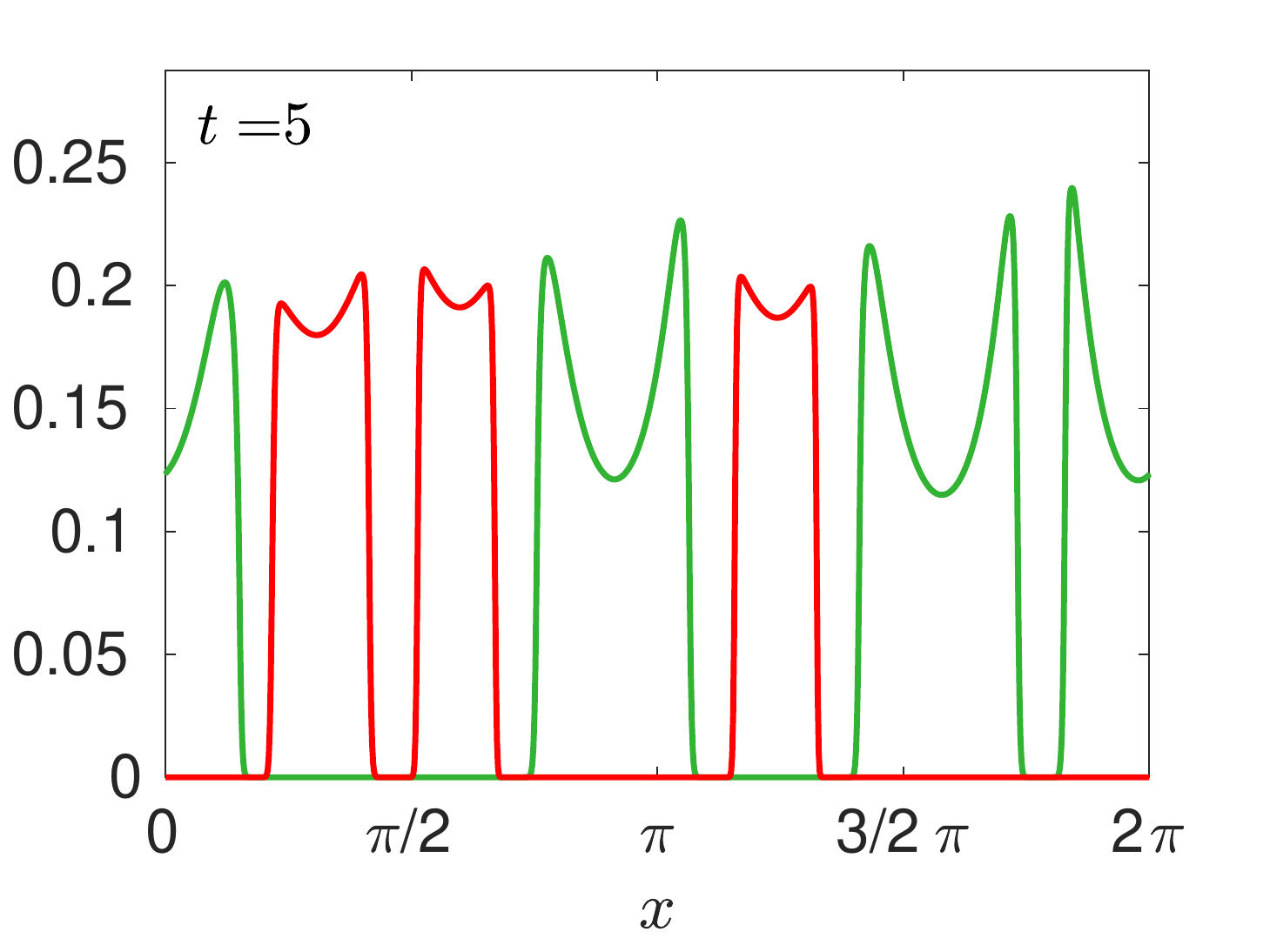}\includegraphics[width=0.33\textwidth]{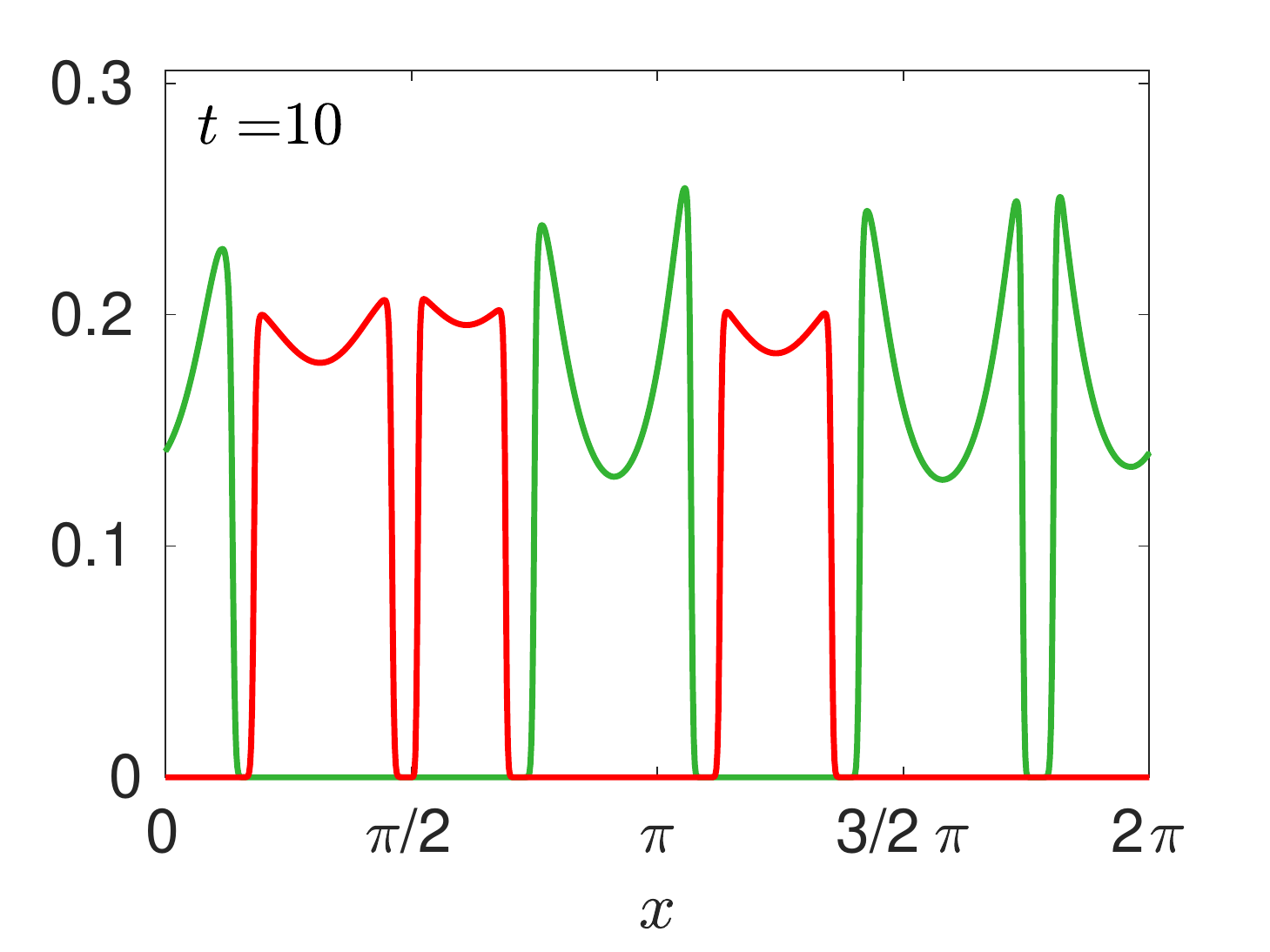}
		\includegraphics[width=0.33\textwidth]{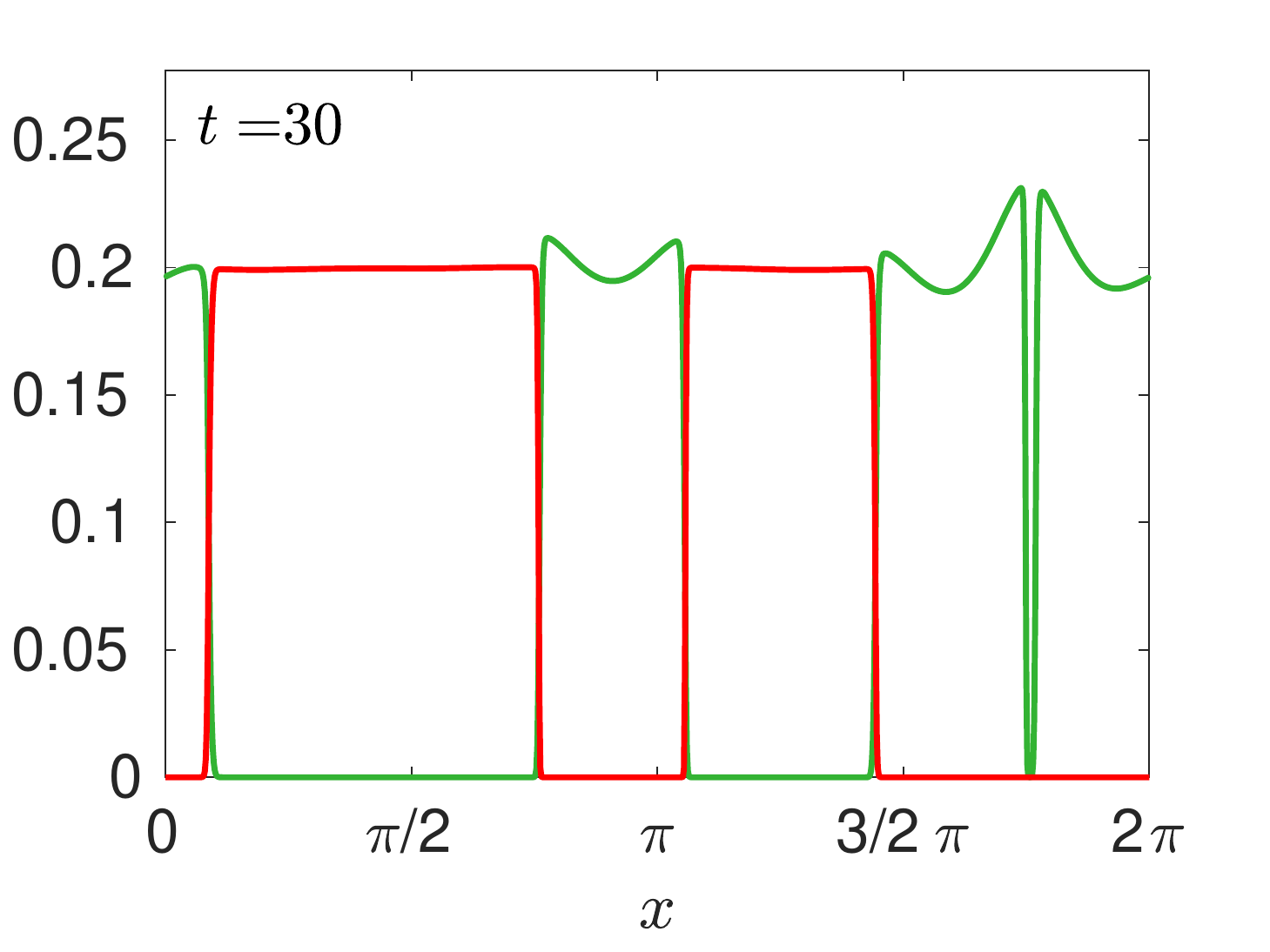}\includegraphics[width=0.33\textwidth]{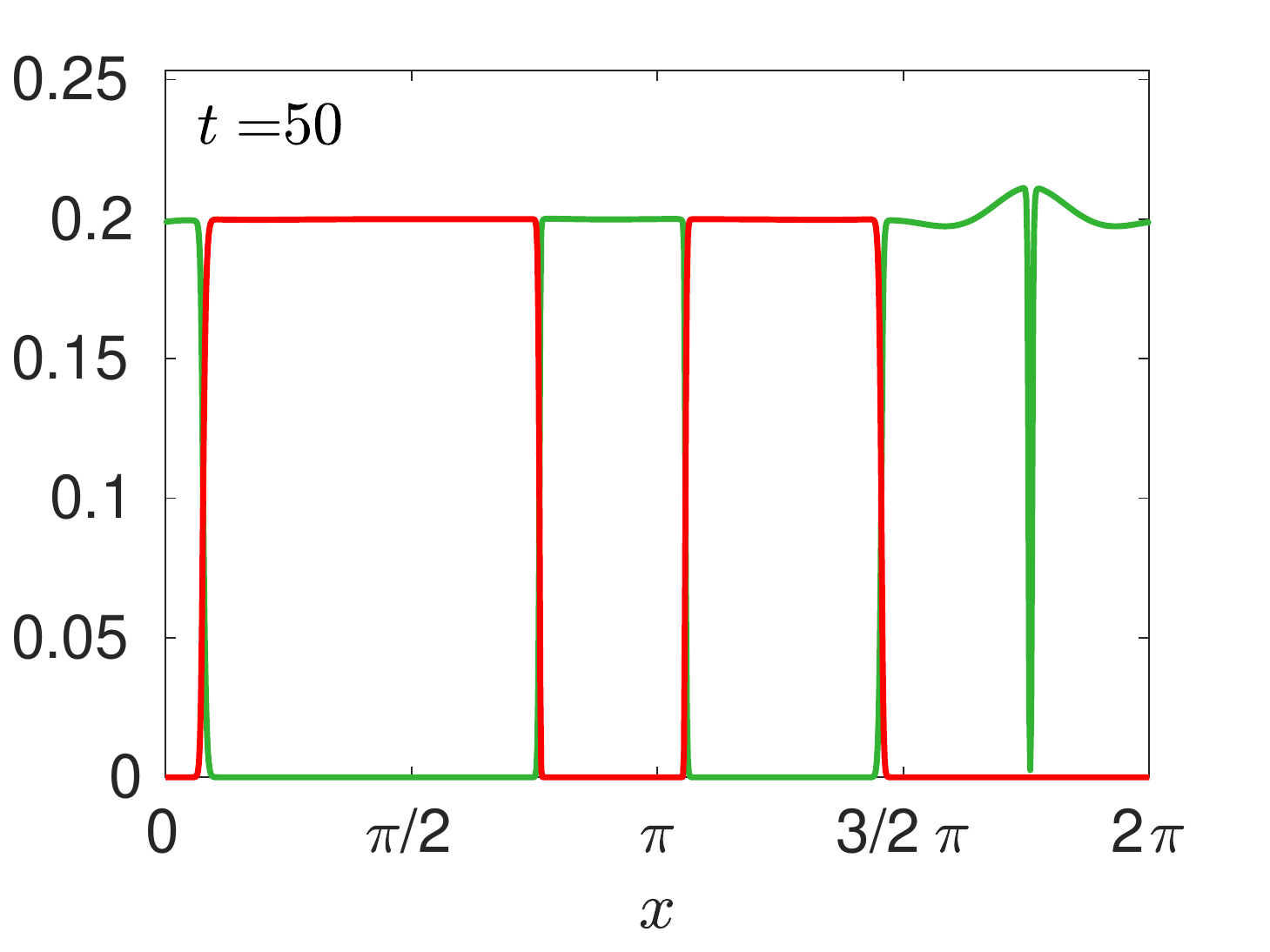}\includegraphics[width=0.33\textwidth]{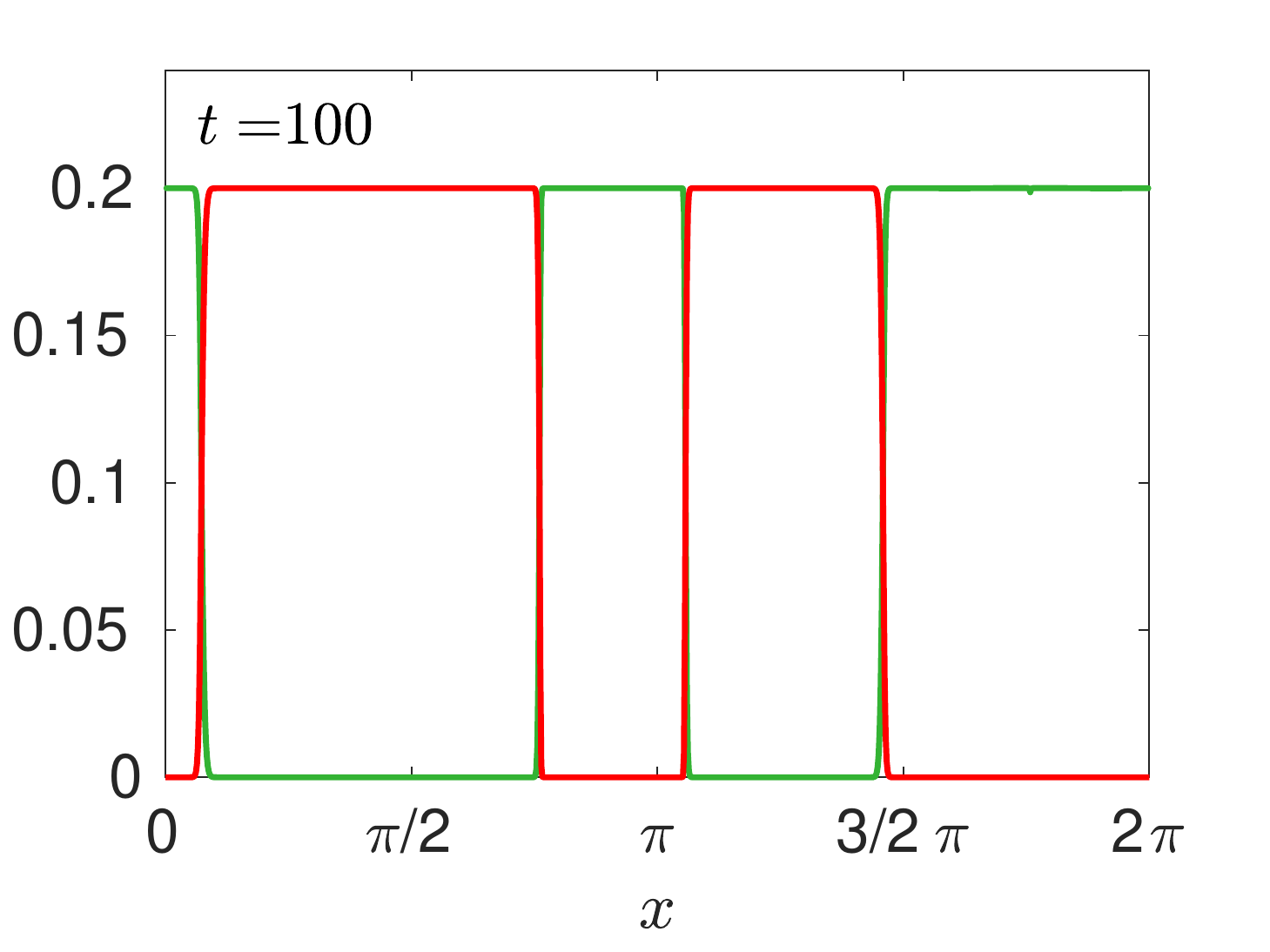}
	\end{center}
	\caption{\textit{The simulation of system \eqref{1.1} with reaction functions as \eqref{6.1.1} and kernel $ \rho $ as Gaussian in \eqref{6.1}. The green curves represent species $ u_1 $, the red represents species $ u_2 $. We set our parameters as in \eqref{6.2}. Thus, one has $ r_1=r_2=0.2 $ which implies the coexistence of the two species and after $ t=100  $ the distributions of the two species stay the same.  }}
	\label{Figure5}
\end{figure}
	As for the simulation in Figure \ref{Figure5}, we can see the same coexistence as in Figure \ref{Figure3} and the sum of the two populations
	\[ (u_1+u_2)(t,x) \xrightarrow{L^1} r, \quad t\to \infty. \] 
	However, the final patterns of two species at $ t=100 $ in Figure \ref{Figure5}. (i) and Figure \ref{Figure3}. (i)  are evidently different.

\subsection{The case $ r_1\neq r_2 $ implies exclusion principle}
Our second scenario complements the results in Theorem \ref{THM5.5}, without loss of generality we allow $ r_1 >r_2 $. This means species $ u_1 $ is favored in the environment. We call this scenario the exclusion principle. Our parameters for the reaction functions \eqref{6.1.1} are given as
\begin{equation}\label{6.6}
b_1=1.5,\;b_2=1.2,\;\mu=1,\;\gamma=1,\;K=0.2.
\end{equation}
therefore we can calculate 
\[r_1=0.5>r_2=0.2.   \]
As before, we trace the curve $ t\longmapsto E[(u_1,u_2)(t,\cdot)] $ in numerical simulation and  we also plot the curve $ t\longmapsto E_i[u_i(t,\cdot)],\;i=1,2, $ respectively. Moreover, we plot the variation of the mean value of the total number of individuals for each species.
\begin{figure}[H]
	\begin{center}
		\includegraphics[width=0.5\textwidth]{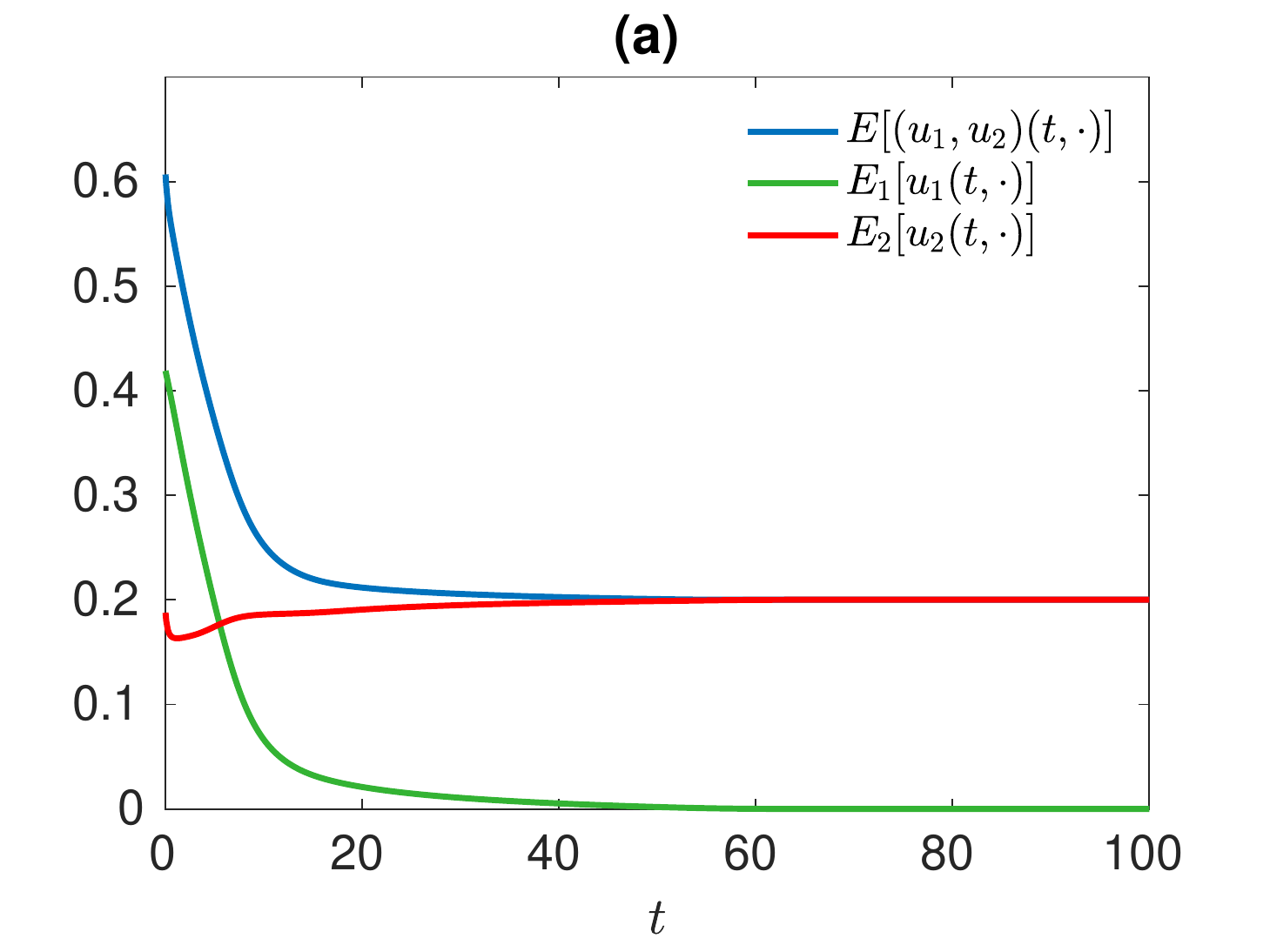}\includegraphics[width=0.5\textwidth]{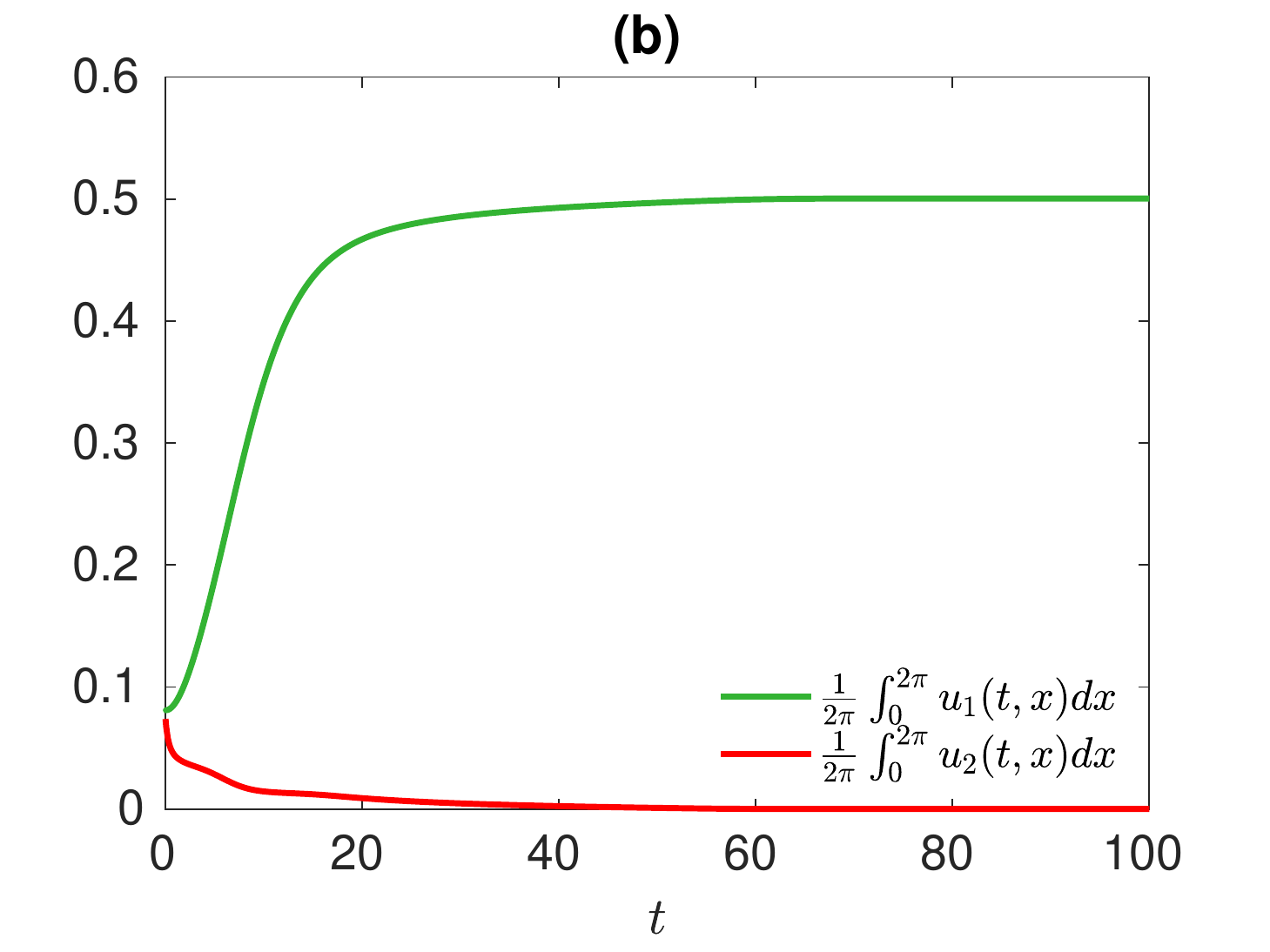}
	\end{center}
	\caption{\textit{
			Figure (a) is the plot of the curves $ t\longmapsto E_i[u_i(t,\cdot)],\;i=1,2, $ (green and red curves respectively) and $ t\longmapsto E[u_1,u_2)(t,\cdot)] $ (blue curve) under system \eqref{1.1} with reaction functions as \eqref{6.1.1} and kernel $ \rho $ as Gaussian in \eqref{6.1}.  We set our parameters as in \eqref{6.6}.  Thus, one has $ r_1=0.5>r_2=0.2 $. We trace the curve $ t\longmapsto E[(u_1,u_2)(t,\cdot)] $ which is decreasing. We can also see the curve $ t\longmapsto E_1[u_1(t,\cdot)] $ is decreasing while $ t\longmapsto E_2[u_2(t,\cdot)] $ is not monotone decreasing and their limits exist.  Figure (b) is the plot of mean value of individuals for each species.}}
	\label{Figure7}
\end{figure}
By tracing the curve $ t\longmapsto E[(u_1,u_2)(t,\cdot)] $, we can see from Figure \ref{Figure7} that it is strictly decreasing and it confirms again the result which has been proved in Theorem \ref{THM4.6}. We can also see that the curve $ t\longmapsto E_1[u_1(t,\cdot)] $ is decreasing while $ t\longmapsto E_2[u_2(t,\cdot)] $ is not monotone decreasing and their limits are 
\[ \lim_{t\to\infty}E_1[u_1(t,\cdot)] =0,\quad \lim_{t\to\infty}E_2[u_1(t,\cdot)] =r_2. \]
If we have $ E_{1,\infty}=0,E_{2,\infty}=r_2 $, since $ c_i(x)\in [0,1],\;a.e.\; x\in \T $ for $ i=1,2 $ and by equation \eqref{6.3} one obtains $ c_1(x)=1,\;c_2(x)=0 $. Therefore, we have $ c_1(x)+c_2(x)=1,\;a.e.\; x\in \T $ and the convergence in Theorem \ref{THM5.5} is in the sense of $ L^1 $ (see Remark \ref{REM5.6})
\[ u_1(t,x)\xrightarrow{L^1}r_1,\quad u_2(t,x)\xrightarrow{L^1}0,\;t\to \infty,  \]
and
\[ (u_1+u_2)(t,x)\xrightarrow{L^1}  r_1,\;t\to \infty. \]
This means if $ r_1>r_2 $ (resp. $ r_2>r_1 $), the species $ u_1 $ will exclude $ u_2 $ (resp. $ u_2 $ will exclude $ u_1 $) when $ t $ tends to infinity. Therefore, we can conclude the exclusion principle as in the beginning of this section. We plot the evolution of the solution as follows.
\begin{figure}[H]
	\begin{center}
	\includegraphics[width=0.33\textwidth]{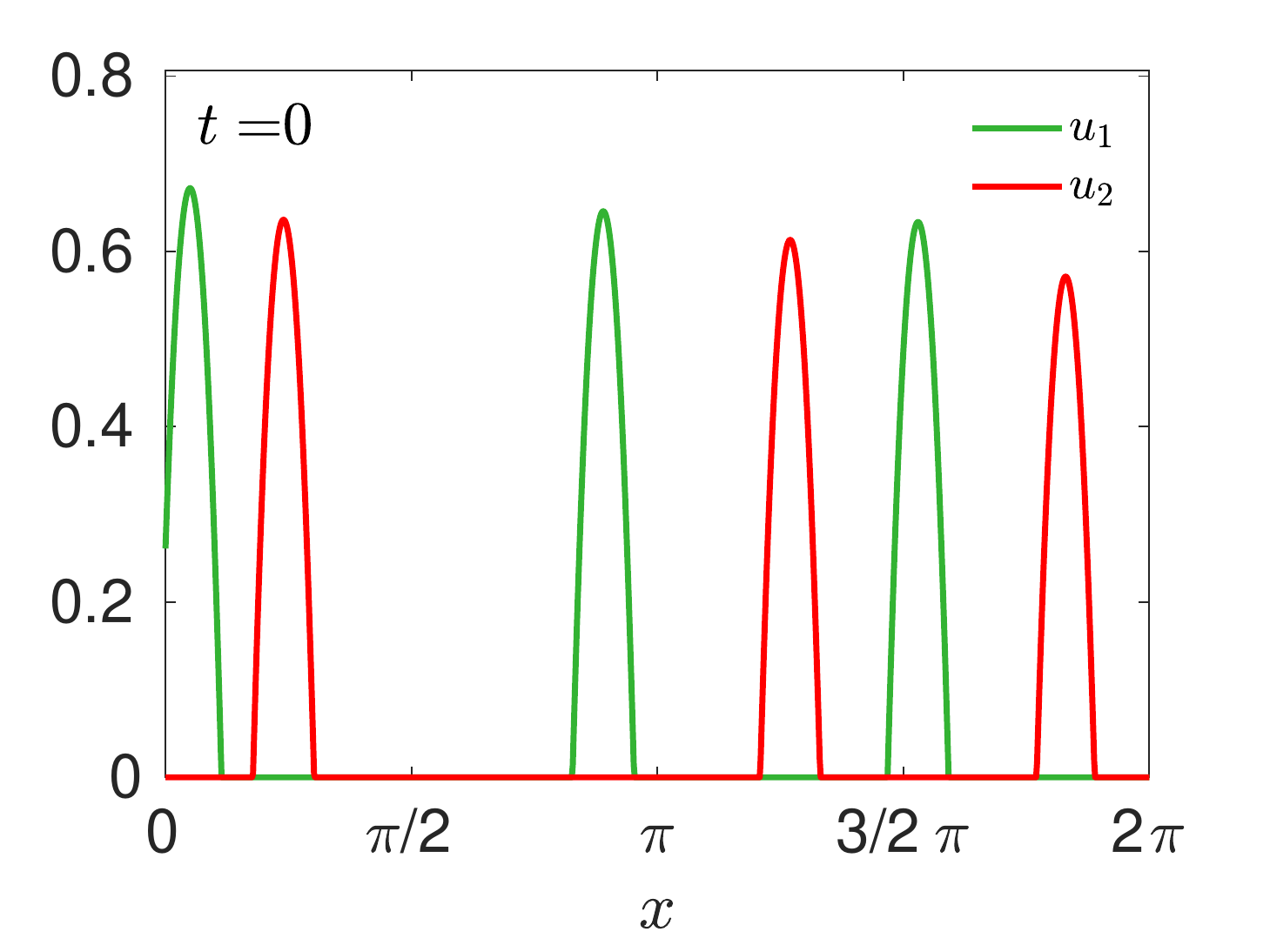}\includegraphics[width=0.33\textwidth]{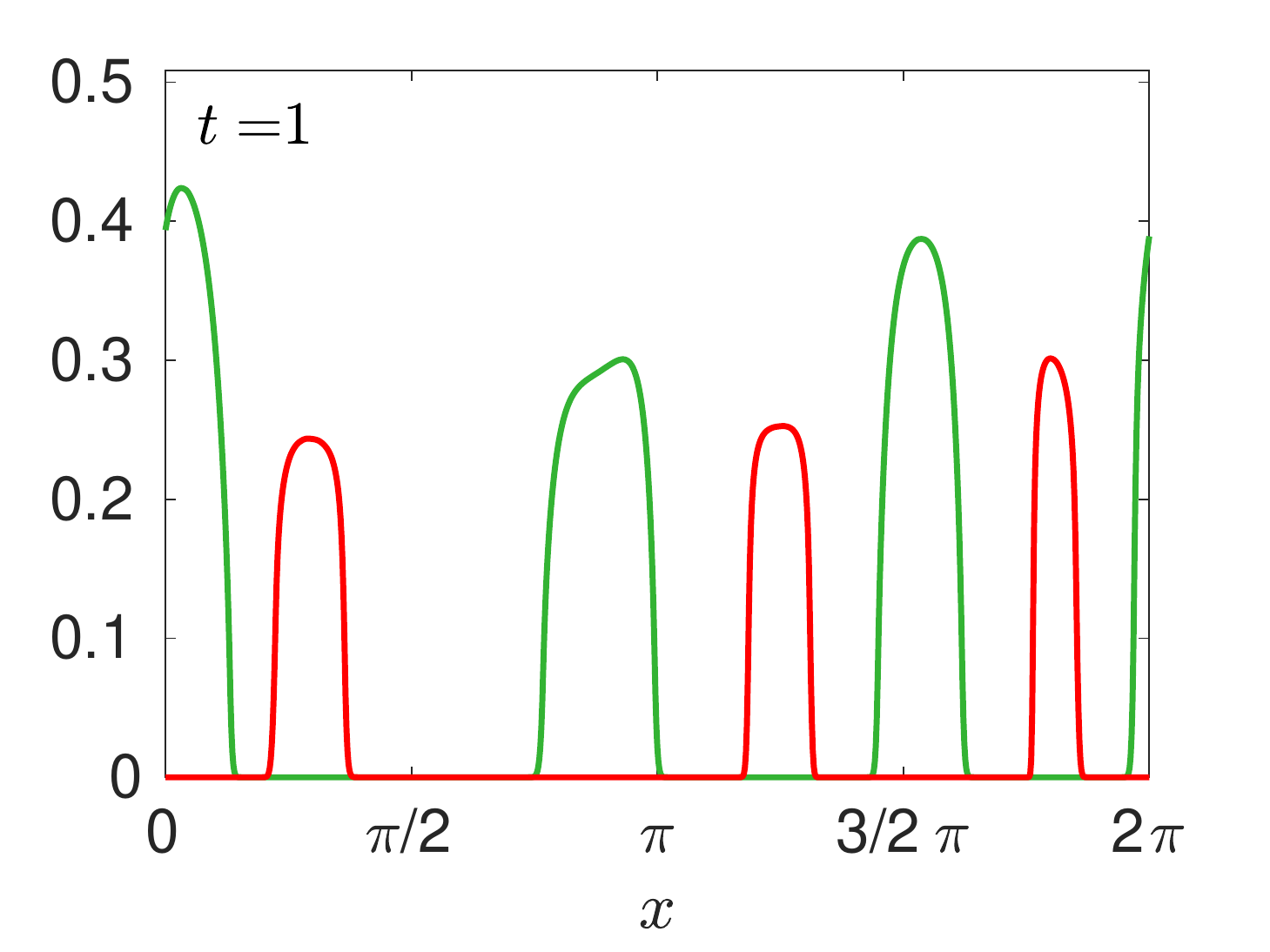}\includegraphics[width=0.33\textwidth]{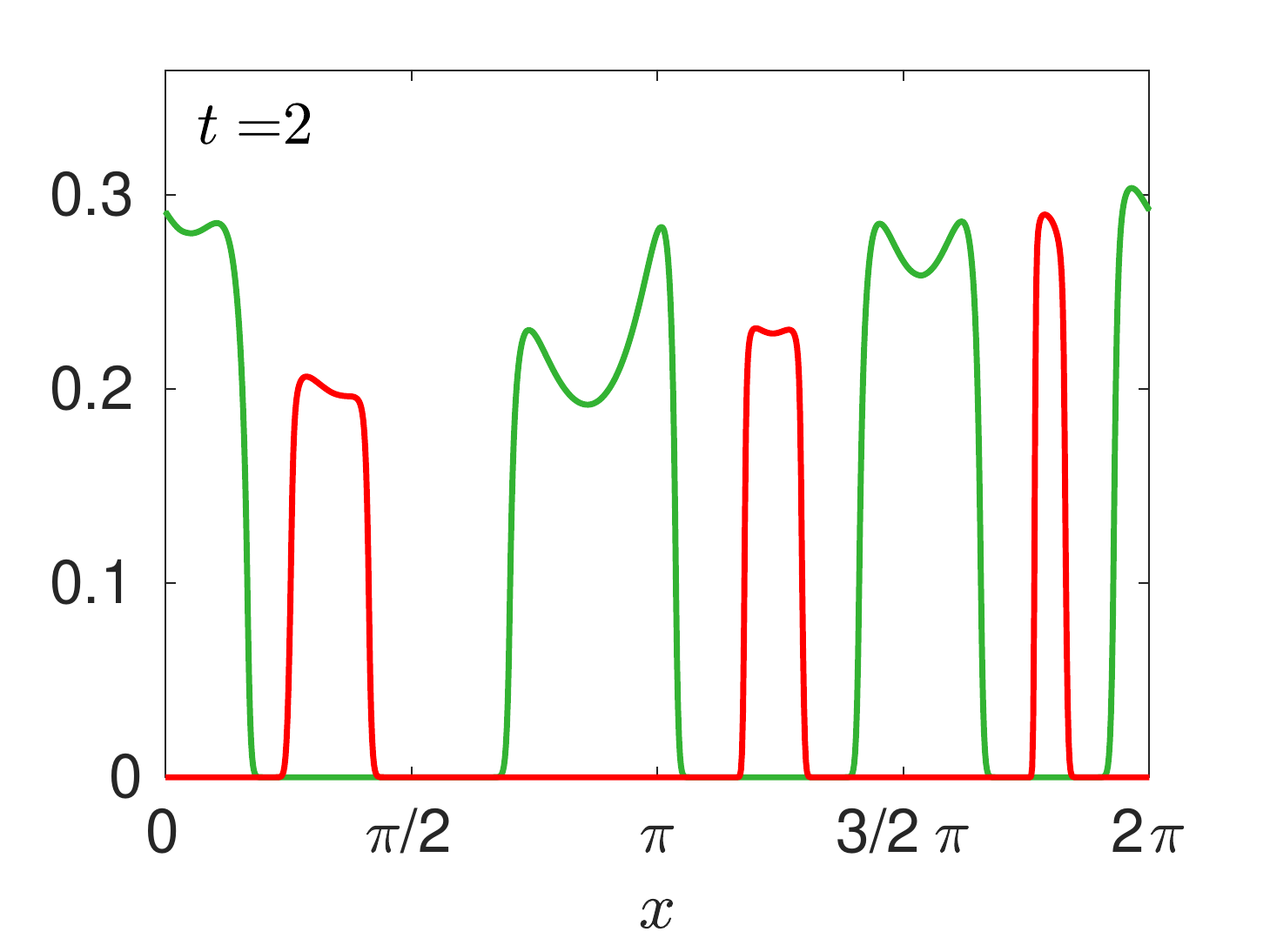}
	\includegraphics[width=0.33\textwidth]{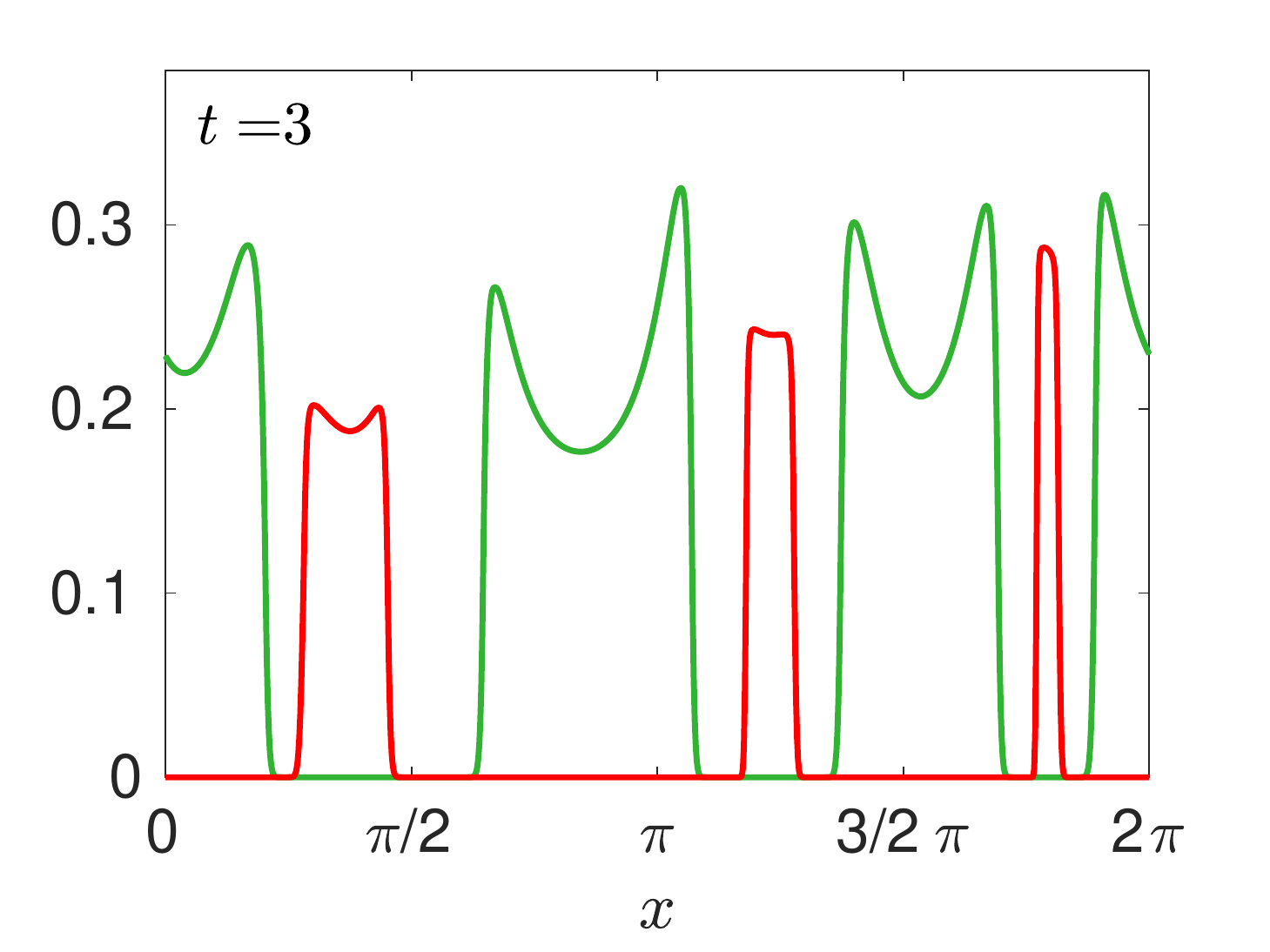}\includegraphics[width=0.33\textwidth]{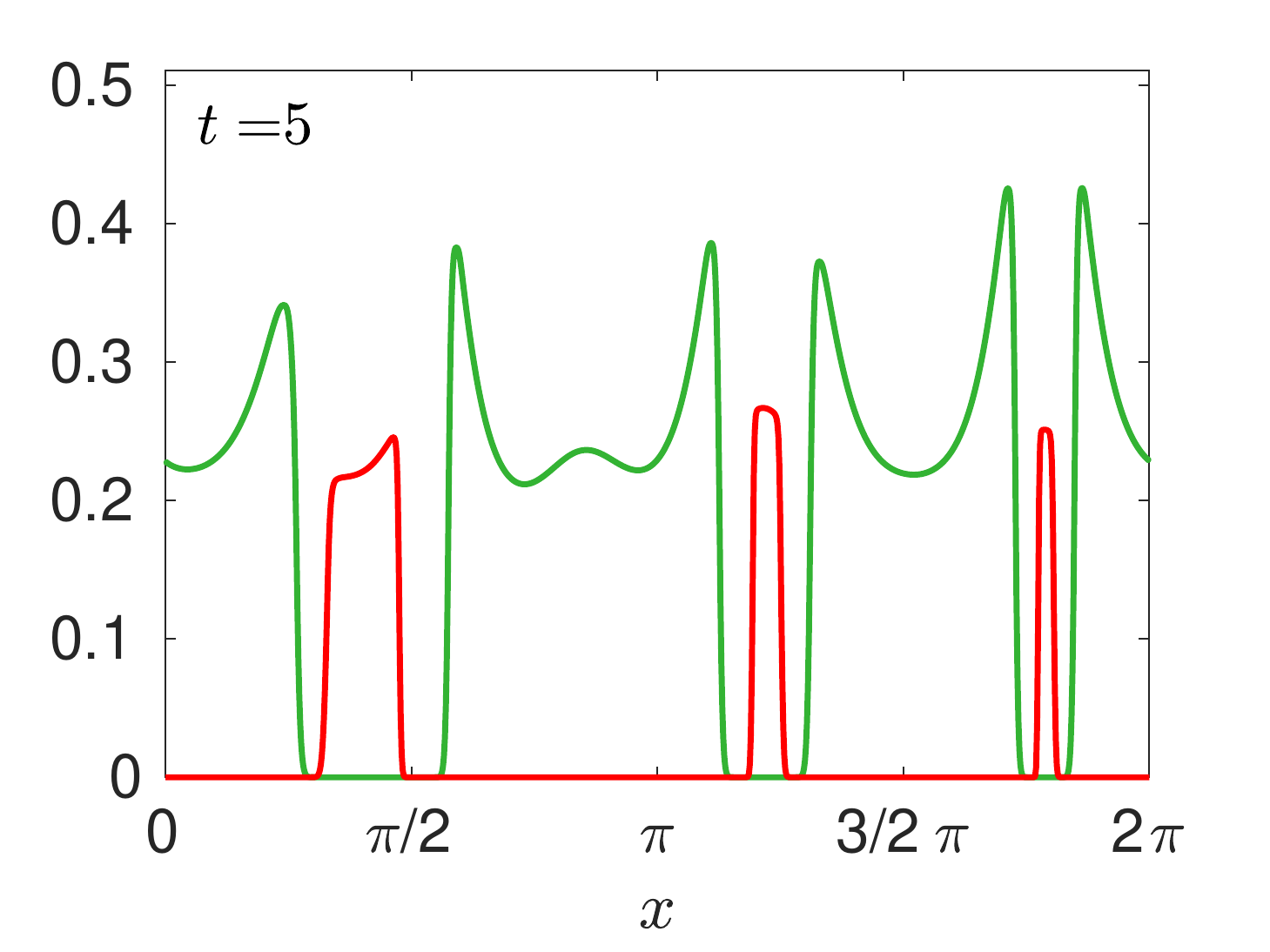}\includegraphics[width=0.33\textwidth]{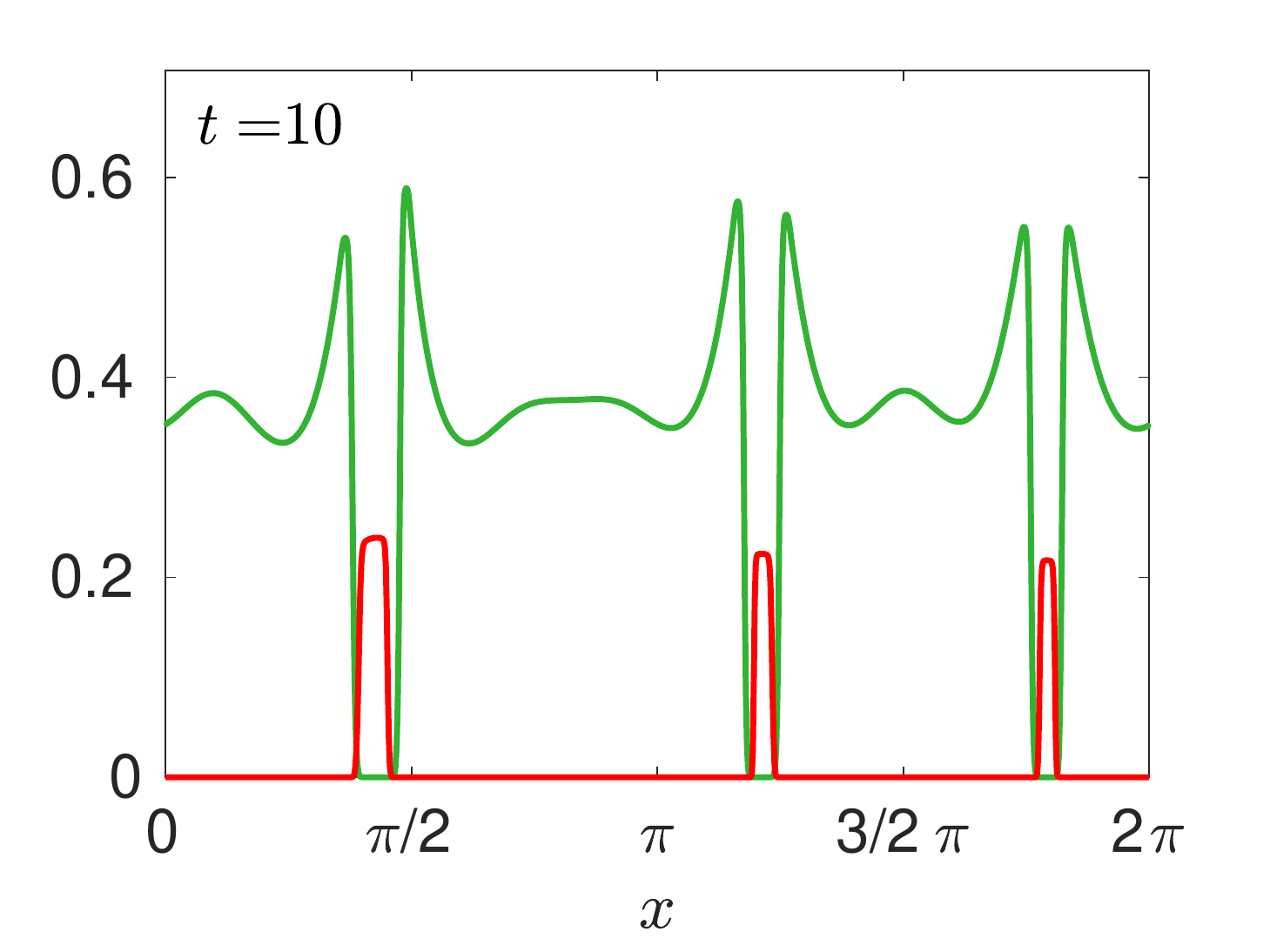}
	\includegraphics[width=0.33\textwidth]{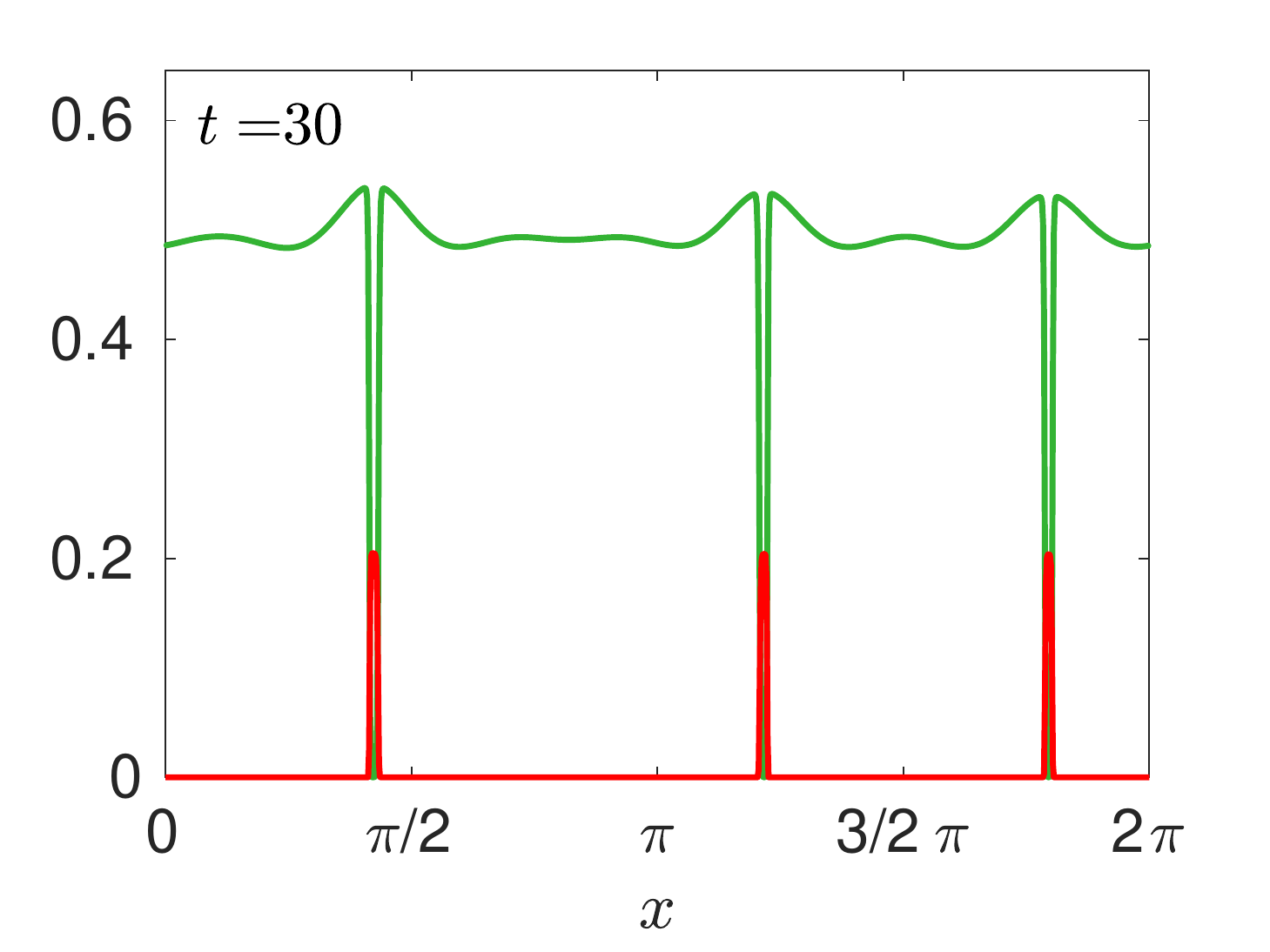}\includegraphics[width=0.33\textwidth]{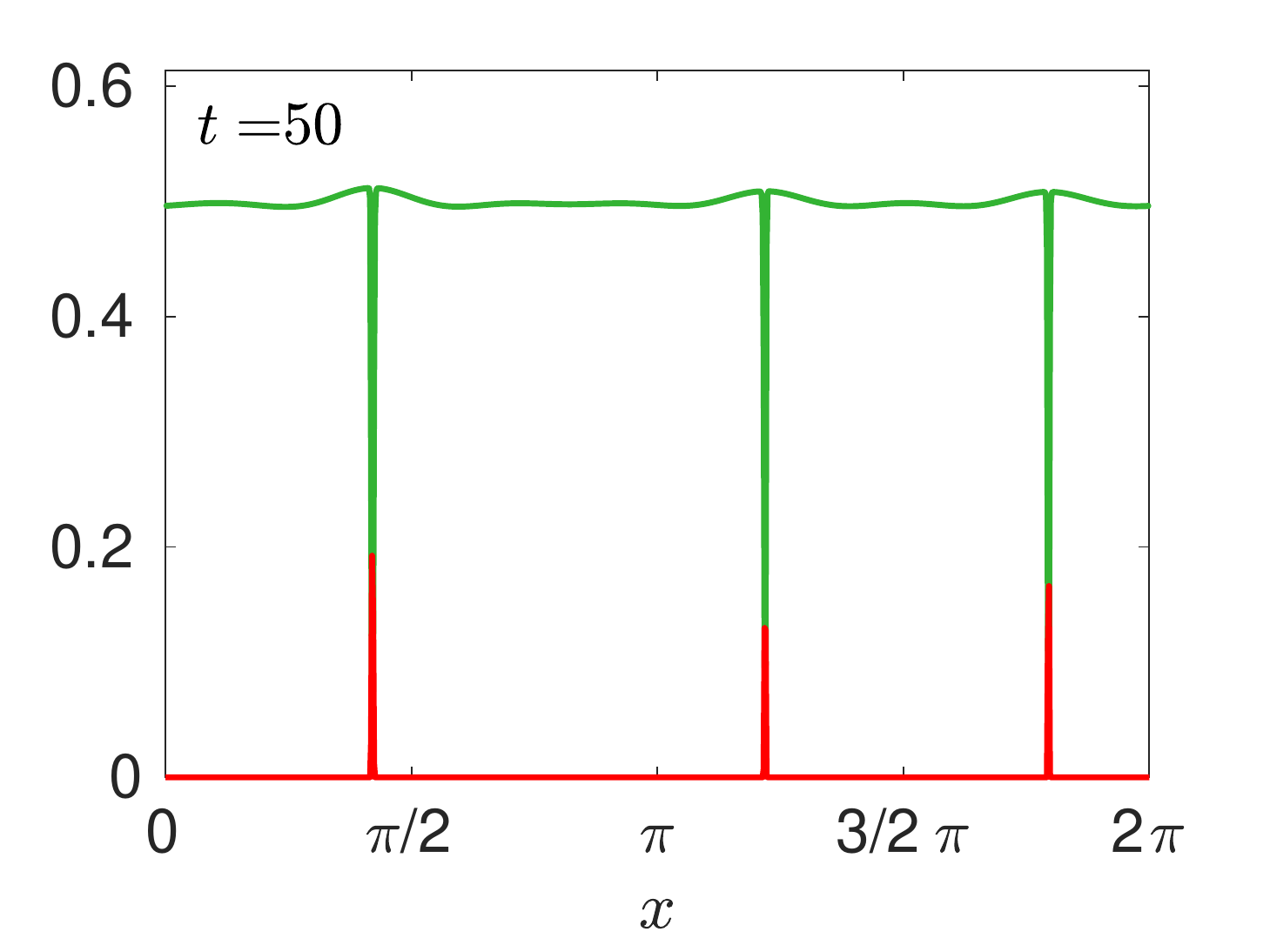}\includegraphics[width=0.33\textwidth]{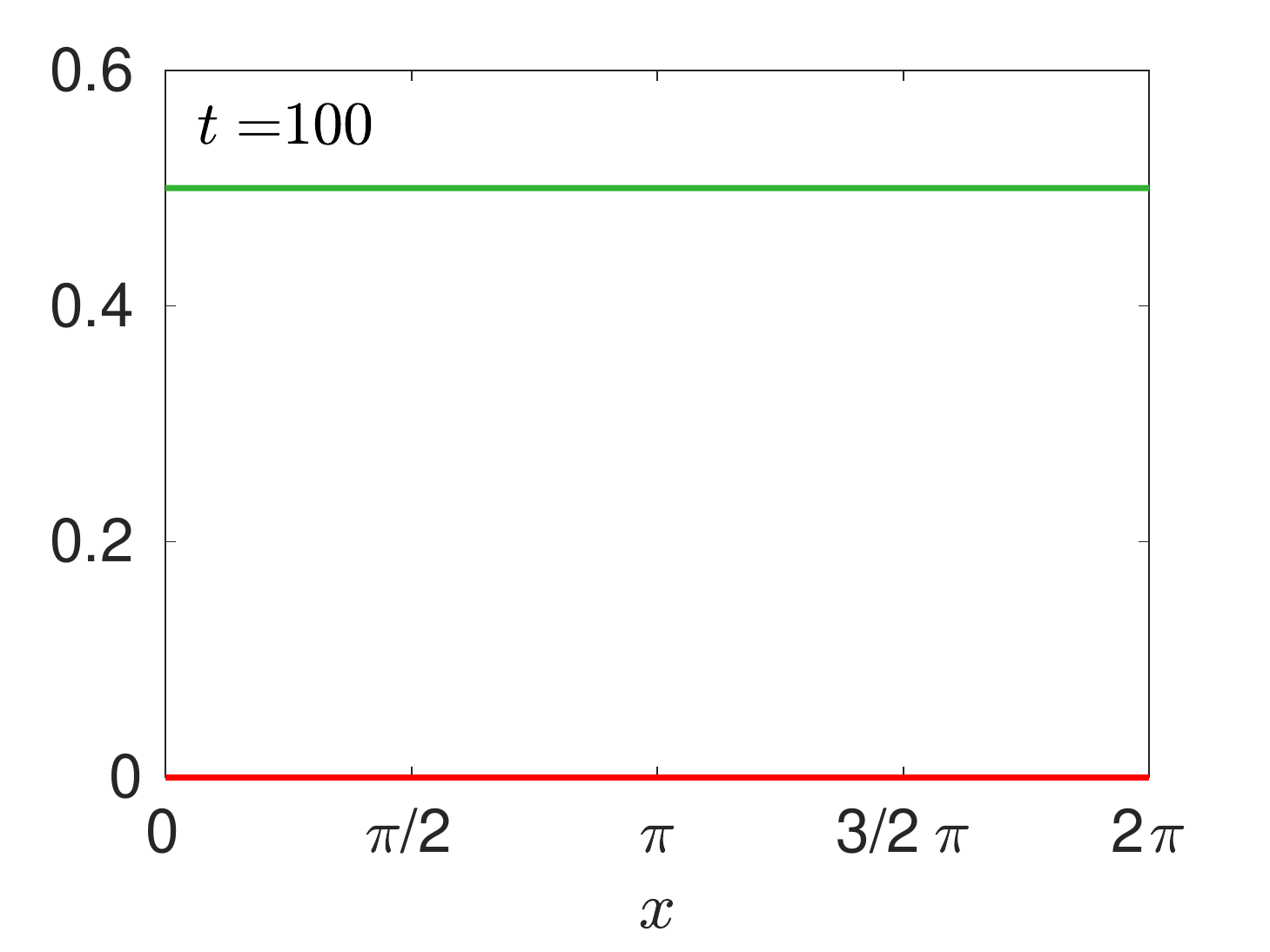}
	\end{center}
	\caption{\textit{The simulation of system \eqref{1.1} with kernel $ \rho $ as in \eqref{6.1}. The green curves represent species $ u_1 $, the red represents species $ u_2 $. We set our parameters as in \eqref{6.6}. Thus, one has $ r_1=0.5>r_2=0.2 $.  As we have $ r_1>r_2 $, we see the principle of exclusion of the two species and the populations maintain the segregation property as time evolves and after $ t=100  $ the distributions of the two species stay the same.  }}
	\label{Figure8}
\end{figure}
In the simulation of Figure \ref{Figure8}, species $ u_1 $ shows its dominance over $ u_2 $ when $ t=5 $. As for asymptotic behavior, in the last figure when $ t=100 $, we can see species $ u_1 $ crowd out species $ u_2 $ completely.
\\
\bigskip

\noindent\textbf{Acknowledgement} 
The authors would like to express their gratitude to Prigent Corentin (Ph.D student at the institute of mathematics of Bordeaux) for his suggestions in the numerical schema (see Appendix).
\section{Appendix}
For simplicity, we give the numerical scheme for the following one species and one dimensional model with periodic boundary condition
\begin{equation*}
\left\lbrace\begin{array}{rll}
\partial_t u+\partial_x\left(u{\bf v}\right)&=&uh(u),\;t>0,\;x\in\T,\\[0.25cm]
\mathbf{v}(t,x)&=&-\partial_x(K\circ u(t,\cdot))(x),\\[0.25cm]
u(0,x)&=&u_0(x)\in L_{per}^1(\T).
\end{array}\right.
\end{equation*}
The numerical method used is based on upwind scheme. We refer to 	\cite{Engquist1981, Leveque2002} for more results about this subject. We briefly illustrate our numerical scheme in this section: the approximation of the convolution term is as follows
\begin{align*}\label{}
(K\circ u(t,\cdot))(x)
=\int_{\T} u(t,y)K(x-y)dy
\approx\sum_{j}K(x-x_j)u(t,x_j)\Delta x.
\end{align*}
In addition, we define
\begin{equation*}
l_i^n:=\sum_{j}K(x_i-x_j)u(t_n,x_j)\Delta x,
\end{equation*}
for $ i=1,2,...,M,\ n=0,1,2,...,N $.
% One requires $ K $ to be of bounded variation by which we can define the Riemann-Stieltjes integral to justify the above approximation. 
We use the numerical scheme as illustrated in \cite{Hillen2006,Toro2013} to deal with the nonlocal convection term and the scheme reads as follows
\begin{align*}
	&u^{n+1}_i=u^n_i+ \frac{\Delta t}{\Delta x}\bigg(\phi(u_{i+1}^-,u_{i}^+)-\phi(u_{i}^-,u_{i-1}^+)\bigg)+\Delta t \; u_i^nh(u_i^n),\\\nonumber
	&i=1,2,...,M,\ n=0,1,2,...,N,
\end{align*}
with $ \phi(u_{i+1}^-,u_{i}^+),\; \phi(u_{i}^-,u_{i-1}^+) $ defined as
\begin{equation*}
	\begin{aligned}
		\phi(u_{i+1}^-,u_{i}^+)= v_{i+\frac{1}{2}}^n\frac{u_{i+1}^-+u_{i}^+}{2}-|v_{i+\frac{1}{2}}^n|\frac{u_{i+1}^--u_{i}^+}{2},\\
		\phi(u_{i}^-,u_{i-1}^+)= v_{i-\frac{1}{2}}^n\frac{u_{i}^-+u_{i-1}^+}{2}-|v_{i-\frac{1}{2}}^n|\frac{u_{i}^--u_{i-1}^+}{2},
	\end{aligned}
\end{equation*}
where
\begin{equation*}
	v_{i+\frac{1}{2}}^n=\dfrac{l_{i+1}^n-l_i^n}{\Delta x},\ i=0,1,2,\cdots,M,
\end{equation*}
and
\begin{equation*}
	\begin{array}{ll}
		u_{i}^-&=u_{i}^n-\dfrac{1}{2} \mathrm{minmod} (u_{i+1}^n-u_{i}^n,u_{i}^n-u_{i-1}^n),\\[0.25cm]
		u_{i}^+&=u_{i}^n+\dfrac{1}{2} \mathrm{minmod} (u_{i+1}^n-u_{i}^n,u_{i}^n-u_{i-1}^n),
	\end{array}\ i=0,1,2,\cdots,M,
\end{equation*}
where the function $ \mathrm{minmod}(a,b) $ is defined as 
\[ \mathrm{minmod}(a,b)=\begin{cases}
\mathrm{sign}(a) \min \lbrace a,b\rbrace,\; &\text{if}\;\mathrm{sign}(a)=\mathrm{sign}(b),\\
0,\;& \text{otherwise.}
\end{cases} \]
By the periodic boundary condition, one has
$ v_{\frac{1}{2}}^n=v_{M+\frac{1}{2}}^n$ and $ u_0^n=u_M^n,\;  u_1^n=u_{M+1}^n $. Thus,
\[ u_0^{\pm}=u_M^{\pm},\;u_1^{\pm}=u_{M+1}^{\pm}. \]

\end{document}